\documentclass[11pt]{amsart}

\usepackage{amssymb,amscd,amsthm, verbatim,amsmath,color,fancyhdr, mathrsfs}
\usepackage{graphicx}
\usepackage{tikz}
\usepackage{tikz-cd}
\usepackage{bm}
\usetikzlibrary{cd}
\usetikzlibrary{shapes.geometric,positioning} 
\usepackage{appendix}
\usepackage[colorlinks=true,allcolors=black]{hyperref}
\makeatletter \def\l@subsection{\@tocline{2}{0pt}{1pc}{5pc}{}} \def\l@subsection{\@tocline{2}{0pt}{2pc}{6pc}{}} \makeatother

\usepackage[letterpaper, left=2.5cm, right=2.5cm, top=2.5cm,
bottom=2.5cm,dvips]{geometry}

\newtheorem{thm}{Theorem}[subsection]
\newtheorem{prop}[thm]{Proposition}
\newtheorem{lem}[thm]{Lemma}

\newtheorem{cor}[thm]{Corollary}

\newtheorem{exmp}[thm]{Example}
\theoremstyle{definition}
\newtheorem{defn}[thm]{Definition}

\theoremstyle{remark}
\newtheorem*{remark}{Remark}

\newcommand\textcat[1]{\textnormal{\textbf{#1}}}
\newcommand\texthom[1]{\underline{\textnormal{#1}}}
\title[Representability of Derived Moduli Stacks of Solutions of PDEs]{A Bornological Perspective on the Representability of Derived Moduli Stacks of Solutions to PDEs }

\author{Rhiannon Savage}
\address{Rhiannon Savage, Department of Mathematics, University College London, London, UK, WC1H  0AY
 and the Heilbronn Institute for Mathematical Research, Bristol, UK}
\email{rhiannon.savage@ucl.ac.uk}
\urladdr{https://www.rhiannonsavage.co.uk}

\date{}

\begin{document}
\begin{abstract}
    Proving representability of derived moduli stacks of solutions to non-linear elliptic partial differential equations generally requires significant analytic machinery. In this paper, we instead show that representability naturally follows from an Artin-Lurie style representability theorem. This necessitates the development of a new model for derived differential geometry using an extension of $\mathcal{C}^\infty$-rings that we call $\mathcal{C}^\infty$-bornological rings. This new theory embeds into the theory of derived bornological geometry recently proposed by Ben-Bassat, Kelly, and Kremnizer. 
\end{abstract}
\maketitle
\vskip.5in

\tableofcontents

\section{Introduction}
\addtocontents{toc}{\setcounter{tocdepth}{-1}}

\subsection*{Motivation}

Geometric stacks are a class of stacks which are glued together from affines in a compatible way. Key examples are algebraic stacks and Deligne-Mumford stacks. A main motivation for defining geometric stacks is to endow certain moduli stacks with useful geometric structures. For example, in \cite{lurie_survey_2009}, Lurie shows that the derived moduli stack of elliptic curves is a $1$-geometric derived Deligne-Mumford stack. 

In recent years, there has been a great deal of interest in proving representability of the derived moduli stack of solutions to non-linear elliptic PDEs as a geometric stack. Given a system of non-linear elliptic partial differential equations on a compact manifold, one can express the moduli space of solutions locally as a \textit{Kuranishi chart} \cite{kuranishi_new_1965}, i.e. the zero set of a smooth map $f:\mathbb{R}^n\rightarrow{\mathbb{R}^m}$. We can consider the concept of a Kuranishi atlas on such a moduli space, see work by Fukaya and Ono \cite{fukaya_arnold_1999}. Traditionally, proving representability of the associated moduli functors as derived manifolds (or $d$-manifolds) involves gluing together these Kuranishi atlases, which is in general quite a complicated procedure.

There have recently been a number of alternative simpler approaches. In \cite{pardon_representability_2023}, Pardon considers this problem in a constructive way using non-linear elliptic Fredholm analysis. His definition of derived smooth manifolds is obtained from the usual category of smooth manifolds by formally adjoining finite limits modulo preserving finite transverse limits \cite[Definition 3.5]{pardon_representability_2023}. In \cite{steffens_representability_2024}, Steffens takes a higher categorical approach to prove representability of the moduli stack by a \textit{derived $\mathcal{C}^\infty$-scheme} which is locally of finite presentation and quasi-smooth. These objects can be considered to be glued together from Kuranishi spaces \cite[Remark 1.0.1]{steffens_representability_2024}.

In this paper, we use an alternative approach and show that representability of the derived moduli stack of solutions follows from a generalisation of the Artin-Lurie representability theorem. This theorem is stated and proved in a previous paper \cite{savage_representability_2024}, and holds for suitably defined \textit{derived bornological geometry frameworks}, in the sense of Ben-Bassat, Kelly, and Kremnizer \cite{ben-bassat_perspective_2024}. This result exhibits the power of this new theory of derived geometry and the associated higher categorical tools in solving challenging problems in different areas of geometry.  

\subsection*{Ind-Banach Spaces and Complete Bornological Spaces}

The main objects of interest to us in derived bornological geometry are Ind-Banach spaces and complete bornological spaces. We introduce these briefly here and state the key properties of them we will use in this paper. 

Suppose that $k$ is a non-trivially valued field. Then, we can take the free filtered cocompletion $\textnormal{Ind}(\textnormal{Ban}_k)$ of the category $\textnormal{Ban}_k$ of Banach spaces over $k$. We refer the reader to Appendix \ref{indobjectappendix} for more details on Ind and SInd objects. We note that the category $\textnormal{Ind}(\textnormal{Ban}_k)$ is not a concrete category. However, we can take the full subcategory $\textnormal{Ind}^m(\textnormal{Ban}_k)$ of essentially monomorphic objects which is concrete. This category is equivalent to the category $\textnormal{CBorn}_k$ of \textit{complete bornological spaces} described in Appendix \ref{bornologyappendix}.

We note that the functor $\varinjlim:\textnormal{Ind}(\textnormal{Ban}_k)\rightarrow{\textnormal{CBorn}_k}$ defines an equivalence of derived categories, see \cite[Theorem 1.3]{henrard_left_2023}), 
\begin{equation*}
    \textnormal{D}(\textnormal{CBorn}_k)\simeq\textnormal{D}(\textnormal{Ind}(\textnormal{Ban}_k))
\end{equation*}Moreover, there is an equivalence between the abelian left hearts \cite[Definition 1.2.18]{schneiders_quasi-abelian_1999}
\begin{equation*}
    \textnormal{LH}(\textnormal{CBorn}_k)\simeq\textnormal{LH}(\textnormal{Ind}(\textnormal{Ban}_k))
\end{equation*}

The categories $\textnormal{IndBan}_k$ and $\textnormal{CBorn}_k$ are bicomplete exact categories with small compact projective generating sets. Fix a Grothendieck universe $V_\aleph$ with $\aleph$ a strongly inaccessible cardinal. Consider the class of cardinals $\kappa$ such that there exists a Banach space $V$ containing a bounded unit disk of cardinality $\kappa$ whose underlying set lies in $V_\aleph$. This is a subset of $\aleph$. We consider the small subcategory $\textnormal{Lin}_k\subseteq \textnormal{Ind}(\textnormal{Ban}_k)$ whose objects are all of the form $\ell^1(\kappa):=\{(x_k)_{k\in\kappa}\mid x_k\in k,\sum_{k\in\kappa} |x_k|<\infty\}$ for $\kappa<\aleph$, and whose morphisms are bounded linear maps. Then, we note that the underlying set of $\textnormal{Lin}_k$ is a set of compact projective generators for $\textnormal{Ind}(\textnormal{Ban}_k)$ and $\textnormal{CBorn}_k$. Moreover, 

\begin{thm}\cite[c.f. Lemma A.1.1]{ben-bassat_perspective_2024}\cite[c.f. Proposition 1.3.0.3]{savage_stacks_2025} \phantomsection\label{siftedcborncompletion}
 There are equivalences of categories
\begin{equation*}
    \textnormal{LH}(\textnormal{CBorn}_k)\simeq\textnormal{LH}(\textnormal{Ind}(\textnormal{Ban}_k))\simeq \textnormal{SInd}(\textnormal{Lin}_k)
\end{equation*}  
\end{thm}

\subsection*{Derived Bornological Geometry}

In recent years, there have been numerous suggestions of suitable frameworks within which to define derived analytic geometry and derived differential geometry. Any such theory must define a collection of derived affine objects. In work of To\"en and Vezzosi, \cite{toen_homotopical_2008}, the derived affine objects are objects in the opposite category of simplicial commutative rings. In a series of papers (\cite{bambozzi_dagger_2016},\cite{bambozzi_stein_2018},\cite{bambozzi_sheafyness_2024},\cite{ben-bassat_perspective_2024},\cite{ben-bassat_non-archimedean_2017},\cite{ben-bassat_frechet_2023},\cite{kelly_analytic_2022}) spanning almost a decade, Bambozzi, Ben-Bassat, Kelly, Kremnizer, and Mukherjee propose that the appropriate derived affine objects for derived  bornological geometry are objects in the opposite $(\infty,1)$-category of \textit{simplicial commutative Ind-Banach algebras}, equivalently \textit{simplicial commutative complete bornological algebras}. These are simplicial commutative algebra objects in the categories $\textnormal{Ind}(\textnormal{Ban}_k)$ and $\textnormal{CBorn}_k$ respectively. They show that this theory of derived bornological geometry provides an appropriate setting to define derived analytic geometry.

Although the theory of derived bornological geometry is a relatively new theory, there is already a wealth of new developments and applications. In particular, there is a version of the Hochschild-Kostant-Rosenberg Theorem \cite{kelly_analytic_2022}, there are six functor formalisms for rigid analytic sheaves \cite{soor_six-functor_2024}, there are derived blow ups \cite{ben-bassat_blow-ups_2023}, and there are developments in the theory of analytic geometry over $\mathbb{F}_1$ \cite{bambozzi_analytic_2019}.

In a recent paper \cite[Theorem 5.6.1]{savage_representability_2024}, we showed that we can obtain a representability theorem which holds in this new model of derived bornological geometry.  Suppose that we have some subcategory $\mathcal{A}$ of derived bornological affines, endowed with a suitable topology $\bm\tau|_\mathcal{A}$ and a class $\textcat{P}$ of distinguished maps, and that several conditions are satisfied such that we have a \textit{representability context} \cite[Definition 5.3.1]{savage_representability_2024}. The aim of the following representability theorem is to give conditions under which derived stacks on the site $(\mathcal{A},\bm\tau|_{\mathcal{A}})$ are $n$-geometric in some sense relative to $\mathcal{A}$, i.e. they are built up from affines in $\mathcal{A}$ in a compatible way. 

\begin{thm}[Representability Theorem]\cite[Theorem 5.6.1]{savage_representability_2024} The following conditions are equivalent for a stack $\mathcal{F}\in\textcat{Stk}(\mathcal{A},\bm\tau|_\mathcal{A})$.
\begin{enumerate}
    \item $\mathcal{F}$ is an $n$-geometric stack,
    \item $\mathcal{F}$ satisfies the following three conditions:
    \begin{enumerate}
\item\makeatletter\def\@currentlabel{a}\makeatother The truncation $t_0(\mathcal{F})$ is an $n$-geometric stack,
        \item \makeatletter\def\@currentlabel{b}\makeatother $\mathcal{F}$ has an obstruction theory relative to $\mathcal{A}$,
        \item \makeatletter\def\@currentlabel{c}\makeatother $\mathcal{F}$ is nilcomplete with respect to $\mathcal{A}$. 
    \end{enumerate}
\end{enumerate}
\end{thm}

In Corollary \ref{mappingstacksimplified}, we prove the following simplification of this theorem for the case of mapping stacks. 

\begin{cor}\label{mapstackrepresentable}
    Suppose that $X=\textnormal{Spec}(A)\in\mathcal{A}^\heartsuit$ is such that $X\rightarrow{*}$ is flat and in $\textcat{P}$. Suppose that $\mathcal{G}$ is in $\textcat{Stk}(\mathcal{A},\bm\tau|_\mathcal{A})_{/X}$. Under the following conditions, $\texthom{Map}_{\textcat{Stk}(\mathcal{A},\bm\tau|_\mathcal{A})_{/X}}(X,\mathcal{G})$ is an $n$-geometric stack,
    \begin{enumerate}
        \item The truncation $t_0(\texthom{Map}_{\textcat{Stk}(\mathcal{A},\bm\tau|_\mathcal{A})_{/X}}(X,\mathcal{G}))$ is an $n$-geometric stack, 
        \item $\mathcal{G}$ is $n$-geometric and the cotangent complex of the morphism $\mathcal{G}\rightarrow{X}$ is perfect.
    \end{enumerate}

    Moreover, if $\mathbb{L}_X$ is perfect, then $\mathbb{L}_{\texthom{Map}_{\textcat{Stk}(\mathcal{A},\bm\tau|_{\mathcal{A}})_{/X}}(X,\mathcal{G})}$ is also perfect. 
\end{cor}

\subsection*{$\mathcal{C}^\infty$-Bornological Rings}

In order to apply the representability theorem to prove representability of the derived moduli stack of solutions to non-linear elliptic PDEs, we need to develop new foundations for derived differential geometry as derived bornological geometry. In differential geometry, one works with manifolds, which can be realised as objects in the opposite category of $\mathcal{C}^\infty$-rings. A $\mathcal{C}^\infty$-ring can be described as an algebra $A$ over the reals such that the multiplication $\mathbb{R}\times\mathbb{R}\rightarrow{\mathbb{R}}$ determines the multiplication on $A$ and such that any smooth map $\mathbb{R}^n\rightarrow{\mathbb{R}^m}$ lifts to a map $A^n\rightarrow{A^m}$. Equivalently, $\mathcal{C}^\infty$-rings can be described as objects in the following category. 
\begin{defn}
    Let $\textnormal{CartSp}$ denote the category whose objects are manifolds of the form $\mathbb{R}^n$ for some $n\in\mathbb{N}$ and whose morphisms are smooth maps. Then, the \textit{category $\mathcal{C}^\infty\textnormal{Ring}$ of $\mathcal{C}^\infty$-rings} is defined to be
    \begin{equation*}
        \mathcal{C}^\infty\textnormal{Ring}:=\textnormal{SInd}(\textnormal{CartSp}^{op})=\textnormal{Fun}^\times(\textnormal{CartSp},\textnormal{Set})
    \end{equation*}
\end{defn}

\begin{remark}
    Alternatively, these can be described as algebras over a Lawvere theory generated by $\mathbb{R}$. 
\end{remark}

By considering a natural notion of a \textit{simplicial $\mathcal{C}^\infty$-ring} one can obtain notions of derived manifolds \cite{spivak_derived_2010}, \cite{carchedi_universal_2019}, and derived $\mathcal{C}^\infty$-affine schemes \cite{steffens_representability_2024}, the category of which we will denote by $\mathcal{C}^\infty\textcat{DAff}$. By \cite[Lemma 5.3.91]{ben-bassat_perspective_2024}, these derived objects embed naturally into the $(\infty,1)$-category of simplicial commutative complete bornological algebras. 

As described in Section \ref{defnasmappingstack}, the derived moduli stack of solutions to non-linear elliptic PDEs can be described as a derived mapping stack of sections. A key ingredient in our representability theorem is requiring that the truncated underived mapping stack needs to be representable. In general, this problem is almost as hard as proving the derived mapping stack is representable. A clear method of attack is to use some abstract categorical properties to show that the stack is representable. However, one quickly runs into problems in working with the category of $\mathcal{C}^\infty$-rings since it is not closed symmetric monoidal. Therefore, we have no well defined notion of a dual object and so we cannot turn problems involving mapping spaces into problems involving tensor products. A solution to this problem would be to find some adjunction 
\begin{equation*}
    L:\textnormal{C}\leftrightarrows{\mathcal{C}^\infty\textnormal{Ring}:R}
\end{equation*}where $\textnormal{C}$ is some suitable closed symmetric monoidal category and the right adjoint is strong monoidal.

However, it seems that no such category $\mathcal{C}$ should exist. Morally, the adjunction should be forgetting the smooth structure in some sense and taking any $\mathcal{C}^\infty$-ring to its underlying space. We note that $\mathcal{C}^\infty(\mathbb{R})\hat{\otimes}\mathcal{C}^\infty(\mathbb{R})\simeq\mathcal{C}^\infty(\mathbb{R}^2)$ as commutative Fr\'echet algebras \cite[Theorem 51.6]{treves_topological_1967}, where here we use the complete projective tensor product. Therefore, we would need this monoidal structure on $\mathcal{C}^\infty$-rings to transfer to one in $\textnormal{C}$ under $R$. This naturally leads us to consider categories $\textnormal{C}$ endowed with the complete projective tensor product, such as Fr\'echet spaces, complete bornological spaces etc. The category of $\mathcal{C}^\infty$-rings, which only sees the smooth structure on finite dimensional manifolds, does not capture the behaviour of smooth functions on such spaces. 

We recall from Theorem \ref{siftedcborncompletion}, that the category $\textnormal{Lin}_\mathbb{R}$, whose objects are of the form $\ell^1(\kappa)$ for $\kappa<\aleph$ and whose morphisms are bounded linear maps, provides a set of compact projective generators for $\textnormal{CBorn}_\mathbb{R}$.  By expanding the category $\textnormal{CartSp}$ to include all the manifolds of the form $\ell^1(\kappa)$ for $\kappa<\aleph$ along with suitably defined smooth maps between them, we can define a suitable extension of the category of $\mathcal{C}^\infty$-rings. We call these objects \textit{$\mathcal{C}^\infty$-bornological rings}. 

\begin{defn}
    Let $\mathcal{C}_{bb}^\infty$ denote the category whose objects are the same objects as $\textnormal{Lin}_\mathbb{R}$ but whose morphisms are smooth and bounded on bounded subsets. The category of \textit{$\mathcal{C}^\infty$-bornological rings}, denoted $\mathcal{C}^\infty\textnormal{BornRing}$, is the category
    \begin{equation*}
        \mathcal{C}^\infty\textnormal{BornRing}:=\textnormal{SInd}(\mathcal{C}^{\infty,op}_{bb})=\textnormal{Fun}^\times(\mathcal{C}^{\infty}_{bb},\textnormal{Set})
    \end{equation*}
\end{defn}

In Section \ref{Cinfinitybornsection}, we prove several important properties of $\mathcal{C}^\infty$-bornological rings. In particular, we have the following desired adjunction.
\begin{thm}[Corollary \ref{importantadjunction}] There is an adjunction
\begin{equation*}
    L:\textnormal{LH}(\textnormal{CBorn}_\mathbb{R})\leftrightarrows{\mathcal{C}^\infty\textnormal{BornRing}}:R
\end{equation*}with $R$ monoidal. 
\end{thm}

We can also define the following derived version. The $(\infty,1)$-category of \textit{derived $\mathcal{C}^\infty$-bornological rings}, denoted $\mathcal{C}^\infty\textcat{DBornRing}$, is the $(\infty,1)$-category 
    \begin{equation*}
       \mathcal{C}^\infty\textcat{DBornRing}:=\mathcal{P}_\Sigma(\mathcal{C}_{bb}^{\infty,op})=\textcat{Fun}^\times(\mathcal{C}_{bb}^\infty,\infty\textcat{Grpd})
    \end{equation*}
The following theorem shows that these derived $\mathcal{C}^\infty$-bornological rings correspond to derived affines in a derived bornological geometry setting. 
\begin{thm}[Theorem \ref{fullyfaithfulembed}] The category of derived $\mathcal{C}^\infty$-bornological rings embeds as a full subcategory of the $(\infty,1)$-category of simplicial commutative complete bornological algebras. 
\end{thm}

In future work, we will study these $\mathcal{C}^\infty$-rings in more detail. In particular, we would like to define a suitable category $\textcat{DBMfd}$ of infinite dimensional manifolds as a subcategory of derived $\mathcal{C}^\infty$-bornological affines which are finitely presented in some sense, and then show that there is an equivalence of $(\infty,1)$-categories between the category $\mathcal{C}^\infty\textcat{DBornRing}$ and the category of finite limit preserving functors from $\textcat{DBMfd}$ to $\infty\textcat{Grpd}$. 

\subsection{Representability of the Derived Moduli Stack of Solutions to Non-Linear Elliptic PDEs}

In Section \ref{section3}, we show that if we define the $(\infty,1)$-category $\textcat{DBAff}$ to be the opposite category of derived $\mathcal{C}^\infty$-bornological rings, then we can define a suitable derived bornological geometry context for derived smooth geometry. Here we have a topology $\bm\tau_{\mathcal{C}^\infty}$, the $\mathcal{C}^\infty$-localisation topology, and we have a class of open immersions maps $\textcat{open}_{\mathcal{C}^\infty}$. This geometry context is a representability context, as shown in Corollary \ref{tuplediffgeometryrep} and, hence, we can apply the Representability Theorem. Suppose that we have some classical derived $\mathcal{C}^\infty$-affine $X\in\mathcal{C}^\infty\textcat{DAff}$. Suppose that $Y\in\mathcal{C}^\infty\textcat{DAff}$ is a system of derived non-linear elliptic partial differential equations, in the sense of Definition \ref{pdedefn}.

\begin{defn}[Definition \ref{solnstackdefn}]
    The \textit{derived moduli stack of solutions to $Y$}, denoted $\textcat{Sol}_X(Y)$, is the mapping stack of sections
    \begin{equation*}
        \textcat{Sol}_X(Y):=\texthom{Map}_{\textcat{Stk}(\mathcal{C}^\infty\textcat{DAff},\bm\tau_{\mathcal{C}^\infty}|_{\mathcal{C}^\infty\textcat{DAff}})_{/X_{dR}}}(X_{dR},Y)
    \end{equation*}
\end{defn}

Using Corollary \ref{mapstackrepresentable}, we obtain the following main result.

\begin{thm}[c.f. Corollary \ref{repsolutionstack}]
    Suppose that $Y=\textnormal{Spec}(B)$ and $X=\textnormal{Spec}(A)$, with $A$ and $B$ finitely presented regular $\mathcal{C}^\infty$-rings. Then, $\textcat{Sol}_X(Y)$ is representable by an object in $\mathcal{C}^\infty\textcat{DBAff}$. 
\end{thm}

\subsection*{Acknowledgements}

I am deeply grateful to my PhD supervisor Kobi Kremnizer for suggesting this project and helping guide it in its early stages. It has also benefited from conversations with David Ben-Zvi, Jack Kelly, Devarshi Mukherjee, Arun Soor, Pelle Steffens, amongst others.  

\subsection*{Funding}

This work was supported by the Additional Funding Programme for Mathematical Sciences, delivered by EPSRC (EP/V521917/1) and the Heilbronn Institute for Mathematical Research. This research was partially conducted during the author's PhD at the University of Oxford, funded by an EPSRC studentship [EP/W523781/1 - no. 2580843].

\addtocontents{toc}{\setcounter{tocdepth}{2}}
\section{A Representability Theorem for Stacks in Derived Geometry Contexts}\label{representabilitytheoremsection}

In \cite{savage_representability_2024}, we showed that we can obtain a representability theorem for derived stacks which holds in any suitably defined \textit{representability context}. These should be suitably defined \textit{derived geometry contexts} which satisfy several extra properties describing the compatibility between the classical and the derived geometry. This paper is intended to be a follow up paper to \cite{savage_representability_2024}. For ease of exposition, we will not redefine everything in full but we will frequently refer back to \cite{savage_representability_2024} for necessary notation, definitions, and theorems. 

\subsection{Derived Geometry Contexts} The formalism of these derived geometry contexts begins with the definition of Raksit \cite{raksit_hochschild_2020} of a \textit{derived algebraic context} $(\mathcal{C},\mathcal{C}_{\geq 0},\mathcal{C}_{\leq 0},\mathcal{C}^0)$, which consists of a stable $(\infty,1)$-category $\mathcal{C}$ equipped with a $t$-structure and a generating set $\mathcal{C}^0$ of compact projectives. In such a context, one can define a category $\textcat{DAlg}^{cn}(\mathcal{C})$ of derived (connective) algebra objects as algebra objects relative to the $\textcat{LSym}$-monad. We can then define the category $\textcat{DAff}^{cn}(\mathcal{C}):=\textcat{DAlg}^{cn}(\mathcal{C})^{op}$ of \textit{derived affine objects}. 

\begin{exmp}By \cite[c.f. Theorems 3.1.41, 3.1.42]{ben-bassat_perspective_2024}, there is a derived algebraic context
 \begin{equation*}
     (\textcat{Ch}(\textnormal{Ind}(\textnormal{Ban}_k)),\textcat{Ch}_{\geq 0}(\textnormal{Ind}(\textnormal{Ban}_k)),\textcat{Ch}_{\leq 0}(\textnormal{Ind}(\textnormal{Ban}_k)),\textcat{L}^H(\textnormal{Lin}_k))
 \end{equation*}where $\textcat{Ch}(-)$ denotes the $(\infty,1)$-category associated with the category of chain complexes equipped with the projective model structure, and $\textcat{L}^H(-)$ denotes the $(\infty,1)$-category associated with a category with weak equivalences. Moreover, we can transfer the $t$-structure to a $t$-structure on $\textcat{Ch}(\textnormal{CBorn}_k)$ and obtain an equivalent derived algebraic context 
 \begin{equation*}
     (\textcat{Ch}(\textnormal{CBorn}_k),\textcat{Ch}_{\geq 0}(\textnormal{CBorn}_k),\widetilde{\textcat{Ch}}_{\leq 0}(\textnormal{CBorn}_k),\textcat{L}^H(\textnormal{Lin}_k))
 \end{equation*}We note that the transferred non-positive part $\widetilde{\textcat{Ch}}_{\leq 0}(\textnormal{CBorn}_k)$ is not ${\textcat{Ch}}_{\leq 0}(\textnormal{CBorn}_k)$. From now on, we will shorten notation for these derived algebraic contexts to $\underline{\textcat{Ch}}(\textnormal{Ind}(\textnormal{Ban}_k))$ and $\underline{\textcat{Ch}}(\textnormal{CBorn}_k)$. In these situations, we have that 
 \begin{equation*}
     \textcat{DAlg}^{cn}(\textcat{Ch}(\textnormal{E}))\simeq \textcat{L}^H(\textnormal{Comm}(\textnormal{sInd}(\textnormal{Ban}_k)))\simeq \textcat{L}^H(\textnormal{Comm}(\textnormal{sCBorn}_k))
 \end{equation*}where $\textnormal{Comm}(\textnormal{sInd}(\textnormal{Ban}_k))$ and $\textnormal{Comm}(\textnormal{sCBorn}_k)$ denote the categories of simplicial commutative Ind-Banach algebras and simplicial commutative bornological algebras respectively. 
\end{exmp}

As described in \cite[Definition 3.6.2]{savage_representability_2024}, a \textit{derived geometry context} is a tuple \begin{equation*}(\mathcal{C},\mathcal{C}_{\geq 0},\mathcal{C}_{\leq 0},\mathcal{C}^0,\bm{\tau},\textcat{P},\mathcal{A},\textcat{M})
\end{equation*}which consists of a derived algebraic context $(\mathcal{C},\mathcal{C}_{\geq 0},\mathcal{C}_{\leq 0},\mathcal{C}^0)$ along with a topology $\bm\tau$ on $\textcat{DAff}^{cn}(\mathcal{C})$, a collection of maps $\textcat{P}$, a subcategory $\mathcal{A}$ of `distinguished affines', and a compatible collection $\textcat{M}$ of $A$-modules for $\textnormal{Spec}(A)\in\mathcal{A}$. We require this tuple to satisfy several conditions detailed in \cite{savage_representability_2024}. 

These derived geometry contexts aim to describe the necessary data for defining useful geometric objects and concepts. By embedding $\mathcal{A}$ in the larger category $\textcat{DAff}^{cn}(\mathcal{C})$, we can use the stronger and more versatile properties of $\textcat{DAff}^{cn}(\mathcal{C})$ to study $\mathcal{A}$. As described in \cite{savage_representability_2024}, within such a context we can define geometric stacks, square-zero extensions, cotangent complexes, and obstruction theories. 

In many of our examples of interest, topologies will consist of maps which are \textit{homotopy monomorphisms}, in the following sense. 

\begin{defn}\phantomsection\label{homotopyepi}
    Suppose that $f:A\rightarrow{B}$ is a morphism in $\textcat{DAlg}^{cn}(\mathcal{C})$. Then, $f$ is a \textit{homotopy epimorphism} if the map $B\otimes_A^\mathbb{L}B\rightarrow{B}$ is an equivalence. The corresponding morphism in $\textcat{DAff}^{cn}(\mathcal{C})$ is called a \textit{homotopy monomorphism}. 
\end{defn}

\begin{exmp}
    In \cite[Corollary 6.8.6]{savage_representability_2024}, we show that there is a derived geometry context modelling derived complex analytic geometry
    \begin{equation*}
        (\underline{\textcat{Ch}}(\textnormal{Ind}(\textnormal{Ban}_\mathbb{C})),\bm{hm}^{fin}_{\textcat{DSt}},{\textcat{fP}^{\bm{hm}}_{\textcat{DSt}}},\textcat{DSt}^{op},\textcat{Mod}^{coad,cn})
    \end{equation*}where we consider some category $\textcat{DSt}$ of derived Stein algebras along with a collection of coadmissible modules $\textcat{Mod}^{coad,cn}$. The topology has covers which are homotopy monomorphisms and the class of maps ${\textcat{fP}^{\bm{hm}}_{\textcat{DSt}}}$ are formally perfect. 
\end{exmp}

Suppose that $(\mathcal{C},\mathcal{C}_{\geq 0},\mathcal{C}_{\leq 0},\mathcal{C}^0,\bm{\tau},\textcat{P},\mathcal{A},\textcat{M})$ is a derived geometry context. For clarity, we will recall the definition of an $n$-geometric stack from \cite{savage_representability_2024}. For the purposes of this paper we only consider $n$-geometricity relative to $\mathcal{A}$. 
\begin{defn}
\begin{enumerate}
    \item A stack $\mathcal{F}$ in $\textcat{Stk}(\mathcal{A},\bm{\tau}|_{\mathcal{A}})$ is \textit{$(-1)$-geometric relative to $\mathcal{A}$} (or \textit{a representable stack}) if it is of the form $\mathcal{F}\simeq \textnormal{Map}_{\mathcal{A}}(-,X)$ for some $X\in\mathcal{A}$,
    \item A morphism of stacks $f:\mathcal{F}\rightarrow{\mathcal{G}}$ in $\textcat{Stk}(\mathcal{A},\bm{\tau}|_{\mathcal{A}})$ is \textit{$(-1)$-representable} if, for any map $X\rightarrow{\mathcal{G}}$, with $X\in\mathcal{A}$ a representable stack, the pullback $\mathcal{F}\times_\mathcal{G} X$ is $(-1)$-geometric,
    \item  A morphism of stacks $f:\mathcal{F}\rightarrow{\mathcal{G}}$ in $\textcat{Stk}(\mathcal{A},\bm{\tau}|_{\mathcal{A}})$ is \textit{in} $(-1)\textcat{-P}|_{\mathcal{A}}$ if it is $(-1)$-representable and, for any map $X\rightarrow{\mathcal{G}}$, with $X\in\mathcal{A}$ a representable stack, the induced map of $(-1)$-geometric stacks $\mathcal{F}\times_\mathcal{G}X\rightarrow{X}$ is represented by a morphism in $\textcat{P}|_\mathcal{A}$.
\end{enumerate}
\end{defn}

Now, for $n\geq 0$, we can inductively build up notions of higher geometric stacks by glueing together representables as follows. 

\begin{defn}\phantomsection\label{geometricstackdefinition}

\begin{enumerate}
    \item Let $\mathcal{F}$ be a stack in $\textcat{Stk}(\mathcal{A},\bm{\tau}|_{\mathcal{A}})$. An \textit{$n$-atlas} for $\mathcal{F}$ is a set of morphisms $\{U_i\rightarrow{\mathcal{F}}\}_{i\in I}$ such that each $U_i$ is $(-1)$-geometric, each map $U_i\rightarrow{\mathcal{F}}$ is in $(n-1)\textcat{-P}|_\mathcal{A}$, and there is an epimorphism of stacks 
        \begin{equation*}
            \coprod_{i\in I} U_i\rightarrow{\mathcal{F}}
        \end{equation*}in $\textcat{Stk}(\mathcal{A},\bm{\tau}|_{\mathcal{A}})$, 
    \item A stack $\mathcal{F}$ in $\textcat{Stk}(\mathcal{A},\bm{\tau}|_{\mathcal{A}})$ is \textit{$n$-geometric} if the diagonal morphism $\mathcal{F}\rightarrow{\mathcal{F}\times\mathcal{F}}$ is $(n-1)$-representable and $\mathcal{F}$ admits an $n$-atlas,
    \item A morphism of stacks $f:\mathcal{F}\rightarrow{\mathcal{G}}$ in $\textcat{Stk}(\mathcal{A},\bm{\tau}|_{\mathcal{A}})$ is \textit{$n$-representable} if, for any map $X\rightarrow{\mathcal{G}}$ with $X\in\mathcal{A}$ a representable stack, the pullback $\mathcal{F}\times_\mathcal{G} X$ is $n$-geometric, 
    \item A morphism of stacks $f:\mathcal{F}\rightarrow{\mathcal{G}}$ in $\textcat{Stk}(\mathcal{A},\bm{\tau}|_{\mathcal{A}})$ is in $n\textcat{-P}|_\mathcal{A}$ if it is $n$-representable and, for any map $X\rightarrow{\mathcal{G}}$ with $X\in\mathcal{A}$ a representable stack, there exists an $n$-atlas of the form $\{U_i\rightarrow \mathcal{F}\times_\mathcal{G} X\}_{i\in I}$ such that each map $U_i\rightarrow{X}$ is in $\textcat{P}|_\mathcal{A}$.
\end{enumerate}
\end{defn}

We denote the subcategory of $\textcat{Stk}(\mathcal{A},\bm\tau|_\mathcal{A})$ consisting of $n$-geometric stacks by $\textcat{Stk}_n(\mathcal{A},\bm\tau|_\mathcal{A},\textcat{P}|_{\mathcal{A}})$.

\subsection{The Representability Theorem}

We would like to give conditions under which a stack $\mathcal{F}$ in $\textcat{Stk}(\mathcal{A},\bm{\tau}|_\mathcal{A})$ is $n$-geometric. Using the $t$-structure on $\mathcal{C}$ we can define $\textcat{DAlg}^\heartsuit(\mathcal{C})=\textcat{DAlg}(\mathcal{C})\times_\mathcal{C}\mathcal{C}^\heartsuit$. There is an induced adjunction on derived affines
\begin{equation*}
    \iota:\textcat{DAff}^{\heartsuit}(\mathcal{C})\leftrightarrows\textcat{DAff}^{cn}(\mathcal{C}):t_0
\end{equation*}

\begin{defn}
    Define the full subcategory $\mathcal{A}^\heartsuit\subseteq\mathcal{A}\subseteq\textcat{DAff}^\heartsuit(\mathcal{C})$ to consist of objects $X$ in $\textcat{DAff}^\heartsuit(\mathcal{C})$ such that $X=t_0(Y)$ for some $Y\in \mathcal{A}$. 
\end{defn}

\begin{defn}
    We define a class $\textcat{P}^\heartsuit$ of maps in $\mathcal{A}^\heartsuit$ and a collection $\bm\tau^\heartsuit$ of covering families in $\textnormal{Ho}(\mathcal{A}^{\heartsuit})$ as follows
    \begin{enumerate}
        \item The collection $\textcat{P}^\heartsuit$ is defined to be the collection of morphisms $f$ such that $f\in\textcat{P}\cap \mathcal{A}^\heartsuit$,
        \item The collection of $\bm\tau^\heartsuit$-covers is the collection $\{t_0(U_i)\rightarrow{t_0(X)}\}_{i\in I}$ such that $\{U_i\rightarrow{X}\}_{i\in I}$ is a $\bm\tau$-cover in $\mathcal{A}$.
    \end{enumerate}
\end{defn}
\begin{defn}
\begin{enumerate}
    \item The \textit{truncation functor} $t_0$ is defined to be the functor \begin{equation*}
    t_0:=(\iota|_{\mathcal{A}^\heartsuit})^*:\textcat{Stk}(\mathcal{A},\bm{\tau}|_{\mathcal{A}})\rightarrow{\textcat{Stk}(\mathcal{A}^\heartsuit,\bm{\tau}^\heartsuit)}
\end{equation*}
\item The \textit{extension functor} $i$ is defined to be its left adjoint \begin{equation*}
    i:=(\iota|_{\mathcal{A}^\heartsuit})_{\#}:\textcat{Stk}(\mathcal{A}^\heartsuit,\bm{\tau}^\heartsuit)\rightarrow{\textcat{Stk}(\mathcal{A},\bm{\tau}|_{\mathcal{A}})}
    \end{equation*}
\end{enumerate}

\end{defn}
We can prove representability of a stack $\mathcal{F}$ in $\textcat{Stk}(\mathcal{A},\bm\tau|_\mathcal{A})$ from representability of $t_0(\mathcal{F})$ by lifting an $n$-atlas of $t_0(\mathcal{F})$ to an $n$-atlas of $\mathcal{F}$. To lift such an $n$-atlas requires $\mathcal{F}$ to be nilcomplete and have an obstruction theory. We refer to \cite[Sections 3 and 4]{savage_representability_2024} for notations of derivations, square-zero extensions, and cotangent complexes. 
\begin{defn}\cite[Definition 4.1.1]{savage_representability_2024} A stack $\mathcal{F}\in\textcat{Stk}(\mathcal{A},\bm{\tau}|_\mathcal{A})$ has an \textit{obstruction theory} if it has a global cotangent complex relative to $\mathcal{A}$ and if, for any $X=\textnormal{Spec}(A)\in\mathcal{A}$, any $A$-module $M\in\textcat{M}_{A,1}:=\textcat{M}_A\times_{\textcat{Mod}_A}\textcat{Mod}_A^{\geq 1}$, and any derivation $d\in\pi_0(\textcat{Der}(A,M))$ corresponding to a morphism $d:A\rightarrow{A\oplus M}$, there is a pullback square
\begin{equation*}
    \begin{tikzcd}
        \mathcal{F}(A\oplus_d\Omega M)\arrow{r}\arrow{d} & \mathcal{F}(A)\arrow{d}\\
        \mathcal{F}(A)\arrow{r} & \mathcal{F}(A\oplus M)
    \end{tikzcd}
\end{equation*}This last condition says that $\mathcal{F}$ is \textit{infinitesimally cartesian relative to $\mathcal{A}$}. 
\end{defn}Suppose that we have a morphism $x:X\rightarrow{\mathcal{F}}$. If $\mathcal{F}$ has an obstruction theory, then lifting of this map to a map $x':X_d[\Omega M]:=\textnormal{Spec}(A\oplus_d\Omega M)\rightarrow{\mathcal{F}}$ is controlled by the cotangent complex $\mathbb{L}_{\mathcal{F},x}$ (see \cite[Proposition 4.1.7]{savage_representability_2024}). 

\begin{defn}\cite[Definition 4.5.3]{savage_representability_2024} A stack $\mathcal{F}\in\textcat{Stk}(\mathcal{A},\bm{\tau}|_\mathcal{A})$ is \textit{nilcomplete with respect to $\mathcal{A}$} if, for every $X=\textnormal{Spec}(A)\in\mathcal{A}$, there is an equivalence 
\begin{equation*}
    \mathcal{F}(A)\rightarrow{\varprojlim_k }\,\mathcal{F}(A_{\leq k})
\end{equation*}
    
\end{defn}

We can ensure that we have strong compatibility between a stack $\mathcal{F}$ and its truncation $t_0(\mathcal{F})$ by defining the notion of a representability context (see \cite[Definition 5.3.1]{savage_representability_2024}),
\begin{equation*}
    (\mathcal{C},\mathcal{C}_{\geq 0},\mathcal{C}_{\leq 0},\mathcal{C}^0,\bm\tau,\textcat{P},\mathcal{A},\textcat{M},\textcat{S})
\end{equation*}This consists of a derived geometry context along with a collection of maps $\textcat{S}$ which control the obstruction theory. The classes $\bm\tau$ and $\textcat{P}$ are required to satisfy the \textit{obstruction conditions}, see \cite[Definition 4.3.2]{savage_representability_2024}, relative to $\mathcal{A}$ for the class $\textcat{S}$ of morphisms. This ensures that any $n$-geometric stack in this context has an obstruction theory. 

In this setting we obtain the following representability theorem. 

\begin{thm}\cite[Theorem 5.6.1]{savage_representability_2024} \phantomsection\label{representabilitytheorem}Suppose that $(\mathcal{C},\mathcal{C}_{\geq 0},\mathcal{C}_{\leq 0},\mathcal{C}^0,\bm\tau,\textcat{P},\mathcal{A},\textcat{M},\textcat{S})$ is a representability context and that $\mathcal{F}$ is a stack in $\textcat{Stk}(\mathcal{A},\bm{\tau}|_{\mathcal{A}})$. The following conditions are equivalent. 
\begin{enumerate}
    \item $\mathcal{F}$ is an $n$-geometric stack in $\textcat{Stk}_n(\mathcal{A},\bm\tau|_\mathcal{A},\textcat{P}|_\mathcal{A})$, 
    \item $\mathcal{F}$ satisfies the following three conditions:
    \begin{enumerate}
\item\makeatletter\def\@currentlabel{a}\makeatother\label{representability1}The truncation $t_0(\mathcal{F})$ is an $n$-geometric stack in $\textcat{Stk}_n(\mathcal{A}^\heartsuit,\bm\tau^\heartsuit,\textcat{P}^\heartsuit)$, 
        \item \makeatletter\def\@currentlabel{b}\makeatother\label{representability2} $\mathcal{F}$ has an obstruction theory relative to $\mathcal{A}$,
        \item \makeatletter\def\@currentlabel{c}\makeatother\label{representability3} $\mathcal{F}$ is nilcomplete with respect to $\mathcal{A}$. 
    \end{enumerate}
\end{enumerate}
\end{thm}

\subsection{Representability of Mapping Stacks}

Many moduli stacks naturally appear as mapping stacks when families of geometric objects over some base object naturally correspond to maps from the base into a suitable \textit{classifying stack}. For example, the moduli stack of principal $G$-bundles on some geometric object $X$ (e.g. a scheme, or a stack), where $G$ is some group object, can be expressed as the mapping stack from $X$ to $\textcat{BG}$. 

For the rest of this section, we fix a representability context $(\mathcal{C},\mathcal{C}_{\geq 0},\mathcal{C}_{\leq 0},\mathcal{C}^0,\bm\tau,\textcat{P},\mathcal{A},\textcat{M},\textcat{S})$. We note that the category of presheaves $\textcat{PSh}(\textcat{DAff}^{cn}(\mathcal{C}))$ is locally cartesian closed, and hence, for any $X\in\mathcal{A}$, the slice category $\textcat{PSh}(\textcat{DAff}^{cn}(\mathcal{C}))_{/X}$ is cartesian closed. Suppose that $\mathcal{F}$ and $\mathcal{G}$ are in $\textcat{PSh}(\textcat{DAff}^{cn}(\mathcal{C}))_{/X}$. We note that the presheaf $\texthom{Map}_{\textcat{PSh}(\textcat{DAff}^{cn}(\mathcal{C}))_{/X}}(\mathcal{F},\mathcal{G})$ acts on $\textcat{DAlg}^{cn}(\mathcal{C})$ by
\begin{equation*}
    \texthom{Map}_{\textcat{PSh}(\textcat{DAff}^{cn}(\mathcal{C}))_{/X}}(\mathcal{F},\mathcal{G})(B):=\textnormal{Map}_{\textcat{PSh}(\textcat{DAff}^{cn}(\mathcal{C}))_{/X}}(\mathcal{F}\times Y,\mathcal{G})
\end{equation*}where $Y=\textnormal{Spec}(B)$.

If $\mathcal{F}$ and $\mathcal{G}$ are stacks in $\textcat{Stk}(\mathcal{A},\bm\tau|_\mathcal{A})_{/X}$, then we obtain similar results and the mapping stack $\texthom{Map}_{\textcat{Stk}(\mathcal{A},\bm\tau|_\mathcal{A})_{/X}}(\mathcal{F},\mathcal{G})$ is obtained by stackifying $\texthom{Map}_{\textcat{PSh}(\mathcal{A})_{/X}}(\mathcal{F},\mathcal{G})$. We can apply the representability theorem to simplify the representability theorem for mapping stacks as follows. 

\begin{thm}\phantomsection\label{mappingstackngeometricthm} Suppose that $X=\textnormal{Spec}(A)\in\mathcal{A}^\heartsuit$. Suppose that $\mathcal{F}$ and $\mathcal{G}$ are in $\textcat{Stk}(\mathcal{A},\bm\tau|_\mathcal{A})_{/X}$. Under the following conditions, $\texthom{Map}_{\textcat{Stk}(\mathcal{A},\bm\tau|_\mathcal{A})_{/X}}(\mathcal{F},\mathcal{G})$ is an $n$-geometric stack in $\textcat{Stk}_n(\mathcal{A},\bm\tau|_\mathcal{A},\textcat{P}|_\mathcal{A})$, 
    \begin{enumerate}
        \item The truncation $t_0(\texthom{Map}_{\textcat{Stk}(\mathcal{A},\bm\tau|_\mathcal{A})_{/X}}(\mathcal{F},\mathcal{G}))$ is in $\textcat{Stk}_n(\mathcal{A}^\heartsuit,\bm\tau^\heartsuit,\textcat{P}^\heartsuit)$, 
        \item The stack $\texthom{Map}_{\textcat{Stk}(\mathcal{A},\bm\tau|_\mathcal{A})_{/X}}(\mathcal{F},\mathcal{G})$ has a global cotangent complex relative to $\mathcal{A}$, 
        \item $\mathcal{G}$ is $n$-geometric, 
        \item The stack $\mathcal{F}$ can be written as a colimit of representable stacks $\varinjlim_i U_i$ where each $U_i=\textnormal{Spec}(C_i)\in\mathcal{A}^\heartsuit$ is flat and $B\otimes^\mathbb{L}C_i\in\mathcal{A}^{op}$ for any $B\in\mathcal{A}^{op}$. 
    \end{enumerate}
\end{thm}
\begin{remark}
    We note that if the morphism $U_i\rightarrow{*}$ is in $\textcat{P}$, then the final condition holds by \cite[c.f. Definition 2.1.2]{savage_representability_2024}. 
\end{remark}
\begin{proof}
We just need to check that the conditions of Theorem \ref{representabilitytheorem} hold. Indeed, Condition (\ref{representability1}) immediately holds. Condition (\ref{representability2}) is satisfied if we can show that the mapping stack is infinitesimally cartesian relative to $\mathcal{A}$. Indeed, since $\mathcal{F}$ can be written as a colimit $\varinjlim_i U_i$ where $U_i=\textnormal{Spec}(C_i)\in\mathcal{A}^\heartsuit$ with $C_i$ flat as an $A$-module, we see that 
\begin{equation*}
    \texthom{Map}_{\textcat{Stk}(\mathcal{A},\bm\tau|_\mathcal{A})_{/X}}(\mathcal{F},\mathcal{G})\simeq \varprojlim_i\texthom{Map}_{\textcat{Stk}(\mathcal{A},\bm\tau|_\mathcal{A})_{/X}}(U_i,\mathcal{G})
\end{equation*}Therefore, since the property of being infinitesimally cartesian is stable by limits, we can assume that $\mathcal{F}$ is representable by some $U=\textnormal{Spec}(C)\in\mathcal{A}^\heartsuit$. Suppose that $Y=\textnormal{Spec}(B)$ is in $\mathcal{A}$, $M\in\textcat{M}_{B,1}$, and that we have a derivation $d\in\pi_0(\textcat{Der}(B,M))$ corresponding to a morphism $d:B\rightarrow{B\oplus M}$. Then, using the Yoneda Lemma, we note that the commutative square 
\begin{equation*}
        \begin{tikzcd}
            \texthom{Map}_{\textcat{Stk}(\mathcal{A},\bm\tau|_\mathcal{A})_{/X}}(U,\mathcal{G})(B\oplus_d\Omega M)\arrow{r} \arrow{d}& \texthom{Map}_{\textcat{Stk}(\mathcal{A},\bm\tau|_\mathcal{A})_{/X}}(U,\mathcal{G})(B) \arrow{d}\\
            \texthom{Map}_{\textcat{Stk}(\mathcal{A},\bm\tau|_\mathcal{A})_{/X}}(U,\mathcal{G})(B)\arrow{r} & \texthom{Map}_{\textcat{Stk}(\mathcal{A},\bm\tau|_\mathcal{A})_{/X}}(U,\mathcal{G})(B\oplus M)
        \end{tikzcd}
    \end{equation*}is equivalent to the commutative square
\begin{equation*}
        \begin{tikzcd}
            \mathcal{G}((B\oplus_d\Omega M)\otimes^\mathbb{L}C)\arrow{r} \arrow{d}& \mathcal{G}(B\otimes^\mathbb{L}C) \arrow{d}\\
            \mathcal{G}(B\otimes^\mathbb{L}C)\arrow{r} & \mathcal{G}((B\oplus M)\otimes^\mathbb{L}C)
        \end{tikzcd}
    \end{equation*}Using the flatness of $C$, we can see that $(B\oplus_d\Omega M)\otimes^\mathbb{L}C\simeq (B\otimes^\mathbb{L}C)\oplus_{d'}\Omega M'$ where $M'=M\otimes^\mathbb{L}C$ and $d'$ is the induced derivation. Now, since $\mathcal{G}$ is $n$-geometric, it is infinitesimally cartesian relative to $\mathcal{A}$ by \cite[Theorem 4.4.1]{savage_representability_2024}. Therefore, we see that the above square is a pullback square, and hence $\texthom{Map}_{\textcat{Stk}(\mathcal{A},\bm\tau|_\mathcal{A})_{/X}}(\mathcal{F},\mathcal{G})$ is infinitesimally cartesian. 

Finally, we need to show that $\texthom{Map}_{\textcat{Stk}(\mathcal{A},\bm\tau|_\mathcal{A})_{/X}}(\mathcal{F},\mathcal{G})$ is nilcomplete. Once again it suffices, using our fourth condition, to prove the statement when $\mathcal{F}$ is a representable stack $U=\textnormal{Spec}(C)\in\mathcal{A}^\heartsuit$ such that $B\otimes^\mathbb{L}C\in\mathcal{A}^{op}$ for any $B\in\mathcal{A}^{op}$. Indeed, suppose that $Y=\textnormal{Spec}(B)\in\mathcal{A}$. Then, 
    \begin{equation*}
        \texthom{Map}_{\textcat{Stk}(\mathcal{A},\bm\tau|_\mathcal{A})_{/X}}(U,\mathcal{G})(B)\simeq 
        \mathcal{G}(B\otimes^\mathbb{L}C)
    \end{equation*}By \cite[Lemma 5.2.1]{savage_representability_2024}, we see that $(B\otimes^\mathbb{L}C)_{\leq k}\simeq B_{\leq k}\otimes^\mathbb{L}C$. Therefore, our result follows since $\mathcal{G}$ is $n$-geometric, and hence nilcomplete.  
\end{proof}

\subsection{Perfect Quasi-coherent Sheaves and Base-Change}

For a derived affine $X=\textnormal{Spec}(A)\in\textcat{DAff}^{cn}(\mathcal{C})$, we define $\textcat{QCoh}(X):=\textcat{Mod}_A$ and, for any morphism $f:Y\rightarrow{X}$, $\textcat{QCoh}(f)$ is defined to be the colimit-preserving functor $B\otimes_A^\mathbb{L}-:\textcat{Mod}_A\rightarrow{\textcat{Mod}_B}$. 
This defines a $\textcat{Cat}$-valued presheaf
\begin{equation*}
    \textcat{QCoh}:\textcat{DAff}^{cn}(\mathcal{C})^{op}\rightarrow{\textcat{Pr}^{\mathbb{L},\otimes}}
\end{equation*}where $\textcat{Pr}^{\mathbb{L},\otimes}$ is the $(\infty,1)$-category whose objects are locally presentable monoidal $(\infty,1)$-categories and morphisms are left adjoint functors. 

We left Kan extend this functor to a functor defined on $\textcat{PSh}(\textcat{DAff}^{cn}(\mathcal{C}))^{op}$. We note that, since each $\textcat{Mod}_A$ is a stable $(\infty,1)$-category, and the category of stable locally presentable monoidal $(\infty,1)$-categories is closed under colimits, then $\textcat{QCoh}(\mathcal{F})$ is a stable $(\infty,1)$-category for any $\mathcal{F}\in\textcat{PSh}(\textcat{DAff}^{cn}(\mathcal{C}))$. Suppose that we have a morphism $f:\mathcal{F}\rightarrow{\mathcal{G}}$ of presheaves in $\textcat{PSh}(\textcat{DAff}^{cn}(\mathcal{C}))$. Then, there is a natural functor $f^*:\textcat{QCoh}(\mathcal{G})\rightarrow{\textcat{QCoh}(\mathcal{F})}$ which has a right adjoint given by right Kan extension $f_*:\textcat{QCoh}(\mathcal{F})\rightarrow{\textcat{QCoh}(\mathcal{G})}$.

For any $\mathcal{F}\in\textcat{PSh}(\textcat{DAff}^{cn}(\mathcal{C}))$, we denote by $\mathcal{O}_\mathcal{F}$ the monoidal unit in $\textcat{QCoh}(\mathcal{F})$. 

\begin{defn}
    Define $\textcat{Perf}(\mathcal{F})$ to be the subcategory of perfect objects in $\textcat{QCoh}(\mathcal{F})$ in the sense of Appendix \ref{perfectappendix}, i.e. retracts of finite colimits of $\coprod_E \mathcal{O}_\mathcal{F}$ for some finite set $E$. 
\end{defn}

We note that objects in $\textcat{Perf}(\mathcal{F})$ are strongly dualisable and reflexive in $\textcat{QCoh}(\mathcal{F})$. If we have a morphism $f:\mathcal{F}\rightarrow{\mathcal{G}}$ of presheaves in $\textcat{PSh}(\textcat{DAff}^{cn}(\mathcal{C}))$, then since $f^*$ is symmetric monoidal and preserves colimits, $f^*\textcat{Perf}(\mathcal{G})\subseteq\textcat{Perf}(\mathcal{F})$. If we have a morphism of representables $f:Y=\textnormal{Spec}(B)\rightarrow{\textnormal{Spec}(A)=X}$ in $\textcat{PSh}(\textcat{DAff}^{cn}(\mathcal{C}))$, then we note that the morphism $f_*$ corresponds to the morphism $f_*:\textcat{Mod}_B\rightarrow{\textcat{Mod}_A}$ which is exact and symmetric monoidal. Hence, in particular, $f_*$ preserves finite colimits, and therefore $f_*\textcat{Perf}(Y)\subseteq\textcat{Perf}(X)$. 

\begin{defn}
    Suppose that we have a morphism $f:\mathcal{F}\rightarrow{\mathcal{G}}$ in $\textcat{PSh}(\textcat{DAff}^{cn}(\mathcal{C}))$. Then, we say that \textit{$f$ satisfies the perfect projection formula} if, for any $\mathcal{M}\in\textcat{Perf}(\mathcal{F})$ and $\mathcal{N}\in\textcat{QCoh}(\mathcal{G})$, there is an equivalence
    \begin{equation*}
        f_*(\mathcal{M}\otimes_{\textcat{QCoh}(\mathcal{F})}f^*(\mathcal{N}))\simeq f_*(\mathcal{M})\otimes_{\textcat{QCoh}(\mathcal{G})}\mathcal{N}
    \end{equation*}
\end{defn}

\begin{remark}
    If $f:Y\rightarrow{X}$ is a morphism of representables, then $f$ satisfies the perfect projection formula. 
\end{remark}

Suppose that $\mathcal{F}\in\textcat{PSh}(\textcat{DAff}^{cn}(\mathcal{C}))$. For any $\mathcal{M}\in\textcat{QCoh}(\mathcal{F})$, we consider the dual object $\mathcal{M}^\vee:=\texthom{Map}_{\textcat{QCoh}(\mathcal{F})}(\mathcal{M},\mathcal{O}_\mathcal{F})$ and denote the associated functor by $(-)_\mathcal{F}^\vee:\textcat{QCoh}(\mathcal{F})\rightarrow{\textcat{QCoh}(\mathcal{F})}$. Suppose that we have a map $f:\mathcal{F}\rightarrow{\mathcal{G}}$ in $\textcat{PSh}(\textcat{DAff}^{cn}(\mathcal{C}))$. Then, we define the plus pushforward $f_+:\textcat{QCoh}(\mathcal{F})\rightarrow{\textcat{QCoh}(\mathcal{G})}$ by
\begin{equation*}
    f_+:=(-)_\mathcal{G}^\vee \circ f_*\circ (-)_\mathcal{F}^\vee
\end{equation*}

\begin{lem}\phantomsection\label{perfectprojectionformula}Suppose that $f:\mathcal{F}\rightarrow{\mathcal{G}}$ is a map such that $f_*\textcat{Perf}(\mathcal{F})\subseteq \textcat{Perf}(\mathcal{G})$ and which satisfies the perfect projection formula. Then, for any $\mathcal{M}\in\textcat{Perf}(\mathcal{F})$ and $\mathcal{N}\in\textcat{QCoh}(\mathcal{G})$, there is an equivalence, natural in $\mathcal{M}$ and $\mathcal{N}$, 
\begin{equation*}
    \textnormal{Map}_{\textcat{QCoh}(\mathcal{G})}(f_+(\mathcal{M}),\mathcal{N})\simeq \textnormal{Map}_{\textcat{QCoh}(\mathcal{F})}(\mathcal{M},f^*(\mathcal{N}))
\end{equation*}
\end{lem}

\begin{proof}Using the perfect projection formula and our conditions, the result easily follows from the following chain of equivalences. 
\begin{align*}
    f_*(\texthom{Map}_{\textcat{QCoh}(\mathcal{F})}(\mathcal{M},f^*(\mathcal{N)}))&\simeq f_*(\mathcal{M}^\vee\otimes_{\textcat{QCoh}(\mathcal{F})}f^*(\mathcal{N}))\\
    &\simeq f_*(\mathcal{M}^\vee)\otimes_{\textcat{QCoh}(\mathcal{G})}\mathcal{N}\\
    &\simeq \texthom{Map}_{\textcat{QCoh}(\mathcal{G})}(f_*(\mathcal{M}^\vee)^\vee,\mathcal{N})\\
    &\simeq \texthom{Map}_{\textcat{QCoh}(\mathcal{G})}(f_+(\mathcal{M}),\mathcal{N})
\end{align*}
\end{proof}

\begin{defn}
     Suppose that $f:\mathcal{F}\rightarrow{\mathcal{G}}$ and $g:\mathcal{H}\rightarrow{\mathcal{G}}$ are morphisms of presheaves in $\textcat{PSh}(\textcat{DAff}^{cn}(\mathcal{C}))$. Then, $(f,g)$ \textit{satisfies perfect base change} if, given the pullback diagram
    \begin{equation*}
        \begin{tikzcd}
            \mathcal{F}\times_\mathcal{G}\mathcal{H}\arrow{r}{g'}\arrow{d}{f'} & \mathcal{F} \arrow{d}{f}\\
            \mathcal{H}\arrow{r}{g} & \mathcal{G}
        \end{tikzcd}
    \end{equation*}the natural map $g^*f_*\mathcal{M}\rightarrow{f_*'(g')^*\mathcal{M}}$ is an equivalence whenever $\mathcal{M}\in\textcat{Perf}(\mathcal{F})$. 
\end{defn}

\begin{prop}\phantomsection\label{basechangeperfect} Suppose that we have a pullback diagram
        \begin{equation*}
         \begin{tikzcd}
            \mathcal{F}\times_\mathcal{G}\mathcal{H}\arrow{r}{g'} 
 \arrow{d}{f'} & \mathcal{F} \arrow{d}{f}\\
            \mathcal{H}\arrow{r}{g} & \mathcal{G}
        \end{tikzcd}   
        \end{equation*}Suppose that $(f,g)$ satisfies perfect base change and that $f_*\textcat{Perf}(\mathcal{F})\subseteq\textcat{Perf}(\mathcal{G})$. Suppose that $\mathcal{M}\in\textcat{Perf}(\mathcal{F})$. Then, there is a natural equivalence
        \begin{equation*}
            (f')_+(g')^*\mathcal{M}\rightarrow{g^*f_+\mathcal{M}}
        \end{equation*}
\end{prop}
\begin{proof}
Using our assumptions, we have the following chain of equivalences
\begin{equation*}
    (f')_+(g')^*\mathcal{M}\simeq (-)_{\mathcal{H}}^\vee\circ (f')_*\circ (g')^* \mathcal{M}^\vee\simeq (-)^\vee_\mathcal{H}\circ g^*f_*\mathcal{M}^\vee\simeq g^*f_+\mathcal{M}
\end{equation*}

\end{proof}

\subsection{The Cotangent Complex of the Mapping Presheaf}

We note that, if a morphism $f:\mathcal{F}\rightarrow{\mathcal{G}}$ of presheaves in $\textcat{PSh}(\textcat{DAff}^{cn}(\mathcal{C}))$ has a global cotangent complex $\mathbb{L}_{\mathcal{F}/\mathcal{G},u}$ for every point $u:U\rightarrow{\mathcal{F}}$, then this determines an object $\mathbb{L}_{\mathcal{F}/\mathcal{G}}\in\textcat{QCoh}(\mathcal{F})$ such that $u^*(\mathbb{L}_{\mathcal{F}/\mathcal{G}})\simeq \mathbb{L}_{\mathcal{F}/\mathcal{G},u}$. 

Suppose that $X=\textnormal{Spec}(A)\in\mathcal{A}^\heartsuit$ and that $\mathcal{F}$ and $\mathcal{G}$ are in $\textcat{PSh}(\textcat{DAff}^{cn}(\mathcal{C}))_{/X}$. Suppose that we have a point $u:U=\textnormal{Spec}(C)\rightarrow{\texthom{Map}_{\textcat{PSh}(\textcat{DAff}^{cn}(\mathcal{C}))_{/X}}(\mathcal{F},\mathcal{G})}$ and consider the following pullback diagram in $\textcat{PSh}(\textcat{DAff}^{cn}(\mathcal{C}))$.
\begin{equation}\phantomsection\label{pullbackperfectbasechange}
    \begin{tikzcd}
         \mathcal{F}\times_{X}U\arrow{r}{u'} 
 \arrow{d}{\pi_u} & \mathcal{F}\times_X\texthom{Map}_{\textcat{PSh}(\textcat{DAff}^{cn}(\mathcal{C}))_{/X}}(\mathcal{F},\mathcal{G}) \arrow{d}{\pi}\\
            U\arrow{r}{u} & \texthom{Map}_{\textcat{PSh}(\textcat{DAff}^{cn}(\mathcal{C}))_{/X}}(\mathcal{F},\mathcal{G})
    \end{tikzcd}
\end{equation}

\begin{prop}\phantomsection\label{mappingcotangentcomplex}
    Suppose that the following conditions are satisfied:
    \begin{enumerate}
        \item\label{mapitem1} The morphism $\mathcal{G}\rightarrow{X}$ has a global cotangent complex and this is an object in $\textcat{Perf}(\mathcal{G})$, 
        \item\label{mapitem2} The morphism $\mathcal{F}\rightarrow{X}$ is $(-1)$-representable,
        \item\label{mapitem3} $(\pi,u)$ satisfies perfect base change for any morphism $u$ and $\pi_*$ preserves perfect objects.
    \end{enumerate}Then, the morphism $\texthom{Map}_{\textcat{PSh}(\textcat{DAff}^{cn}(\mathcal{C}))_{/X}}(\mathcal{F},\mathcal{G})\rightarrow{X}$ has a global cotangent complex given by
    \begin{equation*}
        \mathbb{L}_{\texthom{Map}_{\textcat{PSh}(\textcat{DAff}^{cn}(\mathcal{C}))_{/X}}(\mathcal{F},\mathcal{G})/X}=\pi_+\circ ev^*(\mathbb{L}_{\mathcal{G}/X})
    \end{equation*}where $ev$ is the evaluation morphism $ev:\mathcal{F}\times_X\texthom{Map}_{\textcat{PSh}(\textcat{DAff}^{cn}(\mathcal{C}))_{/X}}(\mathcal{F},\mathcal{G})\rightarrow{\mathcal{G}}$. 
    
\end{prop}

\begin{proof}
    Denote by $f_u$ the following composition of maps defining $\mathcal{F}\times_XU$ as a point of $\mathcal{G}$, 
    \begin{equation*}
        \mathcal{F}\times_XU\xrightarrow{u'}{\mathcal{F}\times_X\texthom{Map}_{\textcat{PSh}(\textcat{DAff}^{cn}(\mathcal{C}))_{/X}}(\mathcal{F},\mathcal{G})}\xrightarrow{ev}{\mathcal{G}}
    \end{equation*}By our assumptions, $\mathcal{F}\times_XU$ is representable. Suppose that $M\in\textcat{M}_C^{cn}$. In the following, we will abbreviate $\textcat{PSh}(\textcat{DAff}^{cn}(\mathcal{C}))$ to $\textcat{PSh}$. Then, 
    \begin{align*}
\textcat{Der}_{\texthom{Map}_{\textcat{PSh}_{/X}}(\mathcal{F},\mathcal{G})}(U,M)&:=\textnormal{Map}_{{}^{U}/\textcat{PSh}}(U[M],\texthom{Map}_{\textcat{PSh}_{/X}}(\mathcal{F},\mathcal{G}))\\
    &\simeq \textnormal{Map}_{{}^{\mathcal{F}\times_XU}/\textcat{PSh}_{/X}}(\mathcal{F}\times_X U[M],\mathcal{G})\\
    &\simeq \textnormal{Map}_{{}^{\mathcal{F}\times_XU}/\textcat{PSh}_{/X}}((\mathcal{F}\times_X U)[(\pi_u)^*(M)],\mathcal{G})\\
    &\simeq \textnormal{Map}_{\textcat{QCoh}(\mathcal{F}\times_XU)}(f_u^*(\mathbb{L}_{\mathcal{G}/X}),(\pi_u)^*(M))\\
\intertext{Now, since $\pi_u$ satisfies the perfect projection formula, $(\pi_u)_*$ preserves perfect objects, and $f_u^*(\mathbb{L}_{\mathcal{G}/X})\in\textcat{Perf}(\mathcal{F}\times_XU)$, we can apply Lemma \ref{perfectprojectionformula},}
    &\simeq \textnormal{Map}_{\textcat{QCoh}(U)}((\pi_u)_+\circ f_u^*(\mathbb{L}_{\mathcal{G}/X}),M)\\
    &\simeq \textnormal{Map}_{\textcat{QCoh}(U)}((\pi_u)_+\circ (u')^*\circ ev^*(\mathbb{L}_{\mathcal{G}/X}),M)\\
\intertext{Now, since $(\pi,u)$ satisfies perfect base-change and $\pi_*$ preserves perfect objects then, since $ev^*(\mathbb{L}_{\mathcal{G}/X})$ is perfect, we can apply Proposition \ref{basechangeperfect} to conclude that}
    &\simeq \textnormal{Map}_{\textcat{QCoh}(U)}(u^*\circ \pi_+\circ ev^*(\mathbb{L}_{\mathcal{G}/X}),M)
    \end{align*}By its definition, we can easily see that $\mathbb{L}_{\texthom{Map}_{\textcat{PSh}(\textcat{DAff}^{cn}(\mathcal{C}))_{/X}}(\mathcal{F},\mathcal{G})/X}$ is global.

\end{proof}

\begin{remark}We can then define $\mathbb{L}_{\texthom{Map}_{\textcat{PSh}(\textcat{DAff}^{cn}(\mathcal{C}))_{/X}}(\mathcal{F},\mathcal{G})}$ to be the fibre of the morphism 
\begin{equation*}
\mathbb{L}_{\texthom{Map}_{\textcat{PSh}(\textcat{DAff}^{cn}(\mathcal{C}))_{/X}}(\mathcal{F},\mathcal{G})/X}\rightarrow\mathbb{L}_X[1]
\end{equation*}in $\textcat{QCoh}(\texthom{Map}_{\textcat{PSh}(\textcat{DAff}^{cn}(\mathcal{C}))_{/X}}(\mathcal{F},\mathcal{G}))$.
\end{remark}

\begin{cor}\phantomsection\label{mappingstacksimplified}
    Suppose that $X=\textnormal{Spec}(A)\in\mathcal{A}^\heartsuit$ is such that $B\otimes^\mathbb{L}A\in\mathcal{A}^{op}$ for any $A\in\mathcal{A}^{op}$. Suppose that $\mathcal{G}$ is in $\textcat{Stk}(\mathcal{A},\bm\tau|_\mathcal{A})_{/X}$. Under the following conditions, $\texthom{Map}_{\textcat{Stk}(\mathcal{A},\bm\tau|_\mathcal{A})_{/X}}(X,\mathcal{G})$ is an $n$-geometric stack in $\textcat{Stk}_n(\mathcal{A},\bm\tau|_\mathcal{A},\textcat{P}|_\mathcal{A})$. 
    \begin{enumerate}
        \item\label{mapcor1} The truncation $t_0(\texthom{Map}_{\textcat{Stk}(\mathcal{A},\bm\tau|_\mathcal{A})_{/X}}(X,\mathcal{G}))$ is in $\textcat{Stk}_n(\mathcal{A}^\heartsuit,\bm\tau^\heartsuit,\textcat{P}^\heartsuit)$, 
        \item\label{mapcor2} $\mathcal{G}$ is $n$-geometric and the cotangent complex of the morphism $\mathcal{G}\rightarrow{X}$ is in $\textcat{Perf}(\mathcal{G})$, 
    \end{enumerate}

    Moreover, if $\mathbb{L}_X$ is perfect, then $\mathbb{L}_{\texthom{Map}_{\textcat{Stk}(\mathcal{A},\bm\tau|_{\mathcal{A}})_{/X}}(X,\mathcal{G})}$ is also perfect. 
\end{cor}

\begin{proof}
    This follows by Theorem \ref{mappingstackngeometricthm} using Proposition \ref{mappingcotangentcomplex}. 
    To see the final result, we note that, since $\mathbb{L}_{\mathcal{G}/X}$ is perfect, $\mathbb{L}_{\texthom{Map}_{\textcat{Stk}(\mathcal{A},\bm\tau|_\mathcal{A})_{/X}}(X,\mathcal{G})/X}=( ev^*(\mathbb{L}_{\mathcal{G}/X}))^{\vee\vee}$ is also perfect. Therefore, since the fibre of a morphism of perfect objects is also perfect, $\mathbb{L}_{\texthom{Map}_{\textcat{Stk}(\mathcal{A},\bm\tau|_{\mathcal{A}})_{/X}}(X,\mathcal{G})}$ is perfect. 

\end{proof}

\section{$\mathcal{C}^\infty$-Bornological Rings}\label{Cinfinitybornsection}

As explained in the introduction, there seems to be a natural extension of $\mathcal{C}^\infty$-rings to also include smooth functions on certain infinite dimensional manifolds. In this section, we describe this theory in more detail.

\subsection{Convenient Manifolds}

One approach to defining infinite dimensional manifolds comes from manifolds modelled on the notion of a \textit{convenient space}. We remark that the category of finite dimensional manifolds is not closed monoidal but the category of convenient manifolds is. Our main reference for this section is the book \textit{The Convenient Setting of Global Analysis} by Kriegl and Michor \cite{kriegl_convenient_1997}. 

\begin{defn}\cite[Section 1.2]{kriegl_convenient_1997} Suppose that $X$ is a locally convex topological vector space and that $c:\mathbb{R}\rightarrow{X}$ is a curve.
\begin{enumerate}
     \item $c$ is \textit{differentiable} if the derivative $c'(t):=\lim_{s\rightarrow{0}}\frac{1}{s}(c(t+s)-c(t))$ at $t$ exists for all $t$,
    \item  $c$ is \textit{smooth} if all iterated derivatives exist,
    \item The \textit{$c^\infty$-topology} on a locally convex space $X$ is the final topology with respect to all smooth curves $\mathbb{R}\rightarrow{X}$. Its open sets will be called \textit{$c^\infty$-open}.
\end{enumerate}
\end{defn}

In the following, for $X$ a locally convex topological vector space, $X^*$ denotes the locally convex space of continuous linear functionals on $X$. 

\begin{defn}\phantomsection\label{convenientspacedefn}
   Suppose that $X$ is a locally convex topological vector space. Then, $X$ is \textit{convenient} if it satisfies the equivalent conditions of \cite[Theorem 2.14]{kriegl_convenient_1997}. In particular, $X$ is convenient if, whenever $c:\mathbb{R}\rightarrow{X}$ is a curve such that $f\circ c:\mathbb{R}\rightarrow{\mathbb{R}}$ is smooth for all $f\in X^*$, then $c$ is smooth. 
\end{defn}

We denote by $\textnormal{Con}$ the category of convenient spaces equipped with the von Neumann bornology, with morphisms being bounded linear maps. We note that the category $\textnormal{Con}$ is a full subcategory of $\textnormal{CBorn}_\mathbb{R}$ and that Banach and Fr\'echet spaces are convenient spaces. The category $\textnormal{Con}$ is a cartesian closed category. As described in \cite{blute_convenient_2012}, the category $\textnormal{Con}$ can be described as the `fixed points' of the adjunction between separated bornological spaces of convex type and locally convex topological vector spaces. 

\begin{defn}\cite[Section 3.11]{kriegl_convenient_1997}
Suppose that we have a morphism $f:X\supseteq U\rightarrow{Y}$ defined on a $c^\infty$-open subset $U$ of $X$, where $X$ and $Y$ are convenient spaces. Then, $f$ is called \textit{smooth} if it maps smooth curves in $U$ to smooth curves in $Y$.
\end{defn}

\begin{remark}
    We can naturally define the notion of a \textit{convenient manifold} whose atlas consists of bijections onto convenient spaces. We can give any convenient space $X$ the structure of a convenient manifold by defining the atlas to be $\{(X,\textnormal{id}_X)\}$. In particular, $\mathbb{R}^n$ can be considered as a convenient manifold. We note that the category of convenient manifolds equipped with smooth mappings is cartesian closed. 
\end{remark}

We have the following relationships between smooth maps of convenient spaces and bounded linear maps. 

\begin{lem}\phantomsection\label{boundediffsmoothlinear}
    \cite[Corollary 2.11]{kriegl_convenient_1997} A linear map $f:X\rightarrow{Y}$ between convenient spaces is bounded if and only if it is smooth.
\end{lem}

\begin{thm}\phantomsection\label{freeconvenientspace}\cite[c.f. Theorem 23.6]{kriegl_convenient_1997} Suppose that we have a convenient vector space $X$. Then, there exists a convenient vector space $\lambda X$ along with a smooth mapping $\delta_X:X\rightarrow{\lambda X}$ such that, for every smooth mapping $f:X\rightarrow{Y}$ with values in a convenient vector space $Y$, there exists a unique linear bounded map $\tilde{f}:\lambda X\rightarrow{Y}$ such that $\tilde{f}\circ\delta_X=f$. Moreover, there is an isomorphism of convenient vector spaces
\begin{equation*}
    \textnormal{Hom}_{\textnormal{Con}}(\lambda X,Y)\simeq \textnormal{Hom}_{\textnormal{Smooth}}(X,Y)
\end{equation*}
\end{thm}

\begin{defn}
    Given a convenient space $X$ we denote by $\mathcal{C}^\infty(X)$ the space of smooth maps $X\rightarrow{\mathbb{R}}$. 
\end{defn}

We note that by \cite[Section 27.17]{kriegl_convenient_1997}, $\mathcal{C}^\infty(X)$ is a convenient space. The bounded sets $B\subseteq \mathcal{C}^\infty(X)$ are those such that, for every smooth curve $c:\mathbb{R}\rightarrow{X}$, the subset $\{f\circ c\mid f\in B\}$ of $\mathcal{C}^\infty(\mathbb{R})$ is bounded. By \cite[Lemma 3.9 and 27.17]{kriegl_convenient_1997}, this is equivalent to saying that for every smooth curve $c:\mathbb{R}\rightarrow{X}$ and every compact subset $I\subseteq \mathbb{R}$, \begin{equation*}
    \sup_{f\in B}\sup_{t\in I}|d^n (f\circ c)(t)|<\infty
\end{equation*}for every $n\in\mathbb{N}$.
\begin{lem}\phantomsection\label{cinfinityalgebrastructure}
    $\mathcal{C}^\infty(X)$ has a convenient algebra structure with multiplication determined by the multiplication map $\mathbb{R}\times\mathbb{R}\rightarrow{\mathbb{R}}$. 
\end{lem}

\begin{proof}
    We first note that $\mathcal{C}^\infty(X)\times\mathcal{C}^\infty(X)\simeq\textnormal{Hom}_{\textnormal{Smooth}}(X,\mathbb{R}\times\mathbb{R})$. Moreover, using the exponential law given in \cite[Section 27.17]{kriegl_convenient_1997}, we see that 
    \begin{align*}
\textnormal{Hom}_{\textnormal{Smooth}}(\mathcal{C}^\infty(X)\times\mathcal{C}^\infty(X),\mathcal{C}^\infty(X))&\simeq \textnormal{Hom}_{\textnormal{Smooth}}(\textnormal{Hom}_{\textnormal{Smooth}}(X,\mathbb{R}\times\mathbb{R})\times X,\mathbb{R})
\end{align*}Composing the evaluation map $\textnormal{Hom}_{\textnormal{Smooth}}(X,\mathbb{R}\times \mathbb{R})\times X\rightarrow{\mathbb{R}\times \mathbb{R}}$, which we can see is smooth by a simple application of the exponential law, with the smooth multiplication map $\mathbb{R}\times\mathbb{R}\rightarrow{\mathbb{R}}$, we obtain a smooth map \begin{equation*}
    \textnormal{Hom}_{\textnormal{Smooth}}(X,\mathbb{R}\times \mathbb{R})\times X\rightarrow{\mathbb{R}\times \mathbb{R}}\rightarrow{\mathbb{R}}
\end{equation*}Hence, we obtain a smooth map $\mathcal{C}^\infty(X)\times\mathcal{C}^\infty(X)\rightarrow{\mathcal{C}^\infty(X)}$ of convenient spaces. Since this morphism is also linear we note that it is bounded by Lemma \ref{boundediffsmoothlinear}. 
\end{proof}

\subsection{$\mathcal{C}^\infty$-Bornological Rings}
We recall from Theorem \ref{siftedcborncompletion}, that there is an equivalence of categories
\begin{equation*}
    \textnormal{LH}(\textnormal{CBorn}_\mathbb{R})\simeq \textnormal{SInd}(\textnormal{Lin}_\mathbb{R})
\end{equation*}where $\textnormal{Lin}_\mathbb{R}$ is the category whose objects are of the form $\ell^1(\kappa)$ for $\kappa<\aleph$. For ease of notation, for the rest of this paper we will abbreviate $\textnormal{Lin}_\mathbb{R}$ to $\textnormal{Lin}$. We recall that the category of $\mathcal{C}^\infty$-rings can be described as the free sifted cocompletion of the opposite of the category $\textnormal{CartSp}$ consisting of manifolds of the form $\mathbb{R}^n$ and smooth maps between them. We note that $\mathbb{R}^n$ is equivalent to $\ell^1(n)$. 

As motivated in the Introduction, we want to expand the category $\textnormal{CartSp}$ to contain infinite dimensional manifolds of the form $\ell^1(\kappa)$ for $\kappa<\aleph$. We note that smooth maps in $\textnormal{CartSp}$ are bounded on bounded subsets but that smooth maps between convenient spaces do not necessarily have this property. In the following we choose to restrict to smooth maps which are bounded on bounded subsets in order to control the bornologies of the objects we are considering.  

\begin{defn}
    We define the category $\mathcal{C}^\infty_{bb}$ to be the category with the same objects as $\textnormal{Lin}$, but whose morphisms are defined to be morphisms which are smooth and bounded in a bornological sense, i.e. morphisms are bounded on bounded subsets. 
\end{defn}
\begin{remark}
We use the notation $\mathcal{C}^\infty_{bb}$ so as to avoid confusion with \cite[Definition 15.1]{kriegl_convenient_1997}.
\end{remark}
We note that in $\textnormal{Lin}$, considered as a subset of convenient spaces, finite products and coproducts exist and coincide. We have the direct sum 
    \begin{equation*}
        \ell^1(\kappa)\oplus\ell^1(\mu):=\ell^1(\kappa\coprod\mu)
    \end{equation*}By Lemma \ref{boundediffsmoothlinear}, the bounded linear maps in $\textnormal{Lin}$ are smooth maps, and hence there is a natural inclusion functor $\iota:\textnormal{Lin}\rightarrow{\mathcal{C}^\infty_{bb}}$. 
\begin{lem}\label{iotapresevfiniteprod}
    $\mathcal{C}^\infty_{bb}$ has finite products and $\iota:\textnormal{Lin}\rightarrow{\mathcal{C}^\infty_{bb}}$ creates them. 
\end{lem}
\begin{proof}
     Suppose that we have a product $\ell^1(\mu)\times\ell^1(\kappa)$ in $\textnormal{Lin}$ and an object $\ell^1(\nu)$ along with morphisms $f:\ell^1(\nu)\rightarrow{\ell^1(\mu)}$ and $g:\ell^1(\nu)\rightarrow{\ell^1(\kappa)}$ in $\mathcal{C}^\infty_{bb}$. We define a map $h:\ell^1(\nu)\rightarrow{\ell^1(\mu)\times\ell^1(\kappa)}$ by sending any sequence $x=(x_i)_{i\in\nu}$ to the sequence $(f(x),g(x))\in\ell^1(\mu)\times\ell^1(\kappa)$. We need to show that $h$ is in $\mathcal{C}^\infty_{bb}$. Indeed, we note that $h$ is smooth since the product of smooth maps is smooth. Moreover, if $B\subseteq \ell^1(\nu)$ is bounded, then $f(B)$ is bounded in $\ell^1(\mu)$ and $g(B)$ is bounded in $\ell^1(\kappa)$, and hence, using the product bornology, $h(B)$ is also bounded
\end{proof}

We make the following natural definition of a $\mathcal{C}^\infty$-bornological ring.
\begin{defn}
    The \textit{category of $\mathcal{C}^\infty$-bornological rings}, denoted by $\mathcal{C}^\infty\textnormal{BornRing}$, is the free sifted cocompletion of $\mathcal{C}^{\infty,op}_{bb}$, i.e. 
    \begin{equation*}
        \mathcal{C}^\infty\textnormal{BornRing}:=\textnormal{SInd}(\mathcal{C}_{bb}^{\infty,op})=\textnormal{Fun}^\times(\mathcal{C}^\infty_{bb},\textnormal{Set})
    \end{equation*}
\end{defn}
\begin{remark}
    We also note that $\mathcal{C}^\infty\textnormal{BornRing}$ can be described as the algebras for the multisorted Lawvere theory defined by $\mathcal{C}^\infty_{bb}$. The monoidal structure is induced by naturally extending the natural monoidal product on $\mathcal{C}^{\infty,op}_{bb}$ inherited from $\textnormal{CBorn}_\mathbb{R}$ to the sifted cocompletion.
\end{remark}
\begin{lem}\phantomsection\label{fullyfaithfulcinf}
    There is a fully faithful embedding $i:\mathcal{C}^\infty\textnormal{Ring}\rightarrow{\mathcal{C}^\infty\textnormal{BornRing}}$. 
\end{lem}

\begin{proof}
 $\textnormal{CartSp}$ is a full subcategory of $\mathcal{C}^{\infty}_{bb}$. Hence we obtain a fully faithful functor on the level of free sifted cocompletions by Lemma \ref{fullyfaithfulsind}. 
\end{proof}

\begin{lem}\label{bornbicomplete}
    $\mathcal{C}^\infty\textnormal{BornRing}$ has all small limits and colimits. 
\end{lem}
\begin{proof}
    Indeed, we easily see, since limits and colimits in $\textnormal{Fun}(\mathcal{C}^\infty_{bb},\textnormal{Set})$ are computed pointwise, that $\mathcal{C}^\infty\textnormal{BornRing}$ has all small limits and filtered colimits, since these commute with finite products. Now, since $\mathcal{C}_{bb}^{\infty,op}$ has finite coproducts we can easily define finite coproducts in $\mathcal{C}^\infty\textnormal{BornRing}$ by extending the definition in $\mathcal{C}_{bb}^{\infty,op}$ by sifted colimits. Hence, since $\mathcal{C}^\infty\textnormal{BornRing}$ has sifted colimits and finite coproducts then it has all small colimits.
\end{proof}

We denote by $\textnormal{Comm}(\textnormal{CBorn}_\mathbb{R})$ the category of commutative algebra objects in $\textnormal{CBorn}_\mathbb{R}$. We can define a functor $\mathcal{C}^\infty_{bb}:\mathcal{C}^{\infty,op}_{bb}\rightarrow{\textnormal{Comm}(\textnormal{CBorn}_\mathbb{R})}$ which takes any element $\ell^1(\kappa)\in\mathcal{C}^\infty_{bb}$ and considers the space $\mathcal{C}^\infty_{bb}(\ell^1(\kappa))$ of smooth bounded functions $\ell^1(\kappa)\rightarrow{\mathbb{R}}$. We can use Lemma \ref{cinfinityalgebrastructure} and the statement that bounded maps are closed under products and sums to show that there is a natural algebra structure on $\mathcal{C}^\infty_{bb}(\ell^1(\kappa))$.

We note however that bounded linear algebra maps $\mathcal{C}^\infty_{bb}(\ell^1(\kappa))\rightarrow{\mathcal{C}^\infty_{bb}(\ell^1(\mu))}$ don't necessarily determine smooth bounded maps $\ell^1(\mu)\rightarrow{\ell^1(\kappa)}$. Indeed, we can easily construct a map into $\ell^\infty(\kappa)$ but it is not clear that we can control convergence of the sum of norms in the situation when $\kappa$ is not finite.

\begin{defn}
    \cite[Section 5]{adam_countably_1999} Suppose that $X$ is a convenient space. 
    \begin{enumerate} 
        \item Denote by $\mathcal{C}^\infty_{lfcs}(X)$ the smallest subalgebra of $\mathcal{C}^\infty(X)$ containing $X^\vee:=\texthom{Hom}_{\textnormal{Con}}(X,\mathbb{R})$, the space of bounded linear functionals on $X$, and which is closed under locally finite countable sums, 
        \item $X$ is \textit{weakly realcompact} if, for any $\phi:\mathcal{C}^\infty_{lfcs}(X)\rightarrow{\mathbb{R}}$, there exists some unique $x\in X$ such that, for any $f\in \mathcal{C}^\infty_{lfcs}(X)$, $\phi\circ f=f(x)$. 
    \end{enumerate}
\end{defn}

\begin{exmp}
    Suppose that $\kappa$ is a non-measurable cardinal. Then, $\ell^1(\kappa)$ is weakly realcompact by \cite[p. 575]{edgar_measurability_1979}. 
\end{exmp}

We note that $\mathcal{C}^\infty_{lfcs}$ does not define a functor $\mathcal{C}^{\infty,op}_{bb}\rightarrow{\textnormal{Comm}(\textnormal{CBorn}_\mathbb{R}})$. Indeed, it is not always true that $\mathcal{C}^\infty$ preserves locally finite neighbourhoods. We define the following functorial extension which will also allow us to control sums defined by a cardinal which is not necessarily countable. 

\begin{defn}\label{ourl1algebra}
    Given an object $\ell^1(\kappa)\in\mathcal{C}^\infty_{bb}$, denote by $\mathcal{C}^\infty_{\ell^1}(\ell^1(\kappa))$ the smallest subalgebra of $\mathcal{C}^\infty(\ell^1(\kappa))$ which
    \begin{enumerate}
        \item\label{algebra1} Contains $\ell^\infty(\kappa)$, 
        \item\label{algebra2} Is closed under sums $\sum_{i\in I} f_i$ of smooth functions such that, for any bounded set $B\subseteq \ell^1(\kappa)$, $\sum_{i\in I}\sup_{x\in B}|f_i(x)|<\infty$, 
        \item Is closed under iterated derivatives,
        \item\label{algebra4} Satisfies that, for every map $\psi: \ell^1(\mu)\rightarrow{\ell^1(\kappa)}$ in $\mathcal{C}^\infty_{bb}$ and every $f\in\mathcal{C}^\infty_{\ell^1}(\ell^1(\kappa))$, $f\circ\psi\in\mathcal{C}^\infty_{\ell^1}(\ell^1(\mu))$. 
    \end{enumerate}
\end{defn}

We note that $\mathcal{C}^\infty_{lfcs}(\ell^1(\kappa))\subseteq \mathcal{C}^\infty_{\ell^1}(\ell^1(\kappa))\subseteq\mathcal{C}^\infty(\ell^1(\kappa))$ with equality when $\kappa$ is finite. Moreover, the bornologies on $\mathcal{C}^\infty_{lfcs}(\ell^1(\kappa))$ and $\mathcal{C}^\infty_{\ell^1}(\ell^1(\kappa))$ are induced from the bornology on $\mathcal{C}^\infty(\ell^1(\kappa))$. 

\begin{remark}
    We note that if a function $f$ is in $\mathcal{C}^\infty_{\ell^1}(\ell^1(\kappa))$ then it must be bounded on bounded subsets by Condition (\ref{algebra2}). 
\end{remark}

\begin{lem}
    $\mathcal{C}^\infty_{\ell^1}$ defines a functor $\mathcal{C}^\infty_{\ell^1}:\mathcal{C}^{\infty,op}_{bb}\rightarrow{\textnormal{Comm}(\textnormal{CBorn}_\mathbb{R})}$.  
\end{lem}

\begin{proof}
   This easily follows from Condition (\ref{algebra4}). 
\end{proof}
We note that $\mathcal{C}^\infty(\mathbb{R}^n)\hat{\otimes}\mathcal{C}^\infty(\mathbb{R}^m)\simeq \mathcal{C}^\infty(\mathbb{R}^{n+m})$ by \cite[Theorem 51.6]{treves_topological_1967} where here we are taking the complete projective tensor product \cite[Definition 1.83]{meyer_local_2007}. The proof of this exploits the nuclearity of $\mathcal{C}^\infty(\mathbb{R}^n)$ and $\mathcal{C}^\infty(\mathbb{R}^m)$. We note however that $\mathcal{C}^\infty(\ell^1(\kappa))$ is not nuclear for $\kappa$ not finite. However, we can prove the following result for our restricted algebras. 

\begin{lem}\label{tensorl1algebra}
    There is an equivalence in $\textnormal{Comm}(\textnormal{CBorn}_\mathbb{R})$
    \begin{equation*}
        \mathcal{C}^\infty_{\ell^1}(\ell^1(\kappa))\hat{\otimes}\mathcal{C}_{\ell^1}^\infty(\ell^1(\mu))\simeq \mathcal{C}^\infty_{\ell^1}(\ell^1(\kappa)\oplus \ell^1(\mu))
    \end{equation*}
\end{lem}

\begin{proof}
Consider the linear map $\phi: \mathcal{C}^\infty_{\ell^1}(\ell^1(\kappa))\hat{\otimes}\mathcal{C}_{\ell^1}^\infty(\ell^1(\mu))\rightarrow \mathcal{C}^\infty_{\ell^1}(\ell^1(\kappa)\oplus\ell^1(\mu))\simeq \mathcal{C}^\infty_{\ell^1}(\ell^1(\kappa\coprod\mu))$ defined by sending any tensor $f\otimes g$ to the morphism $\phi(f\otimes g):\ell^1(\kappa)\oplus\ell^1(\mu)\rightarrow{\mathbb{R}}$ defined by $\phi(f\otimes g)(x,y)=f(x)g(y)$. Since $\ell^\infty(\kappa)\oplus \ell^\infty(\mu)\simeq \ell^\infty(\kappa\coprod\mu)$, we see that the image contains $\ell^\infty(\kappa\coprod\mu)$. 

Suppose now that we have some sum $\sum_{i\in I}f_i\otimes g_i$ in $\mathcal{C}^\infty_{\ell^1}(\ell^1(\kappa))\hat{\otimes}\mathcal{C}_{\ell^1}^\infty(\ell^1(\mu))$. Then, we see that for any bounded set $B\subseteq\ell^1(\kappa)\oplus \ell^1(\mu)$, there are bounded subsets $B_\kappa\subseteq \ell^1(\kappa)$ and $B_\mu\subseteq \ell^1(\mu)$ such that $B\subseteq B_\kappa\oplus B_\mu$. Then, we see that 
\begin{equation*}
  \sup_{x\in B}|f_i(x)g_i(x)|\leq \sup_{(x_\kappa,x_\mu)\in B_\kappa\oplus B_\mu}|f_i(x_\kappa)g_i(x_\mu)|\leq \sup_{x_\kappa\in B_\kappa}|f_i(x_\kappa)|\sup_{x_\mu\in B_\mu}|g_i(x_\mu)|<\infty
\end{equation*}

Hence, we see that \begin{equation*} \sum_{i\in I}\sup_{x\in B}|f_i(x)g_i(y)|\leq\sum_{i\in I}\sup_{x_\kappa\in B_\kappa}|f_i(x_\kappa)|\sum_{i\in I}\sup_{x_\mu\in B_\mu}|g_i(x_\mu)|<\infty
\end{equation*}Therefore, we deduce that the image satisfies Condition (\ref{algebra2}) of Definition \ref{ourl1algebra}. 

Now, since each $d^n f_i$ is in $\mathcal{C}^\infty_{\ell^1}(\ell^1(\kappa))$ and each $d^n g_i$ is in $\mathcal{C}^\infty_{\ell^1}(\ell^1(\mu))$, then, by the product rule, since each $d^n(f_i\cdot g_i)$ is a finite sum of derivatives of $f_i$ multiplied by derivatives of $g_i$, we can also show that the image of $\phi$ is closed under iterated derivatives. Finally, suppose that we have a bounded smooth map $\psi:\ell^1(\eta)\rightarrow{\ell^1(\kappa)\oplus\ell^1(\mu)}$ and consider the maps $\pi_\kappa\circ \psi:\ell^1(\eta)\rightarrow{\ell^1(\kappa)}$ and $\pi_\mu\circ\psi:\ell^1(\eta)\rightarrow{\ell^1(\mu)}$ which must also lie in $\mathcal{C}^\infty_{bb}$. Therefore, by Condition (\ref{algebra4}), we must have that $f\circ \pi_k\circ\psi\in\mathcal{C}^\infty_{\ell^1}
(\ell^1(\eta))$ and $f\circ \pi_\mu\circ \psi\in\mathcal{C}^\infty_{\ell^1}(\ell^1(\eta))$. Then, we note that $\phi(f\otimes g)\circ \psi=(f\circ\pi_\kappa\circ\psi)\cdot(g\circ\pi_\mu\circ\psi)$ which must also lie in $\mathcal{C}^\infty_{\ell^1}(\ell^1(\eta))$.

Hence, since the image satisfies conditions (\ref{algebra1}) to (\ref{algebra4}) of Definition \ref{ourl1algebra}, then using a minimality argument we must have that the image of $\phi$ is exactly $\mathcal{C}^\infty_{\ell^1}(\ell^1(\kappa)\oplus \ell^1(\mu))$. The map is also injective, and hence $\phi$ is a bijection.

Finally, to see that the map is bounded, suppose that we have some bounded subset $B$ of the tensor product $\mathcal{C}^\infty_{\ell^1}(\ell^1(\kappa))\hat{\otimes}\mathcal{C}_{\ell^1}^\infty(\ell^1(\mu))$. Then, $B$ is contained in the convex hull of some $B_{\ell^1(\kappa)}\otimes B_{\ell^1(\mu)}$ where $B_{\ell^1(\kappa)}$ and $B_{\ell^1(\mu)}$ are bounded subsets in $\mathcal{C}^\infty_{\ell^1}(\ell^1(\kappa))$ and $\mathcal{C}^\infty_{\ell^1}(\ell^1(\mu))$ respectively. We observe that for any smooth curve $c:\mathbb{R}\rightarrow{\ell^1(\kappa)\oplus\ell^1(\mu)}$ and every compact subset $J\subseteq \mathbb{R}$
\begin{equation*}
    \sup_{\phi(\sum_{i\in I} f_i\otimes g_i)\in\phi(B)}\sup_{t\in J}|d^n (\phi(\sum_{i\in I} f_i\otimes g_i)\circ c)(t)|\leq \sup_{\phi(\sum_{i\in I} f_i\otimes g_i)\in \phi(B)}\sum_{i\in I}\sup_{t\in J}| d^n (f_i\circ c)\cdot (g_i\circ c)(t)|<\infty
\end{equation*}using the product rule, the fact that $f_i$ and $g_i$ are both smooth and bounded with bounded derivatives, and that $f_i$ and $g_i$ are in $B_{\ell^1(\kappa)}$ and $B_{\ell^1(\mu)}$ respectively. Hence, $\phi(B)$ is bounded.

\end{proof}

In the following, we denote by $\textnormal{Sym}:\textnormal{CBorn}_\mathbb{R}\rightarrow{\textnormal{Comm}(\textnormal{CBorn}_\mathbb{R})}$ the left adjoint to the natural forgetful functor.

\begin{lem}\phantomsection\label{symalgebrabij}
    Suppose that $\ell^1(\mu)$ and $\ell^1(\kappa)$ are in $\mathcal{C}_{bb}^\infty$. Then, the map 
    \begin{equation*}
        \textnormal{Hom}_{\mathcal{C}^\infty_{bb}}(\ell^1(\mu),\ell^1(\kappa))\rightarrow{\textnormal{Hom}_{\textnormal{Comm}(\textnormal{CBorn}_\mathbb{R})}(\mathcal{C}_{\ell^1}^\infty(\ell^1(\kappa)),\mathcal{C}_{\ell^1}^\infty(\ell^1(\mu)))}
    \end{equation*}is a bijection. 
\end{lem}

\begin{proof}
    Suppose that $x\in\ell^1(\mu)$ and $\phi\in \textnormal{Hom}_{\textnormal{Comm}(\textnormal{CBorn}_\mathbb{R})}(\mathcal{C}_{\ell^1}^\infty(\ell^1(\kappa)),\mathcal{C}_{\ell^1}^\infty(\ell^1(\mu)))$. Then, we note that $ev_x\circ\phi$ is a map in $\textnormal{Hom}_{\textnormal{Comm}(\textnormal{CBorn}_\mathbb{R})}(\mathcal{C}^\infty_{\ell^1}(\ell^1(\kappa)),\mathbb{R})$. Since $\ell^1(\kappa)$ is weakly realcompact, there exists a unique $y\in\ell^1(\kappa)$ such that, for all $f\in\mathcal{C}^\infty_{lfcs}(\ell^1(\kappa))\subseteq \mathcal{C}^\infty_{\ell^1}(\ell^1(\kappa))$, $(ev_x\circ\phi)(f)=f(y)$. 

    We define a function $\psi:\ell^1(\mu)\rightarrow{\ell^1(\kappa)}$ by $\psi(x)=y$. We note that $\textnormal{Sym}(\ell^1(\kappa)^*)\subseteq \mathcal{C}^\infty_{lfcs}(\ell^1(\kappa))$. Then, since 
    \begin{equation*}
        \textnormal{Hom}_{\textnormal{Comm}(\textnormal{CBorn}_\mathbb{R})}(\textnormal{Sym}(\ell^\infty(\kappa)), \mathcal{C}^\infty_{\ell^1}(\ell^1(\mu))\simeq \textnormal{Hom}_{\textnormal{CBorn}_\mathbb{R}}(\ell^\infty(\kappa),\mathcal{C}^\infty_{\ell^1}(\ell^1(\mu)))
    \end{equation*}we see that, for all $f\in(\ell^1(\kappa))^*\simeq \ell^\infty(\kappa)$, the image of $\phi$, $\overline{\phi}$, in $\textnormal{CBorn}_\mathbb{R}$ satisfies that  $(ev_{-}\circ{\overline{\phi}})(f)=f\circ \psi$. Hence, since the map $(ev_{-}\circ{\overline{\phi}})(f)$ is bounded linear and we are working with convenient vector spaces, $\psi$ can easily be seen to be smooth by Definition \ref{convenientspacedefn}. It is bounded by an application of the uniform boundedness principle for Banach spaces. 
\end{proof}

\subsection{Derived $\mathcal{C}^\infty$-Bornological Rings}

If $\mathcal{C}^\infty_{\ell^1}(\ell^1(\kappa))$ were compact as a complete bornological algebra, then we could use Lemma \ref{symalgebrabij} to show that $\mathcal{C}^\infty_{\ell^1}$ is fully faithful. However, $\mathcal{C}^\infty_{\ell^1}(\ell^1(\kappa))$ is not necessarily compact and we shouldn't expect $\mathcal{C}^\infty_{\ell^1}$ to be fully faithful. Indeed, there are several examples of $\mathcal{C}^\infty$-rings which are not complete as bornological spaces, for example certain rings of germs of smooth functions. However, as in \cite[Proposition 27]{borisov_quasi-coherent_2017} and \cite[Lemma 5.3.91]{ben-bassat_perspective_2024}, we expect to have a fully faithful embedding in the derived setting. 
\begin{defn}The \textit{$(\infty,1)$-category of derived $\mathcal{C}^\infty$-bornological rings}, which we will denote by $\mathcal{C}^\infty\textcat{DBornRing}$, is the $(\infty,1)$-category
\begin{equation*}
    \mathcal{C}^\infty\textcat{DBornRing}:=\mathcal{P}_\Sigma(\mathcal{C}_{bb}^{\infty,op})=\textcat{Fun}^\times(\mathcal{C}_{bb}^\infty,\infty\textcat{Grpd})
\end{equation*}We define the $(\infty,1)$-category of \textit{derived $\mathcal{C}^\infty$-bornological affines} by \begin{equation*}
\mathcal{C}^\infty\textcat{DBAff}:=\mathcal{C}^\infty\textcat{DBornRing}^{op}
\end{equation*}The \textit{ordinary category of $\mathcal{C}^\infty$-bornological affines} is $\mathcal{C}^\infty\textnormal{BAff}:=\mathcal{C}^\infty\textnormal{BornRing}^{op}$.  
\end{defn}

We can similarly describe the \textit{$(\infty,1)$-category of derived $\mathcal{C}^\infty$-rings} to be the $(\infty,1)$-category $\mathcal{C}^\infty\textcat{DRing}:=\mathcal{P}_\Sigma(\textnormal{CartSp}^{op})$. Similarly, the $(\infty,1)$-category of \textit{derived $\mathcal{C}^\infty$-affines} is defined by $\mathcal{C}^\infty\textcat{DAff}:=\mathcal{C}^\infty\textcat{DRing}^{op}$. By Lemma \ref{fullyfaithfulcinf}, together with \cite[Proposition 5.5.8.22]{lurie_higher_2009} and Proposition \ref{siftedcocompletion}, there exists a fully faithful inclusion functor $\mathcal{C}^\infty\textcat{DAff}\hookrightarrow{\mathcal{C}^\infty\textcat{DBAff}}$. 

\begin{remark}
    We note that derived manifolds in the sense of Carchedi and Steffens \cite[Corollary 1.1.]{carchedi_universal_2019} are a full subcategory of $\mathcal{C}^\infty\textcat{DAff}$, and hence of $\mathcal{C}^\infty\textcat{DBAff}$. 
\end{remark}

We note that, by \cite[Lemma 5.3.91]{ben-bassat_perspective_2024}, there is a fully faithful functor 
\begin{equation*}
    \mathcal{C}^\infty_{\ell^1}=\mathcal{C}^\infty:\mathcal{C}^\infty\textcat{DRing}\rightarrow{\textcat{DAlg}^{cn}(\textcat{Ch}(\textnormal{CBorn}_\mathbb{R}))}
\end{equation*}We can extend the functor $\mathcal{C}^\infty_{\ell^1}:\mathcal{C}^\infty\textnormal{BornRing}\rightarrow{\textnormal{Comm}(\textnormal{CBorn}_\mathbb{R})}$ to a functor
\begin{equation*}
\mathcal{C}^\infty_{\ell^1}:\mathcal{C}^\infty\textcat{DBornRing}\rightarrow{\textcat{DAlg}^{cn}(\textcat{Ch}(\textnormal{CBorn}_\mathbb{R}))}
\end{equation*}by Proposition \ref{siftedcocompletion}. It remains to prove that this functor is fully faithful. As in the proof of \cite[Lemma 5.3.91]{ben-bassat_perspective_2024}, we use the factorisation through $\textnormal{Sym}(\ell^1(\kappa))$ which is a compact object in $\textnormal{Comm}(\textnormal{CBorn}_\mathbb{R})$.

\begin{lem}\label{compositionequivsym}
    There is an equivalence 
    \begin{equation*}
        \textnormal{Hom}_{\mathcal{C}^\infty_{bb}}(\ell^1(\mu),\ell^1(\kappa))\rightarrow{\textnormal{Hom}}_{\textnormal{Comm}(\textnormal{CBorn}_\mathbb{R})}(\textnormal{Sym}(\ell^1(\kappa)),\mathcal{C}^\infty_{\ell^1}(\ell^1(\mu)))
    \end{equation*}
\end{lem}

\begin{proof}Suppose that we have a smooth morphism $\psi:\ell^1(\mu)\rightarrow{\ell^1(\kappa)}$ which is bounded on bounded subsets. Then, we can define a bounded linear algebra morphism $\textnormal{Sym}(\ell^1(\kappa))\rightarrow{\mathcal{C}^\infty_{\ell^1}(\ell^1(\mu))}$ by considering $\mathcal{C}^\infty_{\ell^1}(\psi)$ and restricting to $\textnormal{Sym}(\ell^1(\kappa))$.

Conversely, given a morphism $\phi:\textnormal{Sym}(\ell^1(\kappa))\rightarrow{\mathcal{C}^\infty_{\ell^1}(\ell^1(\mu))}$, we consider the morphism $\overline{\phi}:\ell^1(\kappa)\rightarrow{\mathcal{C}^\infty_{\ell^1}(\ell^1(\mu))}$ in $\textnormal{CBorn}_\mathbb{R}$ induced by the universal property of $\textnormal{Sym}$. Define a morphism $\psi:\ell^1(\mu)\rightarrow{\ell^1(\kappa)}$ by $\psi(x)=(\overline{\phi}(e_k)(x))_{k\in\kappa}$. We note that this is well defined since the family $\{\overline{\phi}(e_k):k\in \kappa\}$ lies in $\mathcal{C}^\infty_{\ell^1}(\ell^1(\kappa))$, and hence we see that $\sum_{k\in\kappa}|\overline{\phi}(e_k)(x)|<\infty$ by Condition (\ref{algebra2}). 

Now, it remains to show that $\psi$ is smooth and bounded on bounded subsets. Suppose that $B$ is a bounded subset of $\ell^1(\mu)$. Then, we easily note that $\sup_{x\in B}|\psi(x)|=\sup_{x\in B}\sum_{k\in\kappa}|\overline{\phi}(e_k)(x)|<\infty$ and therefore $\psi$ is bounded. Now to show that $\psi$ is smooth, we note that by \cite[Theorem 5.20]{kriegl_convenient_1997}, it suffices to show that all iterated derivatives $d^n\psi$ exist and are bounded. Consider the family $\{d^n\overline{\phi}(e_k)\}_{k\in K}$. Each function $d^n\overline{\phi}(e_k)$ must lie in $\mathcal{C}^\infty_{\ell^1}(\ell^1(\mu))$ since the algebra is closed under iterated derivatives. Hence, by the conditions on our algebra, we must have that for each $n\geq 0$ and each bounded set $B\subseteq \ell^1(\mu)$,
\begin{equation*}
    \sum_{k\in\kappa}\sup_{x\in B}|(d^n\overline{\phi}(e_k))(x)|<\infty
\end{equation*}Therefore for each $n\geq 0$, the derivative $d^n\psi$ is bounded. 
\end{proof}

Suppose that $\ell^1(\kappa)\in\mathcal{C}^\infty_{bb}$. Consider the canonical chain of inclusion maps $\ell^1(\kappa)\hookrightarrow{\ell^\infty(\kappa)}\hookrightarrow{\mathcal{C}^\infty_{\ell^1}(\kappa)}$. This map is linear and bounded and hence induces a map $\Phi:\textnormal{Sym}(\ell^1(\kappa))\rightarrow{\mathcal{C}^\infty_{\ell^1}(\ell^1(\kappa))}$ in $\textnormal{Comm}(\textnormal{CBorn}_\mathbb{R})$.

\begin{lem}\label{symhomotopyepil1}
    $\Phi:\textnormal{Sym}(\ell^1(\kappa))\rightarrow{\mathcal{C}^\infty_{\ell^1}(\ell^1(\kappa))}$ is a homotopy epimorphism in $\textnormal{Comm}(\textnormal{CBorn}_\mathbb{R})$, in the sense of Definition \ref{homotopyepi}.
\end{lem}

\begin{proof}It suffices to show that $\mathcal{C}^\infty_{\ell^1}(\ell^1(\kappa))\otimes^\mathbb{L}_{\textnormal{Sym}(\ell^1(\kappa))}\mathcal{C}^\infty_{\ell^1}(\ell^1(\kappa))\simeq \mathcal{C}^\infty_{\ell^1}(\ell^1(\kappa))$. By \cite[Example 7.13]{savage_koszul_2023}, the algebra $\textnormal{Sym}(\ell^1(\kappa))$ has a Koszul complex given by
\begin{equation*}
        \dots\rightarrow{\textnormal{Sym}(\ell^1(\kappa))\hat{\otimes} \bigwedge^2 \ell^1(\kappa)}\rightarrow{\textnormal{Sym}(\ell^1(\kappa))\hat{\otimes} \ell^1(\kappa)}\rightarrow{\textnormal{Sym}(\ell^1(\kappa))}
    \end{equation*}with differential
    \begin{equation*}
        \partial_i':\textnormal{Sym}(\ell^1(\kappa))\hat{\otimes}\bigwedge^i\ell^1(\kappa)\rightarrow{\textnormal{Sym}(\ell^1(\kappa))\hat{\otimes}\bigwedge^{i-1}\ell^1(\kappa)}
    \end{equation*}given by sending any $x\hat{\otimes}v_1\wedge\dots\wedge v_i$ in $\textnormal{Sym}(\ell^1(\kappa))\hat{\otimes}\bigwedge^i\ell^1(\kappa)$ to $\sum_{j=1}^{i}(-1)^{j-1}xv_i\hat{\otimes}v_1\wedge\dots\wedge\hat{v_j}\wedge\dots \wedge v_i$. By Lemma \ref{tensorl1algebra}, we have that $\mathcal{C}^\infty_{\ell^1}(\ell^1(\kappa))\hat{\otimes}\mathcal{C}^\infty_{\ell^1}(\ell^1(\kappa))\simeq \mathcal{C}^\infty_{\ell^1}(\ell^1(\kappa\coprod\kappa))$ and hence, to show that $\Phi$ is a homotopy epimorphism, it suffices to show that the complex
    \begin{equation*}
        \dots\rightarrow{\mathcal{C}^\infty_{\ell^1}(\ell^1(\kappa\coprod\kappa))\hat{\otimes} \bigwedge^2 \ell^1(\kappa)}\rightarrow{\mathcal{C}^\infty_{\ell^1}(\ell^1(\kappa\coprod\kappa))\hat{\otimes} \ell^1(\kappa)}\rightarrow{\mathcal{C}^\infty_{\ell^1}(\ell^1(\kappa\coprod\kappa))
        }
    \end{equation*}is strictly exact in the sense of \cite[Definition 1.1.9]{schneiders_quasi-abelian_1999}. We note that the differential of the above complex is given by the map 
    \begin{equation*}
        \partial_i: \mathcal{C}^\infty_{\ell^1}(\ell^1(\kappa\coprod\kappa))\hat{\otimes} \bigwedge^i\ell^1(\kappa)\rightarrow \mathcal{C}^\infty_{\ell^1}(\ell^1(\kappa\coprod\kappa))\hat{\otimes} \bigwedge^{i-1}\ell^1(\kappa)
    \end{equation*}which sends any $f\hat{\otimes}v_1\wedge\dots\wedge v_i$ to $\sum_{j=1}^i (-1)^{j-1}f\cdot \widetilde{\Phi}(v_j)\hat{\otimes} v_1\wedge\dots\wedge \hat{v_j}\wedge \dots\wedge v_i$, where the map $\widetilde{\Phi}(v_j):\ell^1(\kappa\coprod\kappa)\simeq\ell^1(\kappa)\oplus\ell^1(\kappa)\rightarrow{\mathbb{R}}$ is defined by $\widetilde{\Phi}(v_j)(x,y)=\Phi(v_j)(x)-\Phi(v_j)(y)$. We note that $\widetilde{\Phi}(v_j)$ is clearly in $\mathcal{C}^\infty_{\ell^1}(\ell^1(\kappa\coprod\kappa))$ and that $\partial_i$ is bounded as a map of complete bornological algebras because $\widetilde{\Phi}(v_j)$ is a bounded map. We want to show that the complex is homotopy equivalent to the complex $0\rightarrow{\mathcal{C}^\infty_{\ell^1}(\ell^1(\kappa))}\rightarrow{0}$.

    Motivated by the finite dimensional case, discussed in \cite[Lemma 5.3.91]{ben-bassat_perspective_2024}, we define the map
    \begin{equation*}
        h_i:\mathcal{C}^\infty_{\ell^1}(\ell^1(\kappa\coprod\kappa))\hat{\otimes} \bigwedge^i\ell^1(\kappa)\rightarrow \mathcal{C}^\infty_{\ell^1}(\ell^1(\kappa\coprod\kappa))\hat{\otimes} \bigwedge^{i+1}\ell^1(\kappa)
    \end{equation*}which sends any $f\hat{\otimes} v_1\wedge\dots\wedge v_i$ to $\sum_{k\in\kappa}Df(e_k)\hat{\otimes} e_k\wedge v_1\wedge \dots v_i$ where $e_k$, $k\in\kappa$, are the standard basis elements of $\ell^1(\kappa)$. The map $Df(e_k)$ is defined such that \begin{equation*}
        Df(e_k)(x,y)=\int_0^1 d^1_{e_k}f(y+t(x-y),y) dt
    \end{equation*}where $d^1_{e_k}$ is the directional derivative in the direction $e_k$, see \cite[Theorem 5.20]{kriegl_convenient_1997}. We note that any integral on $[0,1]$ can be realised as the limit of a Riemann sum by \cite[Proposition 2.7]{kriegl_convenient_1997}. Hence, since derivatives and bounded sums are contained in $\mathcal{C}^\infty_{\ell^1}(\ell^1(\kappa\coprod\kappa))$, then $Df(e_k)\in\mathcal{C}^\infty_{\ell^1}(\ell^1(\kappa\coprod\kappa))$ and the sum over $\kappa$ also defines an object in $\mathcal{C}^\infty_{\ell^1}(\ell^1(\kappa\coprod\kappa))$. Moreover, we can see that $h_i$ is a map of complete bornological algebras.

    For each $i\geq 1$, \begin{align*}
       &(h_{i-1}\circ\partial_i+\partial_{i+1}\circ h_i)(f\hat{\otimes}v_1\wedge\dots\wedge v_i)\\
       &=h_{i-1}(\sum_{j=1}^i (-1)^{j-1}f\cdot \widetilde{\Phi}(v_j)\hat{\otimes} v_1\wedge\dots\wedge \hat{v_j}\wedge \dots\wedge v_i)+\partial_{i+1}(\sum_{k\in\kappa}D{f}(e_k)\hat{\otimes} e_k\wedge v_1\wedge \dots \wedge v_i)\\
       &=\sum_{j=1}^i(-1)^{j-1}\sum_{k\in\kappa} D{(f\cdot\widetilde{\Phi}(v_j))}(e_k)\hat{\otimes}e_k\wedge v_1\wedge\dots\wedge\hat{v_j}\wedge\dots\wedge v_i+\sum_{k\in\kappa}D{f}(e_k)\cdot\widetilde{\Phi}(e_k)\hat{\otimes} v_1\wedge\dots \wedge v_i\\
       &+\sum_{k\in\kappa}\sum_{j=1}^i(-1)^j Df(e_k)\cdot\widetilde{\Phi}(v_j)\hat{\otimes}e_k\wedge v_1\wedge \dots \wedge \hat{v_j}
\wedge\dots v_i\\
\intertext{We note that $D(f\cdot\widetilde{\Phi}(v_j))(e_k)=Df(e_k)\cdot\widetilde{\Phi}(v_j)$ and hence, cancelling,}
&=\sum_{k\in\kappa}Df(e_k)\cdot \widetilde{\Phi}(e_k)\hat{\otimes}v_1\wedge\dots \wedge v_i
\end{align*}
By the fundamental theorem of calculus \cite[c.f. Section 13.3]{kriegl_convenient_1997}, 
    \begin{align*}
        f(x,y)-f(y,y)&=\int_0^1 (x-y)d_t^1 f(y+t(x-y),y) dt\\&=\int_0^1 \sum_{k\in\kappa}(x_k-y_k)d_{e_k}^1f(y+t(x-y),y) dt\\
        \intertext{Since the sum is convergent by our assumptions on $\mathcal{C}^\infty_{\ell^1}(\ell^1(\kappa\coprod\kappa))$,}
    &\simeq \sum_{k\in\kappa}\int_0^1 (x_k-y_k)d_{e_k}^1f(y+t(x-y),y) dt\\
    &\simeq \sum_{k\in\kappa}Df(e_k)(x,y)\cdot\widetilde{\Phi}(e_k)(x,y)
    \end{align*}
Hence, we see that 
\begin{equation*}
    (h_{i-1}\circ\partial_i+\partial_{i+1}\circ h_i)(f\hat{\otimes}v_1\wedge\dots\wedge v_i)=f\hat{\otimes} v_1\wedge\dots\wedge v_i-f'\hat{\otimes}v_1\wedge\dots \wedge v_i
\end{equation*}where $f'(x,y)=f(y,y)$. Therefore, $h_{i-1}\circ\partial_i+\partial_{i+1}\circ h_i=id-p_i$ where $p_i$ is the map sending $f\hat{\otimes}v_1\wedge\dots\wedge v_i$ to $f'\hat{\otimes}v_1\wedge\dots\wedge v_i$.

Consider the map $q_\bullet$ from our complex $\mathcal{C}^\infty_{\ell^1}(\ell^1(\kappa\coprod\kappa))\hat{\otimes}\bigwedge^\bullet\ell^1(\kappa)$ to the complex $0\rightarrow\mathcal{C}^\infty_{\ell^1}(\ell^1(\kappa))\rightarrow{0}$ given by $q_i=0$ for $i>0$ and given in degree $0$ by the map  $q_0(f)$ defined by $q_0(f)(x)=f(x,x)$ for $f\in\mathcal{C}^\infty_{\ell^1}(\ell^1(\kappa\coprod\kappa))$. We also have a map $r_\bullet$ from $0\rightarrow\mathcal{C}^\infty_{\ell^1}(\ell^1(\kappa))\rightarrow{0}$ to $\mathcal{C}^\infty_{\ell^1}(\ell^1(\kappa\coprod\kappa))\hat{\otimes}\bigwedge^\bullet\ell^1(\kappa)$ given by $r_i=0$ for $i>0$ and $r_0(f)(x,y)=f(y)$ for $f\in\mathcal{C}^\infty(\ell^1(\kappa))$. 

Since $h_{i-1}\circ\partial_i+\partial_{i+1}\circ h_i=id-p_i$ where $p_i=r_i\circ q_i$, then our complex is homotopy equivalent to the complex $0\rightarrow{\mathcal{C}^\infty_{\ell^1}(\ell^1(\kappa))}\rightarrow{0}$ as desired.

\end{proof}

\begin{thm}\label{fullyfaithfulembed}
The functor $\mathcal{C}^\infty_{\ell^1}:\mathcal{C}^\infty\textcat{DBornRing}\rightarrow{\textcat{DAlg}^{cn}(\textcat{Ch}(\textnormal{CBorn}_\mathbb{R}))}$ is fully faithful. 
\end{thm}

\begin{proof}
    Suppose that $A,B\in\mathcal{C}^\infty\textcat{DBornRing}$. We want to show that the canonical map 
    \begin{equation*}
        \textnormal{Map}_{\mathcal{C}^\infty\textcat{DBornRing}}(A,B)\rightarrow{\textnormal{Map}_{\textcat{DAlg}^{cn}(\textcat{Ch}(\textnormal{CBorn}_\mathbb{R}))}}(\mathcal{C}^\infty_{\ell^1}(A),\mathcal{C}^\infty_{\ell^1}(B))
    \end{equation*}is an equivalence. We can immediately reduce to the case when $A$ is of the form $\ell^1(\kappa)$ for some cardinal $\kappa$. Then, writing $B$ as a formal sifted $(\infty,1)$-colimit $``\varinjlim_i"\ell^1(\mu_i)$, we consider the composite
    \begin{equation}\phantomsection\label{equationcomposite}\begin{aligned}
       &\textnormal{Map}_{\mathcal{C}^\infty\textcat{DBornRing}}(\ell^1(\kappa),``\varinjlim_i"\ell^1(\mu_i))\\
       &\xrightarrow{\mathcal{C}^\infty_{\ell^1}}{\textnormal{Map}_{\textcat{DAlg}^{cn}(\textcat{Ch}(\textnormal{CBorn}_\mathbb{R}))}(\mathcal{C}_{\ell^1}^\infty}(\ell^1(\kappa)),\varinjlim_i{\mathcal{C}_{\ell^1}^\infty}(\ell^1(\mu_i)))\\
        & \xrightarrow{-\circ \Phi}{\textnormal{Map}_{\textcat{DAlg}^{cn}(\textcat{Ch}(\textnormal{CBorn}_\mathbb{R}))}(\textnormal{Sym}(\ell^1(\kappa)),\varinjlim_i {\mathcal{C}_{\ell^1}^\infty}(\ell^1(\mu_i)))}
    \end{aligned}\end{equation}By Lemma \ref{compositionequivsym}, for each $i$, the map
    \begin{equation*}
        \textnormal{Hom}_{\mathcal{C}^\infty_b}(\ell^1(\mu_i),\ell^1(\kappa))\rightarrow{\textnormal{Hom}_{\textnormal{Comm}(\textnormal{CBorn}_\mathbb{R})}(\textnormal{Sym}(\ell^1(\kappa)), {\mathcal{C}_{\ell^1}^\infty}(\ell^1(\mu_i)))}
    \end{equation*}is an equivalence. Now, since $\textnormal{Sym}(\ell^1(\kappa))$ is compact projective in $\textnormal{Comm}(\textnormal{CBorn}_\mathbb{R})$, the composition of maps in Equation (\ref{equationcomposite}) is an equivalence.

Since the morphism $\Phi:\textnormal{Sym}(\ell^1(\kappa))\rightarrow{{\mathcal{C}_{\ell^1}^\infty}(\ell^1(\kappa))}$ is a homotopy epimorphism by Lemma \ref{symhomotopyepil1}, the morphism 
    \begin{equation*}\phantomsection\label{mappingsymepimorphism}
        \textnormal{Map}(\mathcal{C}_{\ell^1}^{\infty}(\ell^1(\kappa)),\varinjlim_i\mathcal{C}_{\ell^1}^{\infty}(\ell^1(\mu_i)))
        \rightarrow{\textnormal{Map}(\textnormal{Sym}(\ell^1(\kappa)),\varinjlim_i \mathcal{C}_{\ell^1}^{\infty}(\ell^1(\mu_i)))}
    \end{equation*}is a monomorphism on $\pi_0$ and an equivalence on $\pi_n$ for any $n$. Since the composition in Equation (\ref{equationcomposite}) is an equivalence, we note that the above equation is also an epimorphism on $\pi_0$, and hence is an equivalence. We then easily deduce that the first map in Equation (\ref{equationcomposite}) is an equivalence, as desired. 
\end{proof}

\subsection{The Free $\mathcal{C}^\infty$-Bornological Ring}

Recall that there exists a functor $\iota:\textnormal{Lin}\rightarrow{\mathcal{C}^\infty_{bb}}$ and that, by Lemma \ref{iotapresevfiniteprod},  $\iota^{op}$ preserves finite coproducts. Therefore, we can use Lemma \ref{siftedcocompletionproperties} and Lemma \ref{bornbicomplete} to extend $\iota^{op}$ to a colimit preserving functor 
\begin{equation*}
    \tilde{L}:\textnormal{SInd}(\textnormal{Lin}^{op})\rightarrow{\textnormal{SInd}(\mathcal{C}_{bb}^{\infty,op})\simeq \mathcal{C}^\infty\textnormal{BornRing}}
\end{equation*}We note that, since $\textnormal{Lin}$ has finite products, there is a finite product preserving functor $\textnormal{Lin}\rightarrow{\textnormal{SInd}(\textnormal{Lin}^{op})}^{op}$. Since sifted colimits commute with finite products, we can left Kan extend to a finite product preserving functor $\textnormal{SInd}(\textnormal{Lin})\rightarrow{\textnormal{SInd}(\textnormal{Lin}^{op})}^{op}$. Hence, taking the opposite functor and composing it with $\tilde{L}$ we obtain a finite coproduct preserving functor \begin{equation}\label{cborncoproductfunctor}
    \textnormal{LH}(\textnormal{CBorn}_\mathbb{R})^{op}\simeq \textnormal{SInd}(\textnormal{Lin})^{op}\rightarrow{\mathcal{C}^\infty\textnormal{BornRing}}
\end{equation}

\begin{thm}\phantomsection\label{colimitpreservingcinfinity}
    There is a colimit-preserving functor
    \begin{equation*}
        L:\textnormal{LH}(\textnormal{CBorn}_\mathbb{R})\rightarrow{\mathcal{C}^\infty\textnormal{BornRing}}
    \end{equation*}which acts on elements of $\textnormal{Lin}$ by $L(\ell^1(\kappa))=\ell^1(\kappa)^\vee:=\texthom{Hom}_{\textnormal{CBorn}_\mathbb{R}}(\ell^1(\kappa),\mathbb{R})\simeq\ell^\infty(\kappa)$. 
\end{thm}

\begin{proof}
    We note that, since $\textnormal{CBorn}_\mathbb{R}$ is closed, there is a colimit-preserving functor $\textnormal{Lin}\xrightarrow{(-)^\vee}{\textnormal{CBorn}_\mathbb{R}^{op}}$ obtained by considering each element of $\textnormal{Lin}$ as a complete bornological space and then taking the dual in $\textnormal{CBorn}_\mathbb{R}$. Therefore, since there is a finite coproduct preserving functor $\textnormal{CBorn}_\mathbb{R}^{op}\hookrightarrow{\textnormal{LH}(\textnormal{CBorn}_\mathbb{R})^{op}}$ by \cite[Corollary 1.4.7]{schneiders_quasi-abelian_1999}, we can compose with the functor from Equation (\ref{cborncoproductfunctor}) to obtain a finite coproduct preserving functor
    \begin{equation*}
        \textnormal{Lin}\rightarrow{\textnormal{CBorn}_\mathbb{R}^{op}}\rightarrow{\textnormal{LH}(\textnormal{CBorn}_\mathbb{R})^{op}}\rightarrow{\textnormal{SInd}(\textnormal{Lin}^{op})}\rightarrow{\mathcal{C}^\infty\textnormal{BornRing}}
    \end{equation*}By Lemma \ref{siftedcocompletionproperties}, we can extend this functor to a colimit-preserving functor 
    \begin{equation*}
        L:\textnormal{LH}(\textnormal{CBorn}_\mathbb{R})\rightarrow{\mathcal{C}^\infty\textnormal{BornRing}}
    \end{equation*}
\end{proof}

By the adjoint functor theorem, we obtain the following result. This shows that we can construct a \textit{free $\mathcal{C}^\infty$-bornological ring} from any complete bornological space. 
\begin{cor}\phantomsection\label{importantadjunction}
    There is an adjunction $
        L:\textnormal{LH}(\textnormal{CBorn}_\mathbb{R})\leftrightarrows{\mathcal{C}^\infty\textnormal{BornRing}}:R$. 
\end{cor}

By the adjoint functor theorem, $R$ is defined by 
\begin{equation*}\begin{aligned}
    R:\mathcal{C}^\infty\textnormal{BornRing}&\rightarrow{\textnormal{LH}(\textnormal{CBorn}_\mathbb{R})\simeq \textnormal{SInd}(\textnormal{Lin})}\\
    A&\rightarrow{\textnormal{Hom}_{\mathcal{C}^\infty\textnormal{BornRing}}(L(-),A)}
\end{aligned}\end{equation*}identified as a product preserving functor $\textnormal{Lin}^{op}\rightarrow{\textnormal{Set}}$.

\begin{prop}\phantomsection\label{Rmonoidalstructure}
    $R$ preserves the monoidal structure when restricted to $\mathcal{C}^\infty\textnormal{Ring}$. 
\end{prop}

\begin{proof}
    Suppose that we have some element $\ell^1(\kappa)\in\textnormal{Lin}$ and consider the dual object $\ell^1(\kappa)^\vee\in\textnormal{CBorn}_\mathbb{R}^{op}$. Since every object in $\textnormal{CBorn}_\mathbb{R}$ can be written as a filtered colimit of objects in $\textnormal{Lin}$, we can identify $\ell^1(\kappa)^\vee$ with a cofiltered limit $\varprojlim_i\ell^1(\kappa_i)$ in $\textnormal{CBorn}_\mathbb{R}^{op}$ and identify with a formal sifted colimit $``\varinjlim_i"\ell^1(\kappa_i)$ in $\mathcal{C}^\infty\textnormal{BornRing}$ using Appendix \ref{indobjectappendix}. Since $A\in\mathcal{C}^\infty\textnormal{Ring}$, $A$ can be written as a formal sifted colimit $``\varinjlim_j"\mathbb{R}^{n_j}$. Then, we note that 
    \begin{align*}
        R(A)(\ell^1(\kappa))&\simeq \textnormal{Hom}_{\mathcal{C}^\infty\textnormal{BornRing}}(L(\ell^1(\kappa)),A)\\
        &\simeq \textnormal{Hom}_{\mathcal{C}^\infty\textnormal{BornRing}}(``\varinjlim_i"\ell^1(\kappa_i),``\varinjlim_j"\mathbb{R}^{n_j})\\
        &\simeq \varprojlim_i\varinjlim_j\textnormal{Hom}_{\mathcal{C}^\infty_{bb}}(\mathbb{R}^{n_j},\ell^1(\kappa_i))
        \intertext{Since smooth maps $\mathbb{R}^{n_j}\rightarrow{\ell^1(\kappa_i)}$ are bounded on bounded subsets,}
        &\simeq \varprojlim_i\varinjlim_j\textnormal{Hom}_{\textnormal{Smooth}}(\mathbb{R}^{n_j},\ell^1(\kappa_i))
\intertext{Using Theorem \ref{freeconvenientspace},}
        &\simeq \varprojlim_i\varinjlim_j\textnormal{Hom}_{\textnormal{Con}}(\lambda\mathbb{R}^{n_j},\ell^1(\kappa_i))\\
\intertext{Since $\lambda \mathbb{R}^{n_j}$ is equivalent to $\mathcal{C}^\infty(\mathbb{R}^{n_j})^\vee=\texthom{Hom}_{\textnormal{Con}}(\mathcal{C}^\infty(\mathbb{R}^{n_j}),\mathbb{R})$ by \cite[Corollary 23.11]{kriegl_convenient_1997},}
&\simeq \varprojlim_i\varinjlim_j\textnormal{Hom}_{\textnormal{Con}}(\mathcal{C}^\infty(\mathbb{R}^{n_j})^\vee,\ell^1(\kappa_i))\\
\intertext{Since $\mathcal{C}^\infty(\mathbb{R}^{n_j})$ is a nuclear and reflexive convenient space by \cite[Proposition 6.1 and Results 6.5]{kriegl_convenient_1997}, then using Lemma \ref{strongdualisablenucref},} 
&\simeq \varprojlim_i\varinjlim_j\textnormal{Hom}_{\textnormal{Con}}(\mathbb{R},\mathcal{C}^\infty(\mathbb{R}^{n_j})\hat{\otimes}\ell^1(\kappa_i))\\
&\simeq \textnormal{Hom}_{\textnormal{Con}}(\mathbb{R},\varprojlim_i\varinjlim_j\mathcal{C}^\infty(\mathbb{R}^{n_j})\hat{\otimes}\ell^1(\kappa_i))\\
\intertext{Since $\mathcal{C}^\infty(\mathbb{R}^{n_j})$ is nuclear, then by \cite[Lemma 5.18]{ben-bassat_frechet_2023},}&\simeq \textnormal{Hom}_{\textnormal{Con}}(\mathbb{R},(\varinjlim_j\mathcal{C}^\infty(\mathbb{R}^{n_j}))\hat{\otimes}\varprojlim_i\ell^1(\kappa_i))\\
\intertext{And, by definition of the functor $\mathcal{C}_{\ell^1}^\infty:\mathcal{C}^\infty\textnormal{BornRing}\rightarrow{\textnormal{Comm}(\textnormal{CBorn}_\mathbb{R})}$,}
&\simeq \textnormal{Hom}_{\textnormal{Con}}(\mathbb{R},\mathcal{C}^\infty(A)\hat{\otimes}\ell^\infty(\kappa_i))
    \end{align*}Suppose that $B\in\mathcal{C}^\infty\textnormal{Ring}$. Since $\mathcal{C}^\infty(A\otimes B)\simeq \mathcal{C}^\infty(A)\hat{\otimes} \,\mathcal{C}^\infty(B)$ for $\mathcal{C}^\infty$-rings by \cite[Theorem 51.6]{treves_topological_1967} and by the definition of the monoidal product in $\textnormal{SInd}(\textnormal{CartSp}^{op})$, we can use our above characterisation of the right adjoint to show that $R(A\otimes B)\simeq R(A)\hat\otimes R(B)$. 
\end{proof}
We recall that $\mathcal{C}^\infty(\mathbb{R}^n)$ is a nuclear complete bornological space for any $n\in\mathbb{N}$. Hence, since sifted colimits of nuclear complete bornological spaces are nuclear by \cite[Proposition 3.52]{bambozzi_stein_2018}, we see that any $\mathcal{C}^\infty$-ring $A$ gets mapped to a nuclear complete bornological algebra $\mathcal{C}^\infty(A)$. 

\begin{lem}\label{reflexivereduced} Suppose that $A,B\in\mathcal{C}^\infty\textnormal{Ring}$ and that $\mathcal{C}^\infty(A)$ is reflexive. Then, \begin{equation*}
    R(A)\hat{\otimes}R(B)\simeq \texthom{Hom}_{\textnormal{LH}(\textnormal{CBorn}_\mathbb{R})}(R(A)^\vee,R(B))
\end{equation*}
\end{lem}
\begin{proof}Indeed, by Proposition \ref{Rmonoidalstructure}, and using fully faithfulness of the inclusion functor $\textnormal{CBorn}_\mathbb{R}\hookrightarrow{\textnormal{LH}(\textnormal{CBorn}_\mathbb{R})}$, we have that
\begin{equation*}
    \texthom{Hom}_{\textnormal{LH}(\textnormal{CBorn}_\mathbb{R})}(R(A)^\vee,R(B))\simeq \texthom{Hom}_{\textnormal{CBorn}_\mathbb{R}}(\mathcal{C}^\infty(A)^\vee,\mathcal{C}^\infty(B))
\end{equation*}
Now, by Corollary \ref{nuclearhomequivCinfinity}, since $\mathcal{C}^\infty(A)$ is reflexive and nuclear, and $\mathcal{C}^\infty(B)$ is a Fr\'echet space, then this is equivalent to $\mathcal{C}^\infty(A)\hat{\otimes}\mathcal{C}^\infty(B)$ from which our result follows.

\end{proof}

\section{Derived Moduli Stacks of Solutions to Non-Linear PDEs}\label{section3}

As explained in the Introduction, the derived moduli stack of solutions to a system of partial differential equations can be expressed as a mapping stack of sections over the de Rham stack. In order to apply the results from Section \ref{representabilitytheoremsection}, we have to embed our derived $\mathcal{C}^\infty$-bornological affines into a suitable derived geometry context. We will frequently refer to \cite{savage_representability_2024} where the necessary definitions were established.

\subsection{The $\mathcal{C}^\infty$-Localisation Topology}

Using Theorem \ref{fullyfaithfulembed}, we see that we can consider $\mathcal{C}^\infty\textcat{DBAff}$ as a full subcategory of $\textcat{DAff}^{cn}(\textcat{Ch}(\textnormal{CBorn}_\mathbb{R}))$ under the functor $\mathcal{C}^\infty_{\ell^1}$. In the following section, we will only explicitly mention this embedding functor when it is essential for the clarity of the exposition. We will often just be restricting to objects in $\mathcal{C}^\infty\textcat{DBAff}$, but in order to easily define suitable topologies and maps between them we will need to embed them as objects in $\textcat{DAff}^{cn}(\textcat{Ch}(\textnormal{CBorn}_\mathbb{R}))$. 

We recall the following definitions we frequently used in  \cite{savage_representability_2024}. 

\begin{defn}
    Suppose that $f:A\rightarrow{B}$ is in $\textcat{DAlg}^{cn}(\textcat{Ch}(\textnormal{CBorn}_\mathbb{R}))$. Then, 
    \begin{enumerate}
        \item $f$ is \textit{flat} if, whenever $M$ is an $A$-module in $\textcat{DAlg}^{\heartsuit}(\textcat{Ch}(\textnormal{CBorn}_\mathbb{R}))$, then $M\otimes^\mathbb{L}_AB$ is in $\textcat{DAlg}^{\heartsuit}(\textcat{Ch}(\textnormal{CBorn}_\mathbb{R}))$,
        \item $f$ is \textit{derived strong} if there is an equivalence
        \begin{equation*}
    \pi_*(A)\otimes^\mathbb{L}_{\pi_0(A)}\pi_0(B)\rightarrow{\pi_*(B)}
\end{equation*}
\item $f$ is \textit{derived strongly flat} if it is derived strong and $\pi_0(f)$ is flat.

    \end{enumerate}
    
\end{defn}
\begin{lem}\label{derivedstrongpi0}
\cite[Lemma 5.2.10]{savage_representability_2024} Suppose that $f:A\rightarrow{B}$ is a derived strong map in $\textcat{DAlg}^{cn}(\mathcal{C})$. Then, $\pi_0(B)\simeq B\otimes_A^\mathbb{L}\pi_0(A)$.
\end{lem}

\begin{lem}\phantomsection\label{derivedstronglyflatflat}
    \cite[c.f. Corollary 2.3.86]{ben-bassat_perspective_2024} Given a morphism $f:A\rightarrow{B}$ in $\textcat{DAlg}^{cn}(\mathcal{C})$, the following are equivalent,
\begin{enumerate}
    \item $f$ is a flat morphism, 
    \item $f$ is a derived strongly flat morphism, 
    \item $B\otimes^\mathbb{L}_A -:\textcat{Mod}_A^{cn}\rightarrow{\textcat{Mod}_B^{cn}}$ commutes with finite limits.
\end{enumerate}
\end{lem}
Recall the definition of a homotopy monomorphism from Definition \ref{homotopyepi}. 

\begin{defn}Let $f:Y=\textnormal{Spec}(B)\rightarrow{\textnormal{Spec}(A)=X}$ be a morphism in $\textcat{DAff}^{cn}(\textcat{Ch}(\textnormal{CBorn}_\mathbb{R}))$. Then,  $f$ is a \textit{$\mathcal{C}^\infty$-localisation} if it is a flat homotopy monomorphism. 
\end{defn}

\begin{exmp}
    Suppose that we have a $\mathcal{C}^\infty$-ring $A$ and $a\in \mathcal{C}^\infty(A)$ is an element. Then, the localisation $\mathcal{C}^\infty(A)\rightarrow{\mathcal{C}^\infty(A)[a^{-1}]}$ is a flat map of commutative $\mathbb{R}$-algebras by \cite[Corollary 2.3.105]{steffens_derived_2023}. Moreover, we can see that, since $\mathcal{C}^\infty(A)[a^{-1}]\otimes_{\mathcal{C}^\infty(A)}\mathcal{C}^\infty(A)[a^{-1}]\simeq \mathcal{C}^\infty(A)[a^{-1}]$ and the localisation is flat, then it is a homotopy epimorphism. 
\end{exmp}

\begin{defn}
    The \textit{finite $\mathcal{C}^\infty$-localisation topology} on $\textnormal{Ho}(\textcat{DAff}^{cn}(\textcat{Ch}(\textnormal{CBorn}_\mathbb{R})))$ has finite covers $\{f_j:U_j=\textnormal{Spec}(B_j)\rightarrow{\textnormal{Spec}(A)=X}\}_{j\in J}$, satisfying the following properties
    \begin{enumerate}
        \item Each $A\rightarrow{B_j}$ is a $\mathcal{C}^\infty$-localisation, 
        \item The covering family is a \textit{formal covering family}, in the sense that the family of functors $\{(f_j)_!:\textcat{Mod}_A^{cn}\rightarrow{\textcat{Mod}^{cn}_{B_j}}\}_{j\in J}$ is conservative.
    \end{enumerate}
\end{defn}
We denote the class of $\mathcal{C}^\infty$-localisations by $\textcat{loc}$ and we denote the finite $\mathcal{C}^\infty$-localisation topology by $\bm\tau_{\textcat{loc}}$. Given a class of maps $\textcat{P}$ and a topology $\bm\tau$, we let $\textcat{P}^{\bm\tau}$ be the class of maps $f:Y\rightarrow{X}$ in $\textcat{DAff}^{cn}(\textcat{Ch}(\textnormal{CBorn}_\mathbb{R}))$ such that there is a $\bm\tau$-cover $\{g_i:U_i\rightarrow{Y}\}_{i\in I}$ with $f\circ g_i\in \textcat{P}$.

\begin{defn}In the following definition, $\textcat{Alg}$ will denote each of the categories $\mathcal{C}^\infty\textcat{DBornRing}$,  $\mathcal{C}^\infty\textnormal{BornRing}$, $\mathcal{C}^\infty\textcat{DRing}$, and $\mathcal{C}^\infty\textnormal{Ring}$, embedded as full subcategories of $\textcat{DAlg}^{cn}(\textcat{Ch}(\textnormal{CBorn}_\mathbb{R}))$. 
\begin{itemize}
    \item The class $\textcat{open}_{\mathcal{C}^\infty}$ of \textit{open immersions} is the subclass of maps $f:Y=\textnormal{Spec}(B)\rightarrow{\textnormal{Spec}(A)=X}$ in $\textcat{loc}^{\bm\tau_{\textcat{loc}}}$ such that, whenever $A\rightarrow{C}$ is a map in $\textcat{open}_{\mathcal{C}^\infty}$ with $C\in\textcat{Alg}$, then $B\otimes 
        ^\mathbb{L}_AC\in\textcat{Alg}$,
    \item The \textit{finite $\mathcal{C}^\infty$-topology}, which we will denote by $\bm\tau_{\mathcal{C}^\infty}$, consists of finite covers $\{U_j=\textnormal{Spec}(B_j)\rightarrow{\textnormal{Spec}(A)=X}\}_{j\in J}$ in $\bm\tau_{\textcat{loc}}$ such that, whenever there is some map $A\rightarrow{C}$, then $C\in\textcat{Alg}$ if and only if $B_j\otimes_A^\mathbb{L}C\in\textcat{Alg}$.
\end{itemize}
\end{defn}

We remark that $\mathcal{C}^\infty$-localisations and $\mathcal{C}^\infty$-open immersions are closed under equivalences, compositions, and pullbacks. 

\begin{cor}\phantomsection\label{strongmapcinfinity}Suppose that $f:Y\rightarrow{X}$ is a map in $\textcat{DAff}^{cn}(\textcat{Ch}(\textnormal{CBorn}_\mathbb{R}))$. Then, 
\begin{enumerate}
    \item $f$ is a $\mathcal{C}^\infty$-localisation if and only if it is derived strong and $t_0(f)$ is a $\mathcal{C}^\infty$-localisation, 
    \item $f$ is an open immersion if it is derived strong and $t_0(f)$ is an open immersion.
\end{enumerate}
\end{cor}

\begin{proof}If a morphism $f:Y\rightarrow{X}$ is a $\mathcal{C}^\infty$-localisation, then it is derived strong by Lemma \ref{derivedstronglyflatflat}. We can then use flatness to easily show that $t_0(f)$ is a $\mathcal{C}^\infty$-localisation. The converse is true using Lemma \ref{derivedstronglyflatflat} and \cite[Proposition 2.6.165]{ben-bassat_perspective_2024}. The result in the case of open immersions follows using Lemma \ref{derivedstrongpi0}, Lemma \ref{derivedstronglyflatflat}, and \cite[ Proposition 2.6.165]{ben-bassat_perspective_2024}, along with the statement that covers in $\bm\tau_{\mathcal{C}^\infty}$ are derived strong.
\end{proof}

Since covers in $\bm\tau_{\mathcal{C}^\infty}$ are covers in the faithfully flat topology, we can easily show that $\textcat{QCoh}$ satisfies hyperdescent using a similar proof to \cite[Lemma 6.7.8]{savage_representability_2024} (see also \cite[c.f. Lemma 2.2.2.13]{toen_homotopical_2008}). If we consider the presheaf $\textcat{Dls}:\textcat{DAff}^{cn}(\textcat{Ch}(\textnormal{CBorn}_\mathbb{R}))^{op}\rightarrow{\textcat{Pr}^{\mathbb{L},\otimes}}$ which sends any $\textnormal{Spec}(A)$ to the category of strongly dualisable modules, then by \cite[c.f. Lemma 7.2.35]{ben-bassat_perspective_2024}, $\textcat{Dls}$ also satisfies hyperdescent for $\bm\tau_{\mathcal{C}^\infty}$-covers. 

We recall the notation $n$-representable$|_{\mathcal{C}^\infty\textcat{DBAff}}$ and $n\textcat{-open}_{\mathcal{C}^\infty}|_{\mathcal{C}^\infty\textcat{DBAff}}$ from Definition \ref{geometricstackdefinition}. 

\begin{cor}\phantomsection\label{infinitesimalcinfinitycor}
   Suppose that we have an $n$-representable$|_{\mathcal{C}^\infty\textcat{DBAff}}$ morphism $f:\mathcal{F}\rightarrow{\mathcal{G}}$ of stacks in $\textcat{Stk}(\mathcal{C}^\infty\textcat{DBAff},\bm\tau_{\mathcal{C}^\infty}|_{\mathcal{C}^\infty\textcat{DBAff}})$. Then, $f$ is in $n\textcat{-open}_{\mathcal{C}^\infty}|_{\mathcal{C}^\infty\textcat{DBAff}}$ if $f$ satisfies that, for any $x:X=\textnormal{Spec}(A)\rightarrow{\mathcal{F}}$ and any $M\in\textcat{M}_{A,1}$, 
    \begin{equation*}
        \pi_0(\textnormal{Map}_{\textcat{Mod}_A}(\mathbb{L}_{\mathcal{F}/\mathcal{G},x},M))=0
    \end{equation*}The converse holds if $t_0(\mathcal{F})\rightarrow{t_0(\mathcal{G})}$ is in $n\textcat{-open}^{\heartsuit}_{\mathcal{C}^\infty}$.
\end{cor}

\begin{proof}
Indeed, we note that, since $\textcat{Dls}$ satisfies descent for $\bm\tau_{\mathcal{C}^\infty}$-covers, the forwards direction follows from \cite[Proposition 6.5.2]{savage_representability_2024} since $\textcat{open}_{\mathcal{C}^\infty}\subseteq\textcat{loc}^{\bm\tau_{\mathcal{C}^\infty}}$ and $\mathcal{C}^\infty$-localisations are formally perfect by \cite[Proposition 3.5.5]{savage_representability_2024}. To prove the converse, we note that, by the proof of \cite[Proposition 6.5.2]{savage_representability_2024}, it suffices to show that if we have a morphism $f:Y=\textnormal{Spec}(B)\rightarrow{\textnormal{Spec}(A)=X}$ in $\mathcal{C}^\infty\textcat{DBAff}$ which is in $\textcat{fP}^{\bm\tau_{\mathcal{C}^\infty}}$ and satisfies that its truncation lies in $\textcat{open}_{\mathcal{C}^\infty}^\heartsuit$, then $f$ lies in $\textcat{open}_{\mathcal{C}^\infty}$. We can also assume that the cover $\{U_i\rightarrow{Y}\}_{i\in I}$ in $\bm{\tau}_{\mathcal{C}^\infty}$, chosen such that each $U_i\rightarrow{X}$ is formally perfect, also satisfies that $t_0(U_i)\rightarrow{t_0(X)}$ is in $\textcat{loc}^\heartsuit$. Therefore, we see that, since $t_0(U_i)\rightarrow{t_0(X)}$ is flat and formally perfect, the formally perfect morphism $U_i\rightarrow{X}$ is derived strong by \cite[Corollary 6.3.4]{savage_representability_2024}. Since covers in $\bm\tau_{\mathcal{C}^\infty}$ are formal covering families and derived strong by Corollary \ref{strongmapcinfinity}, we see that the morphism $f:Y\rightarrow{X}$ is derived strong by \cite[Lemma 6.1.4]{savage_representability_2024}. Hence, we conclude that $f$ is in $\textcat{open}_{\mathcal{C}^\infty}$ by Corollary \ref{strongmapcinfinity}.
\end{proof}

It is clear from our definitions that we have strong relative $(\infty,1)$-geometry tuples
\begin{equation*}
    (\textcat{DAff}^{cn}(\textcat{Ch}(\textnormal{CBorn}_\mathbb{R})),\bm\tau_{\mathcal{C}^\infty},\textcat{open}_{\mathcal{C}^\infty},\mathcal{C}^\infty\textcat{DBAff})
\end{equation*}and \begin{equation*}
    (\mathcal{C}^\infty\textcat{DBAff},\bm\tau_{\mathcal{C}^\infty}|_{\mathcal{C}^\infty\textcat{DBAff}},\textcat{open}_{\mathcal{C}^\infty}|_{\mathcal{C}^\infty\textcat{DBAff}},\mathcal{C}^\infty\textcat{DAff})
\end{equation*}in the sense of \cite[Definition 2.1.2]{savage_representability_2024}. Therefore, by \cite[Proposition 6.7]{kelly_analytic_2022}, we have a chain of fully faithful inclusions
\begin{equation*}
    \textcat{Stk}(\mathcal{C}^\infty\textcat{DAff},\bm\tau_{\mathcal{C}^\infty})\rightarrow{\textcat{Stk}(\mathcal{C}^\infty\textcat{DBAff},\bm\tau_{\mathcal{C}^\infty})}\rightarrow{\textcat{Stk}(\textcat{DAff}^{cn}(\textcat{Ch}(\textnormal{CBorn}_\mathbb{R})),\bm\tau_{\mathcal{C}^\infty})}
\end{equation*}where we have dropped the restriction notation for ease of notation. Moreover, there is a chain of full subcategory inclusions on the level of $n$-geometric stacks.

\subsection{A Representability Context for Derived $\mathcal{C}^\infty$-Bornological Geometry}

We now endow the $(\infty,1)$-geometry tuples in the previous section with the necessary structure so that we obtain representability contexts. Indeed, the following is clear.

\begin{lem}\phantomsection\label{tuplecinfinity}
    $(\mathcal{C}^\infty\textcat{DBAff},\bm\tau_{\mathcal{C}^\infty}|_{\mathcal{C}^\infty\textcat{DBAff}},\textcat{open}_{\mathcal{C}^\infty}|_{\mathcal{C}^\infty\textcat{DBAff}},\mathcal{C}^\infty\textnormal{BAff})$ is a strong relative $(\infty,1)$-geometry tuple. 
\end{lem}

\begin{lem}\phantomsection\label{cinfinitycontinuous}
    $\iota|_{\mathcal{C}^\infty\textnormal{BAff}}:(\mathcal{C}^\infty\textnormal{BAff},\bm{\tau}_{\mathcal{C}^\infty}^\heartsuit)\rightarrow(\mathcal{C}^\infty\textcat{DBAff},\bm{\tau}_{\mathcal{C}^\infty}|_{\mathcal{C}^\infty\textcat{DBAff}})$ is a continuous functor of $(\infty,1)$-sites. 
\end{lem}
\begin{proof}Suppose that $\{t_0(U_j)=\textnormal{Spec}(\pi_0(B_j))\rightarrow{\textnormal{Spec}(\pi_0(A))=t_0(X)}\}_{j\in J}$ is a cover in $\bm{\tau}_{\mathcal{C}^\infty}^\heartsuit$ corresponding to the truncation of a cover $\{U_j\rightarrow{X}\}_{j\in J}$ in $\bm{\tau}_{\mathcal{C}^\infty}|_{\mathcal{C}^\infty\textcat{DBAff}}$. The cover is a formal covering family by \cite[Lemma 6.1.2]{savage_representability_2024}. We note that $\pi_0(A)\rightarrow\pi_0(B_j)$ is a $\mathcal{C}^\infty$-localisation and the morphism $A\rightarrow{B_j}$ is derived strong by Corollary \ref{strongmapcinfinity}. Therefore, using Lemma \ref{derivedstrongpi0}, we can show that $\pi_0(A)\rightarrow{\pi_0(B_j)}$ satisfies the conditions to be a cover in the finite $\mathcal{C}^\infty$-topology.
\end{proof}

\begin{lem}\phantomsection\label{cinfinityrepresentables}
     For any finite collection $\{U_i\}_{i\in I}$ in $\textcat{DAff}^{cn}(\textcat{Ch}(\textnormal{CBorn}_\mathbb{R}))$, the map $\coprod_{i\in I} h(U_i)\rightarrow{h(\coprod_{i\in I} U_i)}$ is an equivalence in $\textcat{Stk}(\mathcal{C}^\infty\textcat{DBAff},\bm{\tau}_{\mathcal{C}^\infty}|_{\mathcal{C}^\infty\textcat{DBAff}})$,
 \end{lem}
 \begin{proof}Suppose that $U_j=\textnormal{Spec}(A_j)$ and consider the morphism $U_i\rightarrow{\coprod_{j\in I} U_j}$. We note that this map is flat and, moreover, by \cite[Lemma 6.6.6]{savage_representability_2024}, it is a homotopy monomorphism. It is clear that the family $\{U_i\rightarrow{\coprod_{j\in I} U_j\}_{i\in I}}$ is a formal covering family. Moreover, for any morphism $\prod_{j\in I}A_j\rightarrow{C}$, we note that $A_i\otimes_{\prod_{j\in I}A_j}C$ is the $j$-th component of $C$, and since $\textcat{Alg}$ is closed under finite equalisers and products we can easily see that $A_i\otimes_{\prod_{j\in I}A_j}C\in\textcat{Alg}$ if and only if $C\in\textcat{Alg}$. Hence, we can deduce that the family $\{U_i\rightarrow{\coprod_{j\in I} U_j\}_{i\in I}}$ is a cover in $\bm\tau_{\mathcal{C}^\infty}$. We then conclude by \cite[Corollary 2.3.7]{savage_representability_2024}. 
 \end{proof}

Let $\textcat{Mod}$ denote the collection of, for each $X=\textnormal{Spec}(A)\in\mathcal{C}^\infty\textcat{DBAff}$, the categories $\textcat{Mod}_A$ of $A$-modules. 

\begin{lem}\phantomsection\label{postnikovcinfinity}$(\texthom{\textbf{Ch}}(\textnormal{CBorn}_\mathbb{R}),\bm\tau_{\mathcal{C}^\infty},\textcat{open}_{\mathcal{C}^\infty},\mathcal{C}^\infty\textcat{DBAff},\textcat{Mod})$ is a flat Postnikov compatible derived geometry context in the sense of \cite[Definition 4.5.2]{savage_representability_2024}.
    
\end{lem}

\begin{proof}We easily note that it is a derived geometry context in the sense of \cite[Definition 3.6.2]{savage_representability_2024} using results from the previous section and the observation that $\textcat{Mod}$ is a good system of categories of modules on $\mathcal{C}^\infty\textcat{DBAff}$ in the sense of \cite[Definition 3.6.1]{savage_representability_2024}. By \cite[Proposition 3.1.69]{ben-bassat_perspective_2024}, it is flat. The conditions for Postnikov compatibility in the sense of \cite[Definition 4.5.2]{savage_representability_2024} also easily follow. In particular, since we are working over $\mathbb{R}$, Condition (4) follows using a similar proof to \cite[Proposition 25.2.4.1]{lurie_spectral_2018}. 
\end{proof}

\begin{lem}\label{obstructioncinfinity}
    $\bm{\tau}_{\mathcal{C}^\infty}$ and $\textcat{open}_{\mathcal{C}^\infty}$ satisfy the obstruction conditions, in the sense of \cite[Definition 4.3.2]{savage_representability_2024}, relative to $\mathcal{C}^\infty\textcat{DBAff}$ for the class $\textcat{hm}$ of homotopy monomorphisms. 
\end{lem}

\begin{proof}We note that any homotopy monomorphism is formally $i$-smooth by \cite[Corollary 4.2.3 and Lemma 6.6.2]{savage_representability_2024}. Therefore, Condition (2) easily follows and Condition (1) follows using \cite[Lemma 6.1.3]{savage_representability_2024}. 

To show that Condition (3) is satisfied, we suppose that we have a $\bm{\tau}_{\mathcal{C}^\infty}$-covering family $\{V_j=\textnormal{Spec}(A_j)\rightarrow{\textnormal{Spec}(A)}=X\}_{j\in J}$ and some $M\in\textcat{Mod}_{A,1}$. This can trivially be refined to a finite $\textcat{hm}$-covering family. Now, define $W_j=\textnormal{Spec}(B_j)=\textnormal{Spec}(A_j\oplus_{d_j'} \Omega M_j')\in\mathcal{C}^\infty\textcat{DBAff}$ with $d_j'$ the induced derivation in $\pi_0(\textcat{Der}(B_j,M_j'))$ and $M_j'=M\otimes^\mathbb{L}_A A_j$. The collection $\{W_j\rightarrow{X_d[\Omega M]}\}_{j\in J}$ is a formal covering family by \cite[Lemma 6.1.5]{savage_representability_2024}. By \cite[Lemma 4.3.3]{savage_representability_2024}, and using that the morphism $V_j\rightarrow{X}$ is a homotopy monomorphism, we have an equivalence
\begin{equation*}
B_j\otimes_{A\oplus_d\Omega M}^\mathbb{L}A\simeq A_j\simeq A_j\otimes^\mathbb{L}_AA_j\simeq (B_j\otimes_{A\oplus_d\Omega M}^\mathbb{L}B_j)\otimes_{A\oplus_d\Omega M}^\mathbb{L}A
\end{equation*}Hence, by \cite[Corollary 3.4.3]{savage_representability_2024}, we have an equivalence $B_j\simeq B_j\otimes^\mathbb{L}_{A\oplus_d\Omega M}B_j$. Therefore, $\{W_j\rightarrow{X_d[\Omega M]}\}_{j\in J}$ is also a finite $\textcat{hm}$-covering family. 

By \cite[Lemma 6.6.4]{savage_representability_2024}, since we have $\textcat{QCoh}$-descent for finite $\textcat{hm}$-covers and since $W_j$ and $X_d[\Omega M]$ are $\mathcal{C}^\infty\textcat{DBAff}$-admissible, we see that this defines an epimorphism of stacks $\coprod_{j\in J}W_j\rightarrow{X_d[\Omega M]}$ in $\textcat{Stk}(\mathcal{C}^\infty\textcat{DBAff},\bm{\tau}_{\mathcal{C}^\infty}|_{\mathcal{C}^\infty\textcat{DBAff}})$ as required. 
    
\end{proof}

\begin{cor}\phantomsection\label{tuplediffgeometryrep} The tuple $(\texthom{\textbf{Ch}}(\textnormal{CBorn}_\mathbb{R}),\bm\tau_{\mathcal{C}^\infty},\textcat{open}_{\mathcal{C}^\infty},\mathcal{C}^\infty\textcat{DBAff},\textcat{Mod},\textcat{hm})$ is a representability context in the sense of \cite[Definition 5.3.1]{savage_representability_2024}. 
\end{cor}

\begin{proof}
    By Lemma \ref{postnikovcinfinity}, our context is a flat Postnikov compatible derived geometry context. Now, Condition (1) follows from Lemma \ref{tuplecinfinity}. Condition (2) follows by Lemma \ref{cinfinitycontinuous}. Condition (3) follows by Lemma \ref{cinfinityrepresentables}. Condition (4) follows by Corollary \ref{strongmapcinfinity} and Condition (5) by Corollary \ref{infinitesimalcinfinitycor}. Finally, Condition (6) follows by Lemma \ref{obstructioncinfinity}. 
\end{proof}

\subsection{Derived Partial Differential Equations}\label{defnasmappingstack}

Classically, a system of partial differential equations of order $k$ can be realised as a submanifold of the \textit{$k$-jet space}. See \cite{krasilshchik_geometry_2011} for more details. 

An alternative formulation appears in the language of $D$-modules. Indeed, suppose that we have a system of partial differential equations $E_0$ of order $k$ defined using the sheaf of differential operators associated to a smooth scheme $X$. Then, we can define a \textit{$D$-module} $\mathcal{M}_{E_0}$ encoding the system, see \cite[Chapter 6]{coutinho_primer_1995}, and then the space of solutions is given by 
\begin{equation*}
    \textnormal{Hom}_{\mathcal{D}_X}(\mathcal{M}_{E_0},\mathcal{O}_{X})
\end{equation*}where $\mathcal{D}_{X}$ is the sheaf of differential operators on $X$. 

Given a smooth scheme $X$, the associated \textit{de Rham space}, $X_{dR}$, is the presheaf on affine schemes defined by sending $\textnormal{Spec}(B)$ to $X(B^{red})$ where $B^{red}$ is the quotient of $B$ by the nilradical, the ideal of nilpotent elements. We remark that there is an equivalence between $\mathcal{D}_X$-modules and quasicoherent sheaves on the de Rham space $X_{dR}$ associated to $X$, and hence we can identify the solution space with global sections of $X_{dR}$. In this section, we generalise these ideas to define the notion of a derived partial differential equation and its solution space. 

A construction of the de Rham space exists in the general setting of $\mathcal{C}^\infty$-schemes, as described in \cite{borisov_beyond_2018}. Recall that there exists a fully faithful sifted colimit preserving functor $\mathcal{C}^\infty:\mathcal{C}^\infty\textnormal{Ring}\rightarrow{\textnormal{Comm}(\textnormal{CBorn}_\mathbb{R})}$. Using \cite[Definition 1]{borisov_beyond_2018}, we can define a functor 
\begin{equation*}\begin{aligned}
R_{nil}:\mathcal{C}^\infty\textnormal{Ring}&\rightarrow{\textnormal{Comm}(\textnormal{CBorn}_\mathbb{R}})\\
    B&\rightarrow{\mathcal{C}^\infty(B)/\overline{\sqrt[nil]{0}}}
\end{aligned}\end{equation*}where, for any ideal $I\leq \mathcal{C}^\infty(B)$, $\sqrt[nil]{I}:=\{f\in \mathcal{C}^\infty(B)\mid \exists k\in\mathbb{Z}_{>0} \text{ such that }f^k\in I\}$. Here $\overline{\sqrt[nil]{0}}$ denotes the closure of the ideal in the natural bornology on $\mathcal{C}^\infty(B)$ and we consider the quotient bornology on $\mathcal{C}^\infty(B)/\overline{\sqrt[nil]{0}}$.

We note that, since $R_{nil}(B)$ is a sifted colimit of an object in the essential image of $\mathcal{C}^\infty$, then it lies in the essential image of $\mathcal{C}^\infty$. For $B\in\mathcal{C}^\infty\textnormal{Ring}$, we define the \textit{reduced $\mathcal{C}^\infty$-ring}, $B^{red}$, to be the object in $\mathcal{C}^\infty\textnormal{Ring}$ corresponding to $R_{nil}(B)$, defined up to isomorphism.

\begin{defn}Suppose that $\mathcal{X}\in\textcat{Stk}(\mathcal{C}^\infty\textcat{DAff},\bm\tau_{\mathcal{C}^\infty}|_{\mathcal{C}^\infty\textcat{DAff}})$ and $Y=\textnormal{Spec}(B)$ is in $\mathcal{C}^\infty\textcat{DAff}$. 
\begin{enumerate}
    \item The \textit{reduced $\mathcal{C}^\infty$-ring}, denoted $B^{red}$, is the object of $\mathcal{C}^\infty\textnormal{Ring}$ defined by $B^{red}:=\pi_0(B)^{red}$, 
    \item The \textit{de Rham space} is the presheaf $\mathcal{X}_{dR}$ defined by sending $\textnormal{Spec}(B)$ to $\mathcal{X}(B^{red})$. 
\end{enumerate}
\end{defn}

\begin{lem}
    $\mathcal{X}_{dR}$ is a stack for the $\bm\tau_{\mathcal{C}^\infty}|_{\mathcal{C}^\infty\textcat{DAff}}$-topology. 
\end{lem}

\begin{proof}
    It suffices to show that if we have a cover $\{U_i=\textnormal{Spec}(C_i)\rightarrow{\textnormal{Spec}(B)=Y}\}_{i\in I}$ in $\bm\tau_{\mathcal{C}^\infty}|_{\mathcal{C}^\infty\textcat{DAff}}$, then $\{\textnormal{Spec}(C_i^{red})\rightarrow{\textnormal{Spec}(B^{red})}\}_{i\in I}$ is also a cover. Indeed, this follows from the observation that $\mathcal{C}^\infty(C_i^{red})\otimes^\mathbb{L}_{\mathcal{C}^\infty(B^{red})}D\simeq \mathcal{C}^\infty((C_i\otimes^\mathbb{L}_B D)^{red})$ for any morphism $B^{red}\rightarrow{D}$ in $\textcat{DAlg}^{cn}(\textcat{Ch}(\textnormal{CBorn}_\mathbb{R}))$. 
\end{proof}

Fix some $\mathcal{X}\in\textcat{Stk}(\mathcal{C}^\infty\textcat{DAff},\bm\tau_{\mathcal{C}^\infty}|_{\mathcal{C}^\infty\textcat{DAff}})$ and consider the induced morphism of stacks $p_{dR}:\mathcal{X}\rightarrow{\mathcal{X}_{dR}}$ in $\textcat{Stk}(\mathcal{C}^\infty\textcat{DAff},\bm\tau_{\mathcal{C}^\infty}|_{\mathcal{C}^\infty\textcat{DAff}})$. There is an induced functor
\begin{equation*}
    p_{dR,*}:\textcat{Stk}(\mathcal{C}^\infty\textcat{DAff},\bm\tau_{\mathcal{C}^\infty}|_{\mathcal{C}^\infty\textcat{DAff}})_{/\mathcal{X}}\rightarrow{\textcat{Stk}(\mathcal{C}^\infty\textcat{DAff},\bm\tau_{\mathcal{C}^\infty}|_{\mathcal{C}^\infty\textcat{DAff}})_{/\mathcal{X}_{dR}}}
\end{equation*}which has a left adjoint given by 
\begin{equation*}
    p_{dR}^*:\textcat{Stk}(\mathcal{C}^\infty\textcat{DAff},\bm\tau_{\mathcal{C}^\infty}|_{\mathcal{C}^\infty\textcat{DAff}})_{/\mathcal{X}_{dR}}\rightarrow{\textcat{Stk}(\mathcal{C}^\infty\textcat{DAff},\bm\tau_{\mathcal{C}^\infty}|_{\mathcal{C}^\infty\textcat{DAff}})_{/\mathcal{X}}}
\end{equation*}

The following definitions and their connections with classical notions is discussed in \cite[Section 3.5]{kryczka_derived_2024}.

\begin{defn}\label{jetdefn}
    Suppose that $\mathcal{F}\in\textcat{Stk}(\mathcal{C}^\infty\textcat{DAff},\bm\tau_{\mathcal{C}^\infty}|_{\mathcal{C}^\infty\textcat{DAff}})_{/\mathcal{X}}$. Then, 
    \begin{enumerate}
        \item The associated \textit{de Rham jet stack} is defined to be $\textcat{Jets}_{\mathcal{X}_{dR}}^\infty(\mathcal{F}):=p_{dR,*}(\mathcal{F})$, 
        \item The associated \textit{jet stack} is defined to be $\textcat{Jets}_{\mathcal{X}}^\infty(\mathcal{F}):=p_{dR}^*\circ p_{dR,*}(\mathcal{F})$ 
    \end{enumerate}
\end{defn}

By definition of the base-change map, we see that we have a pullback square
\begin{equation*}
    \begin{tikzcd}
    \textcat{Jets}_\mathcal{X}^\infty(\mathcal{F})\arrow{r} \arrow{d}
 & \textcat{Jets}_{\mathcal{X}_{dR}}^\infty(\mathcal{F})\arrow{d}\\
        \mathcal{X}\arrow{r}{p_{dR}} & {\mathcal{X}}_{dR}
    \end{tikzcd}
\end{equation*}in $\textcat{Stk}(\mathcal{C}^\infty\textcat{DAff},\bm\tau_{\mathcal{C}^\infty}|_{\mathcal{C}^\infty\textcat{DAff}})$. 

\begin{defn}\cite[c.f. Definition 2.2.3.5]{toen_homotopical_2008}
    Suppose that we have a morphism $f:\mathcal{F}\rightarrow{\mathcal{G}}$ of stacks in $\textcat{Stk}(\mathcal{C}^\infty\textcat{DAff},\bm\tau_{\mathcal{C}^\infty}|_{\mathcal{C}^\infty\textcat{DAff}})$. Then, $f$ is a \textit{closed immersion} if it is $(-1)$-representable and if, for any $X=\textnormal{Spec}(A)\in\mathcal{C}^\infty\textcat{DAff}$, the induced morphism
    \begin{equation*}
        \mathcal{F}\times_\mathcal{G}X:= \textnormal{Spec}(B)\rightarrow{\textnormal{Spec}(A)}
    \end{equation*}induces an epimorphism $\pi_0(A)\rightarrow{\pi_0(B)}$ in $\textnormal{Comm}(\textnormal{CBorn}_\mathbb{R})$. 
\end{defn}

In particular, if we have a closed immersion $\mathcal{G}\rightarrow{\textcat{Jets}^\infty_{\mathcal{X}_{dR}}(\mathcal{F})}$, then this also defines a closed immersion $\mathcal{G}\times_{\textcat{Jets}^\infty_{\mathcal{X}_{dR}}(\mathcal{F})}\textcat{Jets}^\infty_{\mathcal{X}}(\mathcal{F})\rightarrow{\textcat{Jets}^\infty_{\mathcal{X}}(\mathcal{F})}$. 

\begin{defn}\label{pdedefn}Suppose that $\mathcal{F}\in\textcat{Stk}(\mathcal{C}^\infty\textcat{DAff},\bm\tau_{\mathcal{C}^\infty}|_{\mathcal{C}^\infty\textcat{DAff}})_{/\mathcal{X}}$. Then, a stack $\mathcal{G}$ is a \textit{system of derived non-linear partial differential equations on sections} of $\mathcal{F}$ if there is a closed immersion $\mathcal{G}\rightarrow{\textcat{Jets}_{\mathcal{X}_{dR}}^\infty(\mathcal{F})}$ in $\textcat{Stk}(\mathcal{C}^\infty\textcat{DAff},\bm\tau_{\mathcal{C}^\infty}|_{\mathcal{C}^\infty\textcat{DAff}})_{/\mathcal{X}_{dR}}$.
\end{defn}
We can then naturally make the following definition motivated by the classical theory. 
\begin{defn}\label{solnstackdefn}Suppose that we have a morphism $\mathcal{X}\rightarrow\mathcal{F}$ in $\textcat{Stk}(\mathcal{C}^\infty\textcat{DAff},\bm\tau_{\mathcal{C}^\infty}|_{\mathcal{C}^\infty\textcat{DAff}})$. Suppose that $\mathcal{G}$ is a system of derived non-linear partial differential equations on sections of $\mathcal{F}$. Then, \textit{the derived moduli stack of solutions to $\mathcal{G}$}, denoted $\textcat{Sol}_\mathcal{X}(\mathcal{G})$, is the mapping stack of sections
    \begin{equation*}
        \textcat{Sol}_\mathcal{X}(\mathcal{G}):=\texthom{Map}_{\textcat{Stk}(\mathcal{C}^\infty\textcat{DAff},\bm\tau_{\mathcal{C}^\infty}|_{\mathcal{C}^\infty\textcat{DAff}})_{/\mathcal{X}_{dR}}}(\mathcal{X}_{dR},\mathcal{G})
    \end{equation*}
\end{defn}

\begin{remark}
    By \cite[Theorem 3.57]{kryczka_derived_2024}, if $\mathcal{X}$ is a smooth scheme we recover the derived version of the $D$-module solution space.
\end{remark}

\subsection{Finitely Generated and Finitely Presented $\mathcal{C}^\infty$-Rings}\label{finitelypresentedcinfinity}

We first recall the classical definitions of finitely generated and finitely presented $\mathcal{C}^\infty$-rings. 

\begin{defn}
    A $\mathcal{C}^\infty$-ring $A$ is
    \begin{enumerate}
        \item \textit{finitely generated} if $\mathcal{C}^\infty(A)$ is equivalent to a quotient $\mathcal{C}^\infty(\mathbb{R}^n)/I$ in $\textnormal{Comm}(\textnormal{Mod}_\mathbb{Z})$ for some $n\in\mathbb{N}$ and some ideal $I$, 
        \item \textit{finitely presented} if it is finitely generated and the ideal $I$ is finitely generated. 
    \end{enumerate}
\end{defn}

We recall that $\mathcal{C}^\infty(\mathbb{R}^n)$ is a Fr\'echet space and can be considered as a complete bornological space with the von Neumann bornology. We note that a finitely generated $\mathcal{C}^\infty$-ring is not a complete bornological space unless the ideal $I$ is closed in the bornology. Therefore, we make the following definitions. 

\begin{defn}
    A $\mathcal{C}^\infty$-ring $A$ is
    \begin{enumerate}
        \item \textit{bornologically finitely generated} if $\mathcal{C}^\infty(A)$ is equivalent to a quotient $\mathcal{C}^\infty(\mathbb{R}^n)/I$ in the category $\textnormal{Comm}(\textnormal{CBorn}_\mathbb{R})$ for some $n\in\mathbb{N}$ and some ideal $I$ which is closed in the bornology on $\mathcal{C}^\infty(\mathbb{R}^n)$ (equivalently, closed with respect to the topology coming from the Fr\'echet space structure \cite[Section 1.2.2]{meyer_local_2007}), 
        \item \textit{bornologically finitely presented} if it is bornologically finitely generated and the ideal $I$ is finitely generated. 
    \end{enumerate}
\end{defn}

\begin{lem}\label{finprescinfinity}Suppose that $A$ is a bornologically finitely generated $\mathcal{C}^\infty$-ring. Then $\mathcal{C}^\infty(A)$ is reflexive as a complete bornological space. 
\end{lem}

\begin{proof}
    Indeed, we note that $\mathcal{C}^\infty(\mathbb{R}^n)$ is a nuclear Fr\'echet space by \cite[Proposition 6.1]{kriegl_convenient_1997}. Hence, any quotient by a closed ideal is also a nuclear Fr\'echet space by \cite[p. 103]{schaefer_topological_1999}, and hence is reflexive as a complete bornological space. 
\end{proof}

We now want to explore conditions under which classical finitely generated or finitely presented $\mathcal{C}^\infty$-rings are bornologically finitely generated or finitely presented.

\begin{defn}
    A finitely presented $\mathcal{C}^\infty$-ring $A$ with $\mathcal{C}^\infty(A)\simeq \mathcal{C}^\infty(\mathbb{R}^n)/(f_1,\dots,f_m)$ is in addition \textit{regular} if the induced manifold map $f=(f_1,\dots,f_m):\mathbb{R}^n\rightarrow{\mathbb{R}^m}$ is regular, i.e. transverse to the zero map $*\hookrightarrow{\mathbb{R}^m}$.
\end{defn}

\begin{remark}
    We note that the above definition of being a finitely presented regular $\mathcal{C}^\infty$-ring is an underived analogue of the definition of being finitely presented and quasi-smooth in \cite[Remark 1.0.1]{steffens_representability_2024}. 
\end{remark}

\begin{lem}\label{cinfinitymanifold}
    Suppose that we have a $\mathcal{C}^\infty$-ring $A$ which is finitely presented and regular. Then, it is bornologically finitely presented. Moreover, in this situation, $\mathcal{C}^\infty(A)\simeq\mathcal{C}^\infty(M)$ for some closed submanifold $M\subseteq \mathbb{R}^n$. 
\end{lem}

\begin{proof}
    We note that $\mathcal{C}^\infty(A)$ is a quotient in $\textnormal{Comm}(\textnormal{Mod}_\mathbb{Z})$ of $\mathcal{C}^\infty(\mathbb{R}^n)$ by an ideal finitely generated by relations $\{f_1,\dots,f_m\}$. Let $M=f^{-1}(0)$ where $f=(f_1,\dots,f_m):\mathbb{R}^n\rightarrow{\mathbb{R}^m}$. Since $0$ is a regular value, $M$ is a closed submanifold of $\mathbb{R}^n$. Consider the continuous map $\mathcal{C}^\infty(\mathbb{R}^n)\rightarrow{\mathcal{C}^\infty(M)}$ defined by restriction. This has kernel ${(f_1,\dots,f_m)}$, which must be closed in the Fr\'echet topology on $\mathcal{C}^\infty(\mathbb{R}^n)$. We note that, since $M$ is a closed submanifold of $\mathbb{R}^n$, we can extend any smooth function on $M$ to a smooth function on $\mathbb{R}^n$ by \cite[Lemma 2.27]{lee_introduction_2003}. Hence, we see that we have an equivalence $\mathcal{C}^\infty(A)\simeq\mathcal{C}^\infty(\mathbb{R}^n)/(f_1,\dots,f_m)\simeq \mathcal{C}^\infty(M)$.
\end{proof}

Suppose that $A$ is a finitely generated $\mathcal{C}^\infty$-ring. Let $A_k$ be the $\mathcal{C}^\infty$-ring corresponding to the quotient $\mathcal{C}^\infty(A\hat{\otimes} A)/{I^k}$ in $\textnormal{Comm}(\textnormal{CBorn}_\mathbb{R})$. Here $I$ is the kernel of the morphism $\mathcal{C}^\infty(A\hat\otimes A)\rightarrow{\mathcal{C}^\infty(\Delta)}$, where $\Delta\subseteq A\hat\otimes A$ denotes the diagonal. Consider the two natural projection maps $\textnormal{Spec}(A_k)\rightarrow{\textnormal{Spec}(A)}$. 

\begin{lem}\label{coequaliserxdr}
    Suppose that $X=\textnormal{Spec}(A)$, where $A$ is a \textit{finitely presented regular} $\mathcal{C}^\infty$-ring. Then, $X_{dR}$ is equivalent to the coequaliser of the diagram 
    \begin{equation*}
        \varinjlim_{k\in\mathbb{Z}_{\geq 1}}\textnormal{Spec}(A_k)\rightrightarrows{\textnormal{Spec}(A)}
    \end{equation*}in $\textcat{PStk}(\mathcal{C}^\infty\textnormal{Aff})$. Under stackification it is also a colimit in $\textcat{Stk}(\mathcal{C}^\infty\textnormal{Aff},\bm\tau_{\mathcal{C}^\infty}^\heartsuit|_{\mathcal{C}^\infty\textnormal{Aff}})$. 
\end{lem}

\begin{proof}
    This follows from Example 3, Proposition 5, and Proposition 6 in \cite{borisov_beyond_2018} together with Lemma \ref{cinfinitymanifold}. 
\end{proof}

\begin{remark}
    We note that the above Lemma also holds in the situation where $A$ is finitely generated and $R_{nil}$-injective in the sense of \cite[Definition 4]{borisov_beyond_2018}. 
\end{remark}

\begin{lem}\label{bornfinpresAk}
    Suppose that $X=\textnormal{Spec}(A)$ where $A$ is a finitely presented regular $\mathcal{C}^\infty$-ring. For each $k\in\mathbb{Z}_{\geq 1}$, $A_k$ is a bornologically finitely presented $\mathcal{C}^\infty$-ring.  
\end{lem}

\begin{proof}We know that $A$ is a finitely presented regular $\mathcal{C}^\infty$-ring. Hence, it is bornologically finitely presented. Therefore, $\mathcal{C}^\infty(A)$ is equivalent in $\textnormal{Comm}(\textnormal{CBorn}_\mathbb{R})$ to a quotient $\mathcal{C}^\infty(\mathbb{R}^n)/{(f_1,\dots,f_m)}$ for $m,n\in\mathbb{N}$. Therefore, if we choose coordinates $(x,y)=(x_1,\dots,x_n,y_1,\dots,y_n)$ for $\mathbb{R}^{2n}$, we see that $\mathcal{C}^\infty(A\hat\otimes A)$ is equivalent to \begin{equation*}\mathcal{C}^\infty(\mathbb{R}^{2n})/{(f_1(x),\dots,f_m(x),f_1(y),\dots,f_m(y))}
    \end{equation*} in $\textnormal{Comm}(\textnormal{Mod}_\mathbb{Z})$. The morphism $\mathcal{C}^\infty(A\hat\otimes A)\rightarrow{\mathcal{C}^\infty(\Delta)}$ is equivalent to the morphism setting $x=y$. Therefore, we see that $I$ is equivalent to $(x_1-y_1,\dots,x_n-y_n)$ which is finitely generated. Therefore, $I^k$ is finitely generated for any finite $k\in\mathbb{Z}_{\geq 1}$ and hence \begin{equation*}
A_k=\mathcal{C}^\infty(\mathbb{R}^{2n})/{(I^k,f_1(x),\dots,f_m(x),f_1(y),\dots,f_m(y))}
    \end{equation*}is finitely presented for any $k\in\mathbb{Z}_{\geq 1}$.

    It now remains to show that $\mathcal{C}^\infty(A)/I^k$ is a Fr\'echet space. We will let $J$ denote the ideal $(f_1(x),\dots,f_m(x),f_1(y),\dots,f_m(y))$ and we will consider the projection map 
    \begin{equation*}
        \pi:\mathcal{C}^\infty(\mathbb{R}^{2n})\rightarrow{\mathcal{C}^\infty(\mathbb{R}^{2n})/J}=\mathcal{C}^\infty(A)
    \end{equation*}We note that $I^k$ is closed in the Fr\'echet space $\mathcal{C}^\infty(A)$ if and only if $\pi^{-1}(I^k)$ is closed in $\mathcal{C}^\infty(\mathbb{R}^{2n})$. We can see that $\pi^{-1}(I^k)=\underline{I}^{k}+J$ where $\underline{I}^k=(\prod_{i=1}^n(x_i-y_i)^{j_i}\mid \sum_{i=1}^n j_i> k)$. A map $f$ is in $\underline{I}^k$ if and only if $d_{x-y}^{j}f(x,x)=0$ where $j=(j_1,\dots j_n)$ is such that $\sum_{i=1}^n j_i\leq k$. Suppose that we have some $f\in\mathcal{C}^\infty(\mathbb{R}^{2n})$. If we Taylor expand along the diagonal we get $f=\sum_{j}\frac{1}{j!}d^j_{x-y} f(x,x)(x-y)^j+g$ for some $g\in \underline{I}^k$. Therefore, we see that if $f\in \underline{I}^k+J$, then each $d_{x-y}^{j}f(x,x)\in (f_1(x),\dots,f_m(x))$. Conversely, if each $d_{x-y}^{j}f(x,x)\in (f_1(x),\dots,f_m(x))$, then clearly $f\in \underline{I}^k+J$. 

    Now, for each $j=(j_1,\dots,j_n)$ such that $\sum_{i=1}^n j_i<k$, we define a map $\phi_j:\mathcal{C}^\infty(\mathbb{R}^{2n})\rightarrow{\mathcal{C}^\infty(\mathbb{R}^{n})}$ by sending a function $f\in\mathcal{C}^\infty(\mathbb{R}^{2n})$ to the function $\phi_j(f)=d_{x-y}^j f(x,x)$. We recall that $f\in \underline{I}^k+J$ if and only if $\phi_j(f)\in (f_1(x),\dots,f_m(x))$ for all $j$. Hence, we see that $\underline{I}^{k}+J=\bigcap_j \phi_j^{-1}(f_1(x),\dots,f_m(x))$. Since each $\phi_j$ is a continuous map and $(f_1(x),\dots,f_m(x))$ is closed in $\mathcal{C}^\infty(\mathbb{R}^n)$ using Lemma \ref{cinfinitymanifold}, we see that $\underline{I}^k+J$ is an intersection of closed subsets, and hence is also closed in $\mathcal{C}^\infty(\mathbb{R}^{2n})$, as desired.    
  
\end{proof}

\subsection{Representability of the Derived Moduli Stack of Solutions}

We recall from Corollary \ref{tuplediffgeometryrep} that we have the following representability context modelling derived $\mathcal{C}^\infty$-bornological geometry
\begin{equation*}
    (\texthom{\textbf{Ch}}(\textnormal{CBorn}_\mathbb{R}),\bm\tau_{\mathcal{C}^\infty},\textcat{open}_{\mathcal{C}^\infty},\mathcal{C}^\infty\textcat{DBAff},\textcat{Mod},\textcat{hm})
\end{equation*}We will suppose for the rest of this section that $X=\textnormal{Spec}(A)$ lies in $\mathcal{C}^\infty\textcat{DAff}^{\heartsuit}\simeq\mathcal{C}^\infty\textnormal{Aff}$ and that $\mathcal{F}\in\textcat{Stk}(\mathcal{C}^\infty\textcat{DAff},\bm\tau_{\mathcal{C}^\infty}|_{\mathcal{C}^\infty\textcat{DAff}})_{/X}$. Suppose that $Y=\textnormal{Spec}(B)$ is a $(-1)$-representable system of non-linear partial differential equations on sections of $\mathcal{F}$. Consider $\textcat{Sol}_X(Y)$ as a stack in $\textcat{Stk}(\mathcal{C}^\infty\textcat{DAff},\bm\tau_{\mathcal{C}^\infty}|_{\mathcal{C}^\infty\textcat{DAff}})\subseteq \textcat{Stk}(\mathcal{C}^\infty\textcat{DBAff},\bm\tau_{\mathcal{C}^\infty}|_{\mathcal{C}^\infty\textcat{DBAff}})$. We will show that it is geometric as a derived $\mathcal{C}^\infty$-bornological stack.

Recall from Corollary \ref{importantadjunction}, that there is an adjunction
\begin{equation*}
    L:\textnormal{LH}(\textnormal{CBorn}_\mathbb{R})\leftrightarrows{\mathcal{C}^\infty\textnormal{BornRing}}:R
\end{equation*}We note that $\mathbb{R}^n\in\mathcal{C}^\infty\textnormal{BornRing}$ for each $n\in\mathbb{N}$ is in the essential image of $L$, as well as any sifted colimit of objects of this form where morphisms are bounded linear. In particular, we note that any bornologically finitely presented $\mathcal{C}^\infty$-ring is in the essential image of $L$.
\begin{lem}\label{firststeptruncation}
    Suppose that $Y=\textnormal{Spec}(B)$ is a $(-1)$-representable system of partial differential equations on sections of $\mathcal{F}$ and that $B$ is a bornologically finitely presented $\mathcal{C}^\infty$-ring. Let $X'=\textnormal{Spec}(A')\in\mathcal{C}^\infty\textnormal{Aff}$ be such that $\mathcal{C}^\infty(A')$ is reflexive in $\textnormal{CBorn}_\mathbb{R}$. Then, the stack 
    \begin{equation*}
        t_0(\texthom{Map}_{\textcat{Stk}(\mathcal{C}^\infty\textcat{DAff},\bm\tau_{\mathcal{C}^\infty}|_{\mathcal{C}^\infty\textcat{DAff}})_{/X'}}(X',Y))
    \end{equation*}is representable by an object of $\mathcal{C}^\infty\textnormal{BAff}$. 
\end{lem}

\begin{proof}
    Since $B$ is bornologically finitely presented, then it can be written as $L(B')$ where $B'\in\textnormal{LH}(\textnormal{CBorn}_\mathbb{R})$. Suppose that $Z=\textnormal{Spec}(C)\in\mathcal{C}^\infty\textnormal{Aff}$. Then, \begin{align*}
        t_0(\texthom{Map}_{\textcat{Stk}(\mathcal{C}^\infty\textcat{DAff},\bm\tau_{\mathcal{C}^\infty}|_{\mathcal{C}^\infty\textcat{DAff}})_{/X'}}(X',Y))&\simeq \textnormal{Hom}_{{}^{A'/}\mathcal{C}^\infty\textnormal{BornRing}}(L(B'),A'\otimes C)\\
        \intertext{Now, using the adjunction from Corollary \ref{importantadjunction},}
        &\simeq \textnormal{Hom}_{{}^{R(A')/}\textnormal{LH}(\textnormal{CBorn}_\mathbb{R})}(B',R(A'\otimes C))
        &\intertext{Now, by Proposition \ref{Rmonoidalstructure},}
        &\simeq \textnormal{Hom}_{{}^{R(A')/}\textnormal{LH}(\textnormal{CBorn}_\mathbb{R})}(B',R(A')\hat\otimes R(C))\\
        \intertext{Since $\mathcal{C}^\infty(A')$ is reflexive, we can use Lemma \ref{reflexivereduced} to show that this is equivalent to}
        &\simeq \textnormal{Hom}_{\textnormal{LH}(\textnormal{CBorn}_\mathbb{R})}(B'\hat\otimes R(A')^\vee, R(C))\\
    &\simeq\textnormal{Hom}_{\mathcal{C}^\infty\textnormal{BornRing}}(L(B'\hat\otimes R(A')^\vee), C)\\
        &\simeq \textnormal{Hom}_{\mathcal{C}^\infty\textnormal{BAff}}(Z,\textnormal{Spec}(L(B'\hat\otimes R(A')^\vee)))
    \end{align*}

    Therefore, we see that our truncated stack is equivalent, as a stack in $\textcat{Stk}(\mathcal{C}^\infty\textnormal{Aff},\bm\tau_{\mathcal{C}^\infty}^\heartsuit)$, to the stack represented by $\textnormal{Spec}(L(B'\hat\otimes R(A')^\vee))\in\mathcal{C}^\infty\textnormal{BAff}$.  
\end{proof}

We will restrict our attention to the following well behaved PDEs, i.e. ones with perfect cotangent complex. 
\begin{defn}
   Suppose that $X=\textnormal{Spec}(A)$ with $A$ a finitely presented regular $\mathcal{C}^\infty$-ring. A system of derived non-linear partial differential equations $\mathcal{G}$ on sections of $\mathcal{F}\in\textcat{Stk}(\mathcal{C}^\infty\textcat{DAff},\bm\tau_{\mathcal{C}^\infty}|_{\mathcal{C}^\infty\textcat{DAff}})_{/X}$ is \textit{formally perfect} if
   \begin{enumerate}
       \item it is an $n$-geometric stack, 
       \item the global cotangent complex $\mathbb{L}_{\mathcal{G}/X'}$ is an object of $\textcat{Perf}(\mathcal{G})$ for $X'=X$ and $X'=\textnormal{Spec}(A_k)$ for $k\in\mathbb{Z}_{\geq 1}$, where $A_k$ is defined as in Lemma \ref{coequaliserxdr}. 
   \end{enumerate}
\end{defn}

\begin{exmp}
    Suppose that $\mathcal{G}$ is a non-linear elliptic system of PDEs on sections of $X=\textnormal{Spec}(A)$ where $A$ is a finitely presented regular $\mathcal{C}^\infty$-ring. Suppose that $\mathcal{G}$ is representable by a finitely presented regular $\mathcal{C}^\infty$-affine. Suppose also that the morphism $\mathcal{G}\rightarrow{\textcat{Jets}_{X_{dR}}^\infty(X)}$ is finitely presented (in the sense of \cite[Definition 1.2.3.1]{toen_homotopical_2008}) and factors through some $\textnormal{Spec}(A_k)$ for some $k$. Consider the following cofibre sequence (see \cite[Corollary 3.7.8]{savage_representability_2024})
    \begin{equation*}
        \mathbb{L}_{\textcat{Jets}_{X_{dR}}^\infty(X)/X'}\rightarrow{\mathbb{L}_{\mathcal{G}/X'}}\rightarrow{\mathbb{L}_{\mathcal{G}/\textcat{Jets}_{X_{dR}}^\infty(X)}}
    \end{equation*}Since $\mathcal{G}$ is a system of elliptic PDEs, then the cotangent complex $\mathbb{L}_{\mathcal{G}/\textcat{Jets}_{X_{dR}}^\infty(X)}$ is quasi-isomorphic to the dual of a two term complex of finitely generated projective objects, and hence is perfect. Moreover, since $\textcat{Jets}_{X_{dR}}^\infty(X)\simeq X_{dR}$, we see that $\mathbb{L}_{\textcat{Jets}_{X_{dR}}^\infty(X)/X}$ is perfect and hence, by base-change, $\mathbb{L}_{\textcat{Jets}_{X_{dR}}^\infty(X)/X'}$ is perfect. Therefore, we see that $\mathcal{G}$ is a formally perfect system of non-linear PDEs. 
\end{exmp}

\begin{cor}[Representability of the Derived Moduli Stack]\phantomsection\label{repsolutionstack}
    Suppose that
    \begin{enumerate}
        \item $Y=\textnormal{Spec}(B)$ is a $(-1)$-representable system of formally perfect derived non-linear partial differential equations on sections of $\mathcal{F}$, where $B$ is a bornologically finitely presented $\mathcal{C}^\infty$-ring,
        \item $X=\textnormal{Spec}(A)$ is such that $A$ is a finitely presented regular $\mathcal{C}^\infty$-ring. 
    \end{enumerate}Then, $\textcat{Sol}_X(Y)$ is a $(-1)$-geometric stack in $\textcat{Stk}_n(\mathcal{C}^\infty\textcat{DBAff},\tau_{\mathcal{C}^\infty}|_{\mathcal{C}^\infty\textcat{DBAff}},\textcat{open}_{\mathcal{C}^\infty}|_{\mathcal{C}^\infty\textcat{DBAff}})$, i.e. it is representable by an object of $\mathcal{C}^\infty\textcat{DBAff}$.
\end{cor}

\begin{proof}
We recall that 
\begin{equation*}
    \textcat{Sol}_X(Y)=\texthom{Map}_{\textcat{Stk}(\mathcal{C}^\infty\textcat{DAff},\bm\tau_{\mathcal{C}^\infty}|_{\mathcal{C}^\infty\textcat{DAff}})_{/X_{dR}}}(X_{dR},Y)
\end{equation*}Since $A$ is a finitely presented regular $\mathcal{C}^\infty$-ring, then by Lemma \ref{coequaliserxdr}, $X_{dR}$ can be described as the coequaliser of the diagram 
\begin{equation*}
        \varinjlim_{k\in\mathbb{Z}_{\geq 1}}\textnormal{Spec}(A_k)\rightrightarrows{\textnormal{Spec}(A)}
    \end{equation*}in $\textcat{Stk}(\mathcal{C}^\infty\textnormal{Aff},\bm\tau_{\mathcal{C}^\infty}^\heartsuit|_{\mathcal{C}^\infty\textnormal{Aff}})\subseteq \textcat{Stk}(\mathcal{C}^\infty\textcat{DAff},\bm\tau_{\mathcal{C}^\infty}|_{\mathcal{C}^\infty\textcat{DAff}})$. Hence, we see that $\textcat{Sol}_X(Y)$ can be expressed as the limit of mapping stacks of the form
    \begin{equation*}
        \texthom{Map}_{\textcat{Stk}(\mathcal{C}^\infty\textcat{DAff},\bm\tau_{\mathcal{C}^\infty}|_{\mathcal{C}^\infty\textcat{DAff}})_{/X'}}(X',Y)
    \end{equation*}with $X'$ either $\textnormal{Spec}(A)$ or $\textnormal{Spec}(A_k)$ for $k\in\mathbb{Z}_{\geq 1}$. Since $\mathcal{C}^\infty\textcat{DBAff}$ is closed under colimits, it suffices to show that each of these mapping stacks is $(-1)$-representable by an object of $\mathcal{C}^\infty\textcat{DBAff}$, as then the limit will also be $(-1)$-representable. Since derived $\mathcal{C}^\infty$-rings are closed under tensor products, it suffices to check the conditions of Corollary \ref{mappingstacksimplified}.

    We note that $A$ and each $A_k$ satisfy that $\mathcal{C}^\infty(A)$ and $\mathcal{C}^\infty(A_k)$ are reflexive in $\textnormal{CBorn}_\mathbb{R}$ by Lemmas \ref{finprescinfinity} and \ref{bornfinpresAk}. Hence, by Lemma \ref{firststeptruncation}, we see that the truncated mapping stacks are representable by objects $\textnormal{Spec}(L(B'\hat\otimes R(A)^\vee))$ and $\textnormal{Spec}(L(B'\hat\otimes R(A_k)^\vee))$ in $\mathcal{C}^\infty\textnormal{BAff}$ respectively. Hence, Condition (\ref{mapcor1}) is satisfied. 

    Now, since $Y$ is a system of formally perfect derived non-linear PDEs, the cotangent complex of the morphism $Y\rightarrow{X'}$ is in $\textcat{Perf}(Y)$ for $X'=\textnormal{Spec}(A)$ and $X'=\textnormal{Spec}(A_k)$ for $k\in\mathbb{Z}_{\geq 1}$. Therefore, Condition (\ref{mapcor2}) is also satisfied, and hence $\textcat{Sol}_X(Y)$ is $(-1)$-representable by an object of $\mathcal{C}^\infty\textcat{DBAff}$.

\end{proof}

\appendix

\section{Bornological Spaces}\label{bornologyappendix}
Bornological spaces are spaces which possess enough structure to consider questions of boundedness, and thus are an ideal setting for bringing together homological algebra and functional analysis. Our main references are \cite{bambozzi_dagger_2016} and \cite{meyer_local_2007}.

\begin{defn}
Let $X$ be a set. A \textit{bornology} on $X$ is a collection $\mathcal{B}$ of subsets of $X$ such that 
\begin{itemize}
    \item $\mathcal{B}$ covers $X$, i.e.  for every $x\in X$, there exists some $B\in \mathcal{B}$ such that $x\in \mathcal{B}$, 
    \item $\mathcal{B}$ is stable under inclusions, i.e. for every inclusion $A\subset B\in \mathcal{B}$, we have $A\in\mathcal{B}$,
    \item $\mathcal{B}$ is stable under finite unions, i.e. for each $n\in\mathbb{N}$ and $B_1,\dots,B_n\in\mathcal{B}$, we have $\bigcup_{i=1}^nB_i\in\mathcal{B}$.
\end{itemize}
\end{defn}

The pair $(X,\mathcal{B})$ is called a \textit{bornological set}, and the elements of $\mathcal{B}$ are called bounded subsets of $X$. A family of subsets $\mathcal{A}\subset\mathcal{B}$ is called a \textit{basis} for $\mathcal{B}$ if, for any $B\in\mathcal{B}$, there exist $A_1,\dots,A_n\in \mathcal{A}$ such that $B\subset A_1\cup \dots\cup A_n$. A \textit{morphism of bornological sets} is any map which sends bounded subsets to bounded subsets. 

Suppose that we have a complete non-trivially valued field $k$. The case where $k$ is a trivially valued field is addressed in \cite[Section 6]{bambozzi_dagger_2016}.
\begin{defn}
A \textit{bornological vector space} over $k$ is a $k$-vector space $V$ together with a bornology on the underlying set of $V$ such that the maps $(\lambda,v)\rightarrow{\lambda v}$ and $(v,w)\rightarrow{v+w}$ are bounded. 
\end{defn}

\begin{exmp}\phantomsection\label{vonnemannborn}
    Suppose that $V$ is a vector space.
    \begin{enumerate}
        \item The \textit{fine bornology on $V$} is the smallest possible bornology on $V$. A subset $B\subseteq V$ belongs to this bornology if and only if there is a finite-dimensional subspace $W\subseteq V$ such that $B\subseteq W$ and $B$ is bounded in $W$.     \end{enumerate}
Suppose that $V$ is a locally convex topological vector space, 
\begin{enumerate}
    \item The \textit{von Neumann bornology on $V$} is the bornology consisting of von-Neumann bounded subsets, i.e. those which are absorbed by each neighbourhood of the origin in $V$, 
    \item The \textit{precompact bornology on $V$} is the bornology consisting of those subsets $B\subseteq V$ such that their closure in the completion of $V$ is compact. 
\end{enumerate}

\end{exmp}

We will now detail certain categories of bornological vector spaces. We let $k^\circ=\{\lambda\in k\mid |\lambda|\leq 1\}$. Suppose that $V$ is a $k$-vector space. Then, a subset $W$ of $V$ is \textit{convex} if, for every  $v,w\in W$ and $t\in [0,1]$, we have that $(1-t)v+tw\in W$, and \textit{balanced} if, for every $\lambda\in k^\circ$, $\lambda W\subset W$. We will say that $W$  is \textit{absolutely convex (or a disk)} if, for $k$ Archimedean, $W$ is convex and balanced, and if, for $k$ non-Archimedean, $W$ is a $k^\circ$-submodule of $V$.

\begin{defn}
A bornological vector space is said to be of \textit{convex type} if it has a basis made up of absolutely convex subsets. The category of bornological $k$-vector spaces of convex type, equipped with bounded linear maps, will be denoted by $\textnormal{Born}_k$.
\end{defn}

\begin{defn}
A bornological vector space over $k$ is 
\begin{itemize}
    \item \textit{separated} if its only bounded vector subspace is the trivial subspace $\{0\}$, 
    \item \textit{complete} if there exists a small filtered category $I$, a functor $I\rightarrow{\textnormal{Ban}}_k$, and an isomorphism $V\simeq \underset{i\in I}{\varinjlim}\, V_i$ for a filtered colimit of Banach spaces over $k$, for which the system morphisms are all injective and the colimit is calculated in $\textnormal{Born}_k$. 
\end{itemize}
We denote the full subcategories of $\textnormal{Born}_k$ consisting of separated and complete bornological $k$-vector spaces by $\textnormal{SBorn}_k$ and $\textnormal{CBorn}_k$ respectively. 
\end{defn}

We note that we have fully faithful inclusion functors
\begin{equation*}
    \textnormal{CBorn}_k\lhook\joinrel\longrightarrow{\textnormal{SBorn}_k}\lhook\joinrel\longrightarrow{\textnormal{Born}_k}
\end{equation*}which have left adjoints given by the \textit{separation functor} $\textnormal{sep}:\textnormal{Born}_k\rightarrow\textnormal{SBorn}_k$ sending any space $V$ to $V/\overline{\{0\}}$, and the \textnormal{completion functor} $\textnormal{comp}:\textnormal{SBorn}_k\rightarrow{\textnormal{CBorn}_k}$ whose construction is detailed in \cite[Section 3.3]{bambozzi_dagger_2016}.

\begin{exmp}
    \begin{enumerate}
        \item Any vector space endowed with the fine bornology is a complete bornological space, 
        \item The von-Neumann bornology and the precompact bornology on any locally convex topological vector space $V$ is a convex bornology. It is separated if $V$ is Hausdorff and complete if $V$ is complete. 
    \end{enumerate}
\end{exmp}

We note that the category $\textnormal{Born}_k$ is closed symmetric monoidal. Given $V,W\in\textnormal{Born}_k$, we can endow $V\otimes_kW$ with the \textit{projective tensor product bornology} generated by the absolutely convex hulls of subsets of the form
\begin{equation*}
    X\otimes Y=\{x\otimes y\mid x\in X, y\in Y\}
\end{equation*}for bounded disks $X$ in $V$ and $Y$ in $W$. The space $\textnormal{Hom}(V,W)$ can be endowed with a bornology provided by the \textit{equibounded subsets}, i.e. those subsets $L$ consisting of linear maps $f:V\rightarrow{W}$ such that, for each $B$ bounded in $V$, the set $\{f(v)\mid f\in L, v\in B\}$ is bounded in $W$. 

The closed monoidal structure on $\textnormal{CBorn}_k$ is given by the completion $\textnormal{comp}(V\otimes W)$ of the projective tensor product. The internal hom is defined as in $\textnormal{Born}_k$.

\begin{prop}\cite[Lemma 3.53]{bambozzi_dagger_2016}
    $\textnormal{CBorn}_k$ is a bicomplete closed symmetric monoidal quasi-abelian category. 
\end{prop} 

We will often implicitly consider the categories of Banach and Fr\'echet spaces as full subcategories of the category of complete bornological spaces by equipping them with the von Neumann bornology.

\section{Nuclear Complete Bornological Spaces}

\begin{defn}\cite[Definition 3.38]{bambozzi_stein_2018} An object $V\in\textnormal{CBorn}_k$ is \textit{nuclear} if there is an isomorphism 
\begin{equation*}
    V\simeq \varinjlim_i V_i
\end{equation*}where $I\rightarrow{\textnormal{Ban}_k}$ is a diagram of Banach spaces with nuclear monomorphisms $V_i\rightarrow{V_j}$ as transition maps for all $i<j$. 
\end{defn}

We note that, by \cite[Lemma 3.60]{ben-bassat_frechet_2023}, any nuclear Fr\'echet space can be considered as a complete nuclear bornological space by endowing it with the von Neumann bornology. 

\begin{lem}\cite[c.f. Lemma 2.17]{prosmans_topological_2000} Any nuclear object $V\in\textnormal{CBorn}_k$ is isomorphic to a filtered colimit $\varinjlim_j \ell^1$ where $j:J\rightarrow{\textnormal{Ban}_k}$ is a diagram of Banach spaces with nuclear transition maps.
    
\end{lem}

For $V\in\textnormal{CBorn}_k$, we denote by $V^\vee$ the dual $V^\vee:=\texthom{Hom}_{\textnormal{CBorn}_k}(V,k)$. For $V\in\textnormal{CBorn}_\mathbb{R}$, we will say that $V$ is \textit{reflexive} if $(V^\vee)^\vee\simeq V$. We note that any nuclear Fr\'echet space is a reflexive complete bornological space by \cite[Results 6.5]{kriegl_convenient_1997}.

The following result follows using the definition of nuclearity for Banach spaces. 
\begin{lem}\label{nuclearhomequiv}
    If $V$ is a nuclear complete bornological space, then for any $W\in\textnormal{Ban}_k$, there is an equivalence
    \begin{equation*}
        W^\vee\hat{\otimes}V\simeq \texthom{Hom}_{\textnormal{CBorn}_\mathbb{R}}(W,V)
    \end{equation*}
\end{lem}

The proof of the following corollary is due to Jack Kelly. 
\begin{cor}\label{strongdualisablenucref}
    If $V\in\textnormal{CBorn}_\mathbb{R}$ is nuclear and reflexive and $W$ is a Banach space, then we have an equivalence
    \begin{equation*}
        \texthom{Hom}_{\textnormal{CBorn}_\mathbb{R}}(V^\vee,W)\simeq W\hat{\otimes} V
    \end{equation*}
\end{cor}
\begin{proof}
We will first suppose that there exists some Banach space $U$ such that $U^\vee\simeq W$. Using our previous lemma, we see that
\begin{equation*}
    \texthom{Hom}_{\textnormal{CBorn}_\mathbb{R}}(V^\vee,W)\simeq \texthom{Hom}_{\textnormal{CBorn}_\mathbb{R}}(V^\vee,U^\vee)\simeq \texthom{Hom}_{\textnormal{CBorn}_\mathbb{R}}(U,V)\simeq U^\vee\hat{\otimes}V\simeq W\hat{\otimes}V
\end{equation*} We now note that since $V$ is nuclear, it is flat by \cite[Theorem 3.50]{bambozzi_stein_2018}, and hence we see that $-\hat{\otimes}V$ commutes with kernels. Now, since any Banach space $W$ can be written as a kernel of a map of Banach spaces $\ell^1(\kappa)^\vee\simeq\ell^\infty(\kappa)\rightarrow{\ell^\infty(\mu)=\ell^1(\mu)^\vee}$, our result easily follows
\end{proof}

\begin{cor}\label{nuclearhomequivCinfinity}
    Suppose that $V,W$ are Fr\'echet spaces with $V$ nuclear and reflexive. Then, there is an equivalence
    \begin{equation*}
        \texthom{Hom}_{\textnormal{CBorn}_\mathbb{R}}(V^\vee,W)\simeq W\hat{\otimes}V
    \end{equation*}
    
\end{cor}

\begin{proof}
    We note that we may write $W$ as a countable limit of Banach spaces. Since $V$ is a nuclear Fr\'echet space, then by \cite[Lemma 5.18]{ben-bassat_frechet_2023}, we see that $-\hat{\otimes}V$ commutes with countable limits. Hence, we can conclude using the previous Lemma. 
\end{proof}

\section{Perfect, Compact, and Dualisable Objects}\phantomsection\label{perfectappendix}

Let $\mathcal{C}$ be an additive closed symmetric monoidal locally presentable $(\infty,1)$-category with monoidal product $\otimes^\mathbb{L}$ and unit $I$. Denote the internal mapping space by $\texthom{Map}_\mathcal{C}:\mathcal{C}\rightarrow{\mathcal{C}}$. We fix the following definitions.

\begin{defn}\phantomsection\label{differentduals} Suppose that $A\in\mathcal{C}$. Then its \textit{dual} object is $A^\vee:=\texthom{Map}_\mathcal{C}(A,I)$. An object $A\in\mathcal{C}$ is 
    \begin{enumerate}
            \item \textit{compact} if $\textnormal{Map}_\mathcal{C}(A,-)$ commutes with filtered colimits,
        \item \textit{projective} if $\textnormal{Map}_\mathcal{C}(A,-)$ commutes with geometric realisations,
        \item \textit{perfect} if it is a retract of a finite colimit of objects of the form $\coprod_E I$ for some finite set $E$, 
        \item \textit{dualisable} if the map $A^\vee\otimes^{\mathbb{L}}A\rightarrow{\texthom{Map}_\mathcal{C}(A,A)}$ is an equivalence,
        \item \textit{strongly dualisable} if the map $A^\vee\otimes^{\mathbb{L}}B\rightarrow{\texthom{Map}_\mathcal{C}(A,B)}$ is an equivalence for any $B\in\mathcal{C}$,
        \item \textit{reflexive} if $(A^\vee)^\vee\simeq A$.
    \end{enumerate}
\end{defn}

We note that perfect objects are strongly dualisable and projective. If $\mathcal{C}$ is additionally a stable $(\infty,1)$-category, then perfect objects are reflexive. We have the following result. 

\begin{lem}\phantomsection\label{fibreperfect}
    Suppose that $\mathcal{C}$ is a stable $(\infty,1)$-category. If $A\rightarrow{B}\rightarrow{C}$ is a fibre sequence in $\mathcal{C}$ with two of the objects strongly dualisable, then the third object is strongly dualisable.
\end{lem}

\begin{proof}
Suppose that $D\in\mathcal{C}$. The result follows from considering the morphism of fibre-cofibre sequences induced by $(-)^\vee\otimes^\mathbb{L}D$ and $\texthom{Map}_\mathcal{C}(-,D)$ and then applying the five lemma. 
\end{proof}

\section{Ind and Sind Objects}\phantomsection\label{indobjectappendix}

Suppose that $\textnormal{C}$ is a small category with finite colimits. Then, 
\begin{defn}
    The \textit{free filtered cocompletion of $\textnormal{C}$}, denoted $\textnormal{Ind}(\textnormal{C})$, is defined to be the subcategory of presheaves consisting of those which preserve small limits
    \begin{equation*}
        \textnormal{Ind}(\textnormal{C}):=\textnormal{Fun}^{lex}(\textnormal{C}^{op},\textnormal{Set})
    \end{equation*}We will occasionally refer to this category as \textit{the category of Ind-objects in $\textnormal{C}$}.
\end{defn}

\begin{defn}
    An object $X\in\textnormal{Ind}(\textnormal{C})$ is \textit{essentially monomorphic} if it is isomorphic to an object in $\textnormal{Ind}(\textnormal{C})$ of the form $X':I\rightarrow{\textnormal{C}}$ where, for every morphism $i\rightarrow{j}$ in $I$, the corresponding morphism $X'(i)\rightarrow{X'(j)}$ in $\textnormal{C}$ is a monomorphism. 
\end{defn}

In this paper, we will be concerned with the free sifted cocompletion of categories. Suppose that $\textnormal{C}$ is a category with finite coproducts. We recall that a sifted colimit is a colimit which commutes with finite products.
\begin{defn}
    The \textit{free sifted cocompletion of $\textnormal{C}$}, denoted $\textnormal{SInd}(\textnormal{C})$, is defined to be the subcategory of presheaves consisting of those which preserve small products,
    \begin{equation*}
        \textnormal{SInd}(\textnormal{C}):=\textnormal{Fun}^{\times}(\textnormal{C}^{op},\textnormal{Set})
    \end{equation*}
\end{defn}

\begin{lem}\phantomsection\label{fullyfaithfulsind}
    Suppose that we have a fully faithful functor $i:\textnormal{C}\rightarrow{\textnormal{D}}$ for $C,D$ with finite coproducts. Then, there is a fully faithful functor $i:\textnormal{SInd}(\textnormal{C})\rightarrow{\textnormal{SInd}(\textnormal{D})}$. 
\end{lem}

\begin{proof}
    This follows from \cite[Theorem 4.99]{kelly_basic_1982}. 
\end{proof}

\begin{lem}\phantomsection\label{siftedcocompletionproperties} Suppose that we have a functor $F:\textnormal{C}\rightarrow\textnormal{D}$ between two categories $\textnormal{C}$ and $\textnormal{D}$. Suppose that $\textnormal{C}$ has finite coproducts. 
    \begin{enumerate}
        \item If $\textnormal{D}$ has sifted colimits, then $F$ extends to a sifted colimit-preserving functor $\tilde{F}:\textnormal{SInd}(\textnormal{C})\rightarrow{\textnormal{D}}$, 
        \item If $\textnormal{D}$ has all colimits and $\textnormal{F}$ preserves finite coproducts, then $F$ extends to a colimit-preserving functor $\tilde{F}:\textnormal{SInd}(\textnormal{C})\rightarrow{\textnormal{D}}$. 
    \end{enumerate}
\end{lem}

\begin{proof}
    These are standard results and can be found in \cite{adamek_sifted_2001}. 
\end{proof}

By \cite[Corollary 2.7]{adamek_sifted_2001}, we can see that $\textnormal{SInd}(\textnormal{C})$ is equivalent to the category whose objects are functors $X:I\rightarrow{\textnormal{C}}$, where $I$ is a small sifted category and morphisms are defined appropriately. Hence, we can write the objects of $\textnormal{SInd}(\textnormal{C})$ as \textit{formal sifted colimits of objects} in $\textnormal{C}$ as \begin{equation*}
    ``\varinjlim_{i\in I}"X_i:=\varinjlim_{i\in I} h_{X(i)}
\end{equation*}

We note that, in the context of $(\infty,1)$-categories, it is customary to use the notation $\mathcal{P}_\Sigma(\mathcal{C})$ to denote the free sifted cocompletion. By definition, this is the following category of product preserving $(\infty,1)$-functors
\begin{equation*}
    \mathcal{P}_\Sigma(\mathcal{C}):=\textcat{Fun}^\times(\mathcal{C}^{op},\infty\textcat{Grpd})
\end{equation*}

We have the following important results analogous to the above. 
\begin{prop}\phantomsection\label{siftedcocompletion}\cite[Proposition 5.5.8.15]{lurie_higher_2009} Let $\mathcal{C}$ be a small $(\infty,1)$-category with finite coproducts and let $\mathcal{D}$ be an $(\infty,1)$-category with filtered colimits and geometric realisations. Then, there is an equivalence of categories
\begin{equation*}
    \textnormal{Fun}_\Sigma(\mathcal{P}_\Sigma(\mathcal{C}),\mathcal{D})\rightarrow{\textnormal{Fun}(\mathcal{C},\mathcal{D})}
\end{equation*}where $\textnormal{Fun}_\Sigma(\mathcal{P}_\Sigma(\mathcal{C}),\mathcal{D})$ denotes the subcategory of functors preserving filtered colimits and geometric realisations. Moreover, any functor in $\textnormal{Fun}_\Sigma(\mathcal{P}_\Sigma(\mathcal{C}),\mathcal{D})$ preserves sifted colimits. If $\mathcal{D}$ has finite coproducts then any finite coproduct preserving functor $\mathcal{C}\rightarrow{\mathcal{D}}$ can be extended to a small colimit-preserving functor $\mathcal{P}_\Sigma(\mathcal{C})\rightarrow{\mathcal{D}}$.
\end{prop}

\bibliographystyle{plain}
\bibliography{references} 

@misc{pardon_representability_2023,
	title = {Representability in non-linear elliptic {Fredholm} analysis},
	url = {http://arxiv.org/abs/2401.00184},
	urldate = {2024-03-20},
	publisher = {arXiv},
	author = {Pardon, John},
	year = {2023},
	note = {arXiv E-Prints [2401.00184]},
}

@misc{ben-bassat_perspective_2024,
	title = {A perspective on the foundations of derived analytic geometry},
	url = {http://arxiv.org/abs/2405.07936},
	doi = {10.48550/arXiv.2405.07936},
	urldate = {2025-12-10},
	publisher = {arXiv},
	author = {Ben-Bassat, Oren and Kelly, Jack and Kremnizer, Kobi},
	year = {2024},
	note = {arXiv E-Prints [2405.07936]},
}

@book{meyer_local_2007,
	address = {Switzerland},
	series = {{EMS} {Tracts} in {Mathematics}},
	title = {Local and {Analytic} {Cyclic} {Homology}},
	url = {http://undefined/books/etm/39},
	language = {en},
	number = {3},
	urldate = {2022-11-14},
	publisher = {European Mathematical Society},
	author = {Meyer, Ralf},
	year = {2007},
}

@article{prosmans_topological_2000,
	title = {A topological reconstruction theorem for {D}-infinity modules},
	volume = {102},
	language = {en},
	number = {1},
	journal = {Duke Mathematical Journal},
	author = {Prosmans, Fabienne and Schneiders, Jean-Pierre},
	year = {2000},
}

@article{bambozzi_analytic_2019,
	title = {Analytic geometry over {F1} and the {Fargues}-{Fontaine} curve},
	volume = {356},
	issn = {00018708},
	url = {https://linkinghub.elsevier.com/retrieve/pii/S0001870819304323},
	doi = {10.1016/j.aim.2019.106815},
	language = {en},
	urldate = {2025-06-04},
	journal = {Advances in Mathematics},
	author = {Bambozzi, Federico and Ben-Bassat, Oren and Kremnizer, Kobi},
	year = {2019},
	pages = {106815},
}

@book{lee_introduction_2003,
	address = {New York, NY},
	series = {Graduate {Texts} in {Mathematics}},
	title = {Introduction to {Smooth} {Manifolds}},
	volume = {218},
	copyright = {http://www.springer.com/tdm},
	isbn = {978-0-387-95448-6 978-0-387-21752-9},
	url = {http://link.springer.com/10.1007/978-0-387-21752-9},
	doi = {10.1007/978-0-387-21752-9},
	language = {en},
	urldate = {2026-04-08},
	publisher = {Springer New York},
	author = {Lee, John M.},
	year = {2003},
}

@book{schaefer_topological_1999,
	address = {New York, NY},
	series = {Graduate {Texts} in {Mathematics}},
	title = {Topological {Vector} {Spaces}},
	volume = {3},
	copyright = {http://www.springer.com/tdm},
	isbn = {978-1-4612-7155-0 978-1-4612-1468-7},
	url = {http://link.springer.com/10.1007/978-1-4612-1468-7},
	doi = {10.1007/978-1-4612-1468-7},
	language = {en},
	urldate = {2026-04-03},
	publisher = {Springer New York},
	author = {Schaefer, H. H. and Wolff, M. P.},
	year = {1999},
}

@book{lurie_higher_2009,
	address = {Princeton, N.J},
	series = {Annals of mathematics studies},
	title = {Higher topos theory},
	isbn = {978-0-691-14048-3},
	language = {en},
	number = {n° 170},
	publisher = {Princeton university press},
	author = {Lurie, Jacob},
	year = {2009},
}

@phdthesis{savage_stacks_2025,
	title = {Stacks in derived bornological geometry},
	url = {https://ora.ox.ac.uk/objects/uuid:92d32e4e-5dcf-4775-87fc-92d524414b8a},
	urldate = {2025-10-31},
	school = {Oxford University Research Archive},
	author = {Savage, Rhiannon},
	year = {2025},
}

@book{treves_topological_1967,
	address = {New York},
	series = {Pure and {Applied} {Mathematics}},
	title = {Topological vector spaces, distributions and kernels},
	volume = {25},
	isbn = {978-0-12-699450-6},
	language = {en},
	publisher = {Academic Press},
	author = {Treves, François},
	year = {1967},
}

@book{toen_homotopical_2008,
	address = {Providence, Rhode Island},
	series = {Memoirs of the {American} {Mathematical} {Society}},
	title = {Homotopical {Algebraic} {Geometry}. {II}. {Geometric} {Stacks} and {Applications}},
	volume = {193},
	url = {https://hal.archives-ouvertes.fr/hal-00772955},
	number = {902},
	urldate = {2022-06-04},
	publisher = {American Mathematical Society},
	author = {Toën, Bertrand and Vezzosi, Gabriele},
	year = {2008},
}

@misc{steffens_representability_2024,
	title = {Representability of elliptic moduli problems in derived {C}-infinity geometry},
	url = {http://arxiv.org/abs/2404.07931},
	doi = {10.48550/arXiv.2404.07931},
	urldate = {2024-05-27},
	publisher = {arXiv},
	author = {Steffens, Pelle},
	year = {2024},
	note = {arXiv E-Prints [2404.07931]},
}

@misc{steffens_derived_2023,
	title = {Derived {C}-infinity geometry {I}: foundations},
	shorttitle = {Derived \${C}{\textasciicircum}\{{\textbackslash}infty\}\$-{Geometry} {I}},
	url = {http://arxiv.org/abs/2304.08671},
	doi = {10.48550/arXiv.2304.08671},
	urldate = {2024-05-27},
	publisher = {arXiv},
	author = {Steffens, Pelle},
	year = {2023},
	note = {arXiv E-Prints [2304.08671]},
}

@article{spivak_derived_2010,
	title = {Derived smooth manifolds},
	volume = {153},
	issn = {0012-7094},
	url = {http://arxiv.org/abs/0810.5174},
	doi = {10.1215/00127094-2010-021},
	number = {1},
	urldate = {2024-03-20},
	journal = {Duke Mathematical Journal},
	author = {Spivak, David I.},
	year = {2010},
}

@misc{soor_six-functor_2024,
	title = {A six-functor formalism for quasi-coherent sheaves and crystals on rigid-analytic varieties},
	url = {http://arxiv.org/abs/2409.07592},
	doi = {10.48550/arXiv.2409.07592},
	urldate = {2025-06-02},
	publisher = {arXiv},
	author = {Soor, Arun},
	year = {2024},
	note = {arXiv E-Prints [2409.07592]},
}

@article{savage_representability_2024,
	title = {A representability theorem for stacks in derived geometry contexts},
	language = {en},
	urldate = {2024-05-15},
	author = {Savage, Rhiannon},
	year = {2024},
	note = {arXiv E-Prints [2405.08361]},
}

@article{savage_koszul_2023,
	title = {Koszul monoids in quasi-abelian categories},
	volume = {31},
	issn = {1572-9095},
	url = {https://doi.org/10.1007/s10485-023-09756-7},
	doi = {10.1007/s10485-023-09756-7},
	language = {en},
	number = {6},
	urldate = {2024-06-27},
	journal = {Applied Categorical Structures},
	author = {Savage, Rhiannon},
	year = {2023},
	pages = {50},
}

@incollection{lurie_survey_2009,
	address = {Berlin, Heidelberg},
	series = {Abel {Symposia}},
	title = {A survey of elliptic cohomology},
	isbn = {978-3-642-01200-6},
	url = {https://doi.org/10.1007/978-3-642-01200-6_9},
	doi = {10.1007/978-3-642-01200-6_9},
	language = {en},
	urldate = {2023-12-18},
	booktitle = {Algebraic {Topology}: {The} {Abel} {Symposium} 2007},
	publisher = {Springer},
	author = {Lurie, Jacob},
	year = {2009},
	pages = {219--277},
}

@article{lurie_spectral_2018,
	title = {Spectral algebraic geometry},
	url = {https://www.math.ias.edu/~lurie/papers/SAG-rootfile.pdf#page=1214},
	urldate = {2022-08-17},
	author = {Lurie, Jacob},
	year = {2018},
	note = {Available at: https://www.math.ias.edu/{\textasciitilde}lurie/papers/SAG-rootfile.pdf},
}

@inproceedings{kuranishi_new_1965,
	address = {Berlin, Heidelberg},
	title = {New proof for the existence of locally complete families of complex structures},
	isbn = {978-3-642-48016-4},
	doi = {10.1007/978-3-642-48016-4_13},
	language = {en},
	booktitle = {Proceedings of the {Conference} on {Complex} {Analysis}},
	publisher = {Springer},
	author = {Kuranishi, M.},
	year = {1965},
	pages = {142--154},
}

@misc{kryczka_derived_2024,
	title = {Derived moduli spaces of nonlinear {PDEs} {I}: singular propagations},
	shorttitle = {Derived {Moduli} {Spaces} of {Nonlinear} {PDEs} {I}},
	url = {http://arxiv.org/abs/2312.05226},
	language = {en},
	urldate = {2024-07-24},
	publisher = {arXiv},
	author = {Kryczka, Jacob and Sheshmani, Artan and Yau, Shing-Tung},
	year = {2024},
	note = {arXiv E-Prints [2312.05226]},
}

@book{kriegl_convenient_1997,
	address = {Providence, Rhode Island},
	series = {Mathematical {Surveys} and {Monographs}},
	title = {The {Convenient} {Setting} of {Global} {Analysis}},
	volume = {53},
	isbn = {978-0-8218-0780-4 978-0-8218-3396-4},
	url = {http://www.ams.org/surv/053},
	doi = {10.1090/surv/053},
	language = {en},
	urldate = {2025-02-25},
	publisher = {American Mathematical Society},
	author = {Kriegl, Andreas and Michor, Peter},
	year = {1997},
}

@book{coutinho_primer_1995,
	address = {Cambridge},
	series = {London {Mathematical} {Society} {Student} {Texts}},
	title = {A {Primer} of {Algebraic} {D}-{Modules}},
	isbn = {978-0-521-55119-9},
	url = {https://www.cambridge.org/core/books/primer-of-algebraic-dmodules/87B8F8AB3B53DBA8A8BD33A058E54473},
	doi = {10.1017/CBO9780511623653},
	urldate = {2024-06-25},
	publisher = {Cambridge University Press},
	author = {Coutinho, Severino C.},
	year = {1995},
}

@article{edgar_measurability_1979,
	title = {Measurability in a {Banach} space {II}},
	volume = {28},
	issn = {0022-2518},
	url = {https://www.jstor.org/stable/24892249},
	number = {4},
	urldate = {2025-05-19},
	journal = {Indiana University Mathematics Journal},
	author = {Edgar, G. A.},
	year = {1979},
	pages = {559--579},
}

@misc{carchedi_universal_2019,
	title = {On the universal property of derived manifolds},
	url = {http://arxiv.org/abs/1905.06195},
	language = {en},
	urldate = {2024-05-30},
	publisher = {arXiv},
	author = {Carchedi, David and Steffens, Pelle},
	year = {2019},
	note = {arXiv E-Prints [1905.06195]},
}

@article{krasilshchik_geometry_2011,
	title = {Geometry of jet spaces and integrable systems},
	volume = {61},
	issn = {03930440},
	url = {http://arxiv.org/abs/1002.0077},
	doi = {10.1016/j.geomphys.2010.10.012},
	language = {en},
	number = {9},
	urldate = {2024-06-20},
	journal = {Journal of Geometry and Physics},
	author = {Krasil'shchik, Joseph and Verbovetsky, Alexander},
	year = {2011},
	pages = {1633--1674},
}

@book{kelly_basic_1982,
	address = {Cambridge},
	series = {Lecture {Notes} in {Mathematics}},
	title = {Basic {Concepts} of {Enriched} {Category} {Theory}},
	volume = {64},
	language = {en},
	publisher = {Cambridge University Press},
	author = {Kelly, Gregory  Maxwell},
	year = {1982},
}

@article{henrard_left_2023,
	title = {The left heart and exact hull of an additive regular category},
	volume = {39},
	issn = {0213-2230},
	url = {https://ems.press/doi/10.4171/RMI/1388},
	doi = {10.4171/RMI/1388},
	language = {en},
	number = {2},
	urldate = {2023-08-16},
	journal = {Revista Matemática Iberoamericana},
	author = {Henrard, Ruben and Kvamme, Sondre and Van Roosmalen, Adam-Christiaan and Wegner, Sven-Ake},
	year = {2023},
	pages = {439--494},
}

@article{fukaya_arnold_1999,
	title = {Arnold conjecture and {Gromov}-{Witten} invariant},
	volume = {38},
	issn = {0040-9383},
	url = {https://www.sciencedirect.com/science/article/pii/S0040938398000421},
	doi = {10.1016/S0040-9383(98)00042-1},
	number = {5},
	urldate = {2025-06-04},
	journal = {Topology},
	author = {Fukaya, Kenji and Ono, Kaoru},
	year = {1999},
	pages = {933--1048},
}

@misc{borisov_quasi-coherent_2017,
	title = {Quasi-coherent sheaves in differential geometry},
	url = {http://arxiv.org/abs/1707.01145},
	language = {en},
	urldate = {2024-05-30},
	publisher = {arXiv},
	author = {Borisov, Dennis and Kremnizer, Kobi},
	year = {2017},
	note = {arXiv E-Prints [1707.01145]},
}

@misc{ben-bassat_blow-ups_2023,
	title = {Blow-ups and normal bundles in connective and nonconnective derived geometries},
	url = {http://arxiv.org/abs/2303.11990},
	urldate = {2024-01-24},
	publisher = {arXiv},
	author = {Ben-Bassat, Oren and Hekking, Jeroen},
	year = {2023},
	note = {arXiv E-Prints [2303.11990]},
}

@article{adam_countably_1999,
	title = {Countably evaluating homomorphisms on real function algebras},
	volume = {35},
	language = {en},
	number = {2},
	journal = {Archivum Mathematicum},
	author = {Adam, Eva and Biström, Peter and Kriegl, Andreas},
	year = {1999},
	pages = {165--192},
}

@article{adamek_sifted_2001,
	title = {On sifted colimits and generalized varieties},
	volume = {8},
	language = {en},
	number = {3},
	journal = {Theory and Applications of Categories},
	author = {Adámek, Jiří and Rosický, Jiří},
	year = {2001},
	pages = {33--53},
}

@article{blute_convenient_2012,
	title = {A convenient differential category},
	volume = {LIII-3},
	language = {en},
	journal = {Cahiers de topologie et geometrie differentielle categoriques},
	author = {Blute, Richard and Ehrhard, Thomas and Tasson, Christine},
	year = {2012},
}

@article{borisov_beyond_2018,
	title = {Beyond perturbation 1: {De} {Rham} spaces},
	volume = {124},
	issn = {0393-0440},
	shorttitle = {Beyond perturbation 1},
	url = {https://www.sciencedirect.com/science/article/pii/S0393044017302723},
	doi = {10.1016/j.geomphys.2017.11.001},
	urldate = {2024-06-26},
	journal = {Journal of Geometry and Physics},
	author = {Borisov, Dennis and Kremnizer, Kobi},
	year = {2018},
	pages = {208--224},
}

@misc{raksit_hochschild_2020,
	title = {Hochschild homology and the derived de {Rham} complex revisited},
	url = {http://arxiv.org/abs/2007.02576},
	urldate = {2022-07-08},
	publisher = {arXiv},
	author = {Raksit, Arpon},
	year = {2020},
	note = {arXiv E-Prints [2007.02576]},
}

@article{kelly_analytic_2022,
	title = {An analytic {Hochschild}-{Kostant}-{Rosenberg} theorem},
	volume = {410},
	issn = {0001-8708},
	url = {https://www.sciencedirect.com/science/article/pii/S0001870822005114},
	doi = {10.1016/j.aim.2022.108694},
	urldate = {2024-03-29},
	journal = {Advances in Mathematics},
	author = {Kelly, Jack and Kremnizer, Kobi and Mukherjee, Devarshi},
	year = {2022},
}

@article{ben-bassat_frechet_2023,
	title = {Fréchet modules and descent},
	volume = {39},
	language = {en},
	number = {9},
	journal = {Theory and Applications of Categories},
	author = {Ben-Bassat, Oren and Kremnizer, Kobi},
	year = {2023},
	pages = {207--266},
}

@article{bambozzi_sheafyness_2024,
	title = {On the sheafyness property of spectra of {Banach} rings},
	volume = {109},
	number = {1},
	urldate = {2024-04-14},
	journal = {Journal of the London Mathematical Society},
	author = {Bambozzi, Federico and Kremnizer, Kobi},
	year = {2024},
}

@article{bambozzi_stein_2018,
	title = {Stein domains in {Banach} algebraic geometry},
	volume = {274},
	issn = {00221236},
	url = {http://arxiv.org/abs/1511.09045},
	doi = {10.1016/j.jfa.2018.01.003},
	number = {7},
	urldate = {2022-08-16},
	journal = {Journal of Functional Analysis},
	author = {Bambozzi, Federico and Ben-Bassat, Oren and Kremnizer, Kobi},
	year = {2018},
	pages = {1865--1927},
}

@article{bambozzi_dagger_2016,
	title = {Dagger geometry as {Banach} algebraic geometry},
	volume = {162},
	issn = {0022-314X},
	url = {https://www.sciencedirect.com/science/article/pii/S0022314X15003674},
	doi = {10.1016/j.jnt.2015.10.023},
	language = {en},
	urldate = {2022-08-08},
	journal = {Journal of Number Theory},
	author = {Bambozzi, Federico and Ben-Bassat, Oren},
	year = {2016},
	pages = {391--462},
}

@article{schneiders_quasi-abelian_1999,
	title = {Quasi-abelian categories and sheaves},
	volume = {1},
	issn = {0249-633X, 2275-3230},
	url = {http://www.numdam.org/item?id=MSMF_1999_2_76__R3_0},
	doi = {10.24033/msmf.389},
	language = {fr},
	urldate = {2022-08-08},
	journal = {Memoires de la Société Mathématique de France},
	author = {Schneiders, Jean-Pierre},
	year = {1999},
	pages = {1--140},
}

@article{ben-bassat_non-archimedean_2017,
	title = {Non-{Archimedean} analytic geometry as relative algebraic geometry},
	volume = {26},
	issn = {2258-7519},
	url = {https://afst.centre-mersenne.org/articles/10.5802/afst.1526/},
	doi = {10.5802/afst.1526},
	language = {en},
	number = {1},
	urldate = {2022-08-17},
	journal = {Annales de la Faculté des sciences de Toulouse : Mathématiques},
	author = {Ben-Bassat, Oren and Kremnizer, Kobi},
	year = {2017},
	pages = {49--126},
}

\end{document}